\documentclass[12pt,reqno]{amsproc}
\usepackage{}
\usepackage{mathrsfs, amsfonts, amsthm, amsmath, amssymb}

 \theoremstyle{plain}

\theoremstyle{plain}
\newtheorem{theorem}{Theorem}[subsection]

\newtheorem{corollary}[theorem]{Corollary}

\newtheorem{definition}[theorem]{Definition}
\newtheorem{example}[theorem]{Example}

\newtheorem{lemma}[theorem]{Lemma}

\newtheorem{proposition}[theorem]{Proposition}
\newtheorem{remark}[theorem]{Remark}

\makeatletter
\newcommand{\LeftEqNo}{\let\veqno\@@leqno}

\makeatother

 \numberwithin{equation}  {section}
\begin{document}

\

\vspace{-2cm}

\title[A generalization of the Voiculescu theorem]{A generalization of the Voiculescu theorem for normal operators in semifinite von Neumann algebras}
\author{Qihui Li}
\curraddr{School of Science, East China University of Science and Technology, Shanghai, 200237, P. R. China}
\email{qihui\_{}li@126.com}
\thanks{The first author was partly supported by NSFC(Grant No.11671133).}
\author{Junhao Shen}
\curraddr{Department of Mathematics \& Statistics, University of
New Hampshire, Durham, 03824, US}
\email{Junhao.Shen@unh.edu}
\thanks{}
\author{Rui Shi}
\curraddr{School of Mathematical Sciences, Dalian University of
Technology, Dalian, 116024, P. R. China}
\email{ruishi@dlut.edu.cn, ruishi.math@gmail.com}
\thanks{The corresponding author Rui Shi was partly supported by NSFC(Grant No.11401071) and the Fundamental Research Funds for the Central Universities (Grant No.DUT16RC(4)57).}




\begin{abstract}
In this paper, we provide a generalized  version of the Voiculescu
theorem for normal operators   by showing that, in a von Neumann
algebra  with separable pre-dual and a faithful normal semifinite
tracial weight $\tau$, a normal operator is
  an arbitrarily small $(\max\{\|\cdot\|,
\Vert\cdot\Vert_{2}\})$-norm perturbation of a diagonal operator.
Furthermore, in a countably decomposable, properly infinite von
Neumann algebra with a faithful  normal  semifinite tracial weight,
we prove that each self-adjoint operator can be diagonalized modulo
norm ideals satisfying a natural condition.
\end{abstract}

\maketitle

\section{Introduction}

Let $\mathcal{H}$ be a complex separable Hilbert space. Denote by $\mathcal{B%
}(\mathcal{H})$ the set of bounded linear operators on
$\mathcal{H}$ and by
$\mathcal{K}(\mathcal{H})$ the set of compact operators in $\mathcal{B}(%
\mathcal{H})$. 
In 1909, Weyl \cite{Weyl} proved that a self-adjoint operator  in $%
\mathcal{B}(\mathcal{H})$ is a   compact perturbation of a diagonal
operator.
 Later, in 1935, von Neumann \cite{Von2} improved the result by
replacing a  ``compact operator'' with an ``arbitrarily small
Hilbert-Schmidt operator''. That is, for a  self-adjoint
operator $a$ in $\mathcal{B}(%
\mathcal{H})$ and  $\epsilon>0$, there exist a diagonal self-adjoint
operator $d$ in $\mathcal{B}(\mathcal{H})$ and a compact operator
$k$ in $\mathcal{K}(\mathcal{H})$ such that
\begin{equation*}
a=d+k  \mbox{ and } \Vert k \Vert_{2}\le\epsilon,
\end{equation*}
where $\|\cdot\|_2$ is the Hilbert-Schmidt norm of
$\mathcal{B}(\mathcal{H})$. Recall that an operator $d$ in $\mathcal{B}(\mathcal{H})$ is called \emph{%
diagonal} if there exist a family $\{e_n\}_{n=1}^\infty$ of
orthogonal
projections in $\mathcal{B}(\mathcal{H})$ and a family $\{\lambda_n\}_{n=1}^%
\infty$ of complex numbers such that $d=\sum_{n=1}^\infty
\lambda_ne_n$.

In 1957, Kato and Rosenblum \cite{Kato,Rosenblum} showed that, up to unitary
equivalence, the \emph{absolutely continuous part} of a self-adjoint
operator in $\mathcal{B}(\mathcal{H})$ \emph{can not} be changed by
trace-class perturbations.   Thus, if a self-adjoint operator $a$   in $\mathcal{B}(\mathcal{H})$  is not
purely singular, then $a$ {can not} be diagonalized modulo the trace class.
On the other hand, it was shown by Carey and Pincus   \cite {CP} that a purely singular  self-adjoint operator   in $\mathcal{B}(\mathcal{H})$  is a small   trace class perturbation of a   diagonal operator.

In 1958, Kuroda \cite{Kuroda} generalized the Weyl-von Neumann
theorem for every single self-adjoint operator in
$\mathcal{B}(\mathcal{H})$ with respect to a unitarily invariant
norm that is not equivalent to the trace norm. More specifically,
given $\epsilon>0$ and $\Phi(\cdot)$ a unitarily invariant norm
\emph{not equivalent to the trace norm}, for every self-adjoint
operator $a$ in $\mathcal{B}(\mathcal{H})$, there exist a diagonal
self-adjoint operator $d$ in $\mathcal{B}(\mathcal{H})$ and a
compact operator $k$ in $\mathcal{K}(\mathcal{H})$ such that $a=d+k
\mbox{
and } \Phi(k)\le \epsilon. 
$

Answering a question proposed by Halmos \cite{Halmos}, Berg
\cite{Berg} and Sikonia \cite{Siknia} independently provided an
extension of the Weyl-von Neumann theorem for a single normal
operator $a$ in $\mathcal{B}(\mathcal{H}) $ relative to the operator
norm. In detail, for any given $\epsilon>0$, there exist a diagonal
operator $d$ in $\mathcal{B}(\mathcal{H})$ and a compact operator
$k$ in $\mathcal{K}(\mathcal{H})$ such that $a=d+k \mbox{ and }
\Vert k\Vert\le \epsilon.$ In the same paper, Berg proposed an
interesting question:{\emph{Whether the $k$ can be in the
Hilbert-Schmidt class.}}

Note that Berg's result can also be viewed as an extension of the Weyl-von
Neumann theorem for two commuting self-adjoint operators in $\mathcal{B}(%
\mathcal{H})$ relative to the operator norm. The next significant improvement of the
Weyl-von Neumann theorem came in 1979 from Voiculescu \cite{Voi} by proving
a version of the Weyl-von Neumann theorem for $r$ commuting self-adjoint
operators in $\mathcal{B}(\mathcal{H})$ relative to norm ideals. Among many
remarkable results, Voiculescu was able to show that, for $r\geq 2$, given $r$
commuting self-adjoint operators $a_{1},\ldots,a_{r}$ in $\mathcal{B}(%
\mathcal{H})$ and $\epsilon>0$, there exist $r$  commuting self-adjoint
diagonal operators $d_{1},\ldots,d_{r}$ in $\mathcal{B}(\mathcal{H})$ such
that
\begin{equation*}
\Vert a_{i}-d_{i}\Vert_{r}\le\epsilon, \mbox{ for } i=1,\ldots, r,
\end{equation*}
where $\|\cdot\|_r$ is the Schatten $r$-norm of $\mathcal{B}(\mathcal{H})$.
In particular, \emph{a normal operator in $\mathcal{B}(\mathcal{H})$  is a
sum of a diagonal operator and a small Hilbert-Schmidt operator.} These
results of Voiculescu were based on his striking non-commutative Weyl-von
Neumann theorem in \cite{Voi2}, which gave an affirmative answer to the 8th
problem proposed by Halmos \cite{Halmos} in 1970. (More   developments on diagonalizations of self-adjoint operators in
$\mathcal B(\mathcal H)$ can be found in \cite{BV}, \cite{Voi3}-\cite{Voi9}, \cite{Xia1}-\cite{Xia5} and etc.)

A von Neumann algebra is a $\ast $-subalgebra of
$\mathcal{B}(\mathcal{H}_0)$ that is closed in the strong operator
topology for   a complex Hilbert space  $\mathcal{H}_0$. In their
fundamental papers \cite{MurrayVonNeuman1},
\cite{MurrayVonNeuman2}, \cite{MurrayVonNeuman3}, and \cite{MurrayVonNeuman4}%
, Murray and von Neumann discussed some basic properties for von Neumann
algebras and separated von Neumann algebras into type I, type II and type
III. A von Neumann algebra is called \textquotedblleft \emph{semifinite}%
\textquotedblright\ if it has no direct summand of type III, or equivalently
it has a faithful normal semifinite tracial weight (see Remark 8.5.9 in \cite%
{Kadison} for details). A von Neumann algebra $\mathcal{M}$ is \emph{%
countably decomposable} if each orthogonal family of nonzero projections in $%
\mathcal{M}$ is countable. A factor is a von Neumann algebra with a trivial
center. Obviously, for a separable Hilbert space $\mathcal{H}$, $\mathcal{B}(%
\mathcal{H})$ is a countably decomposable semifinite factor.

It is natural to consider extensions of the Weyl-von Neumann theorem in the
setting of von Neumann algebras. Several results were obtained in the case
of semifinite factors. Let $\mathcal{N}$ be a countably decomposable
semifinite factor and $\mathcal{K}(\mathcal{N})$ a two-sided closed ideal
generated by all finite projections in $\mathcal{N}$. A  result by Zsido
\cite{Zaido} in 1975 showed that, for a self-adjoint operator $a$ in $%
\mathcal{N}$, there exists a diagonal operator $d$ in $\mathcal{N}$ such
that $a-d\in \mathcal{K}(\mathcal{N})$, where  an operator $d$ in $\mathcal{N}$ is called \emph{%
diagonal} if there exist a family $\{e_n\}_{n=1}^\infty$ of
orthogonal
projections in $\mathcal{N}$ and a family $\{\lambda_n\}_{n=1}^%
\infty$ of complex numbers such that $d=\sum_{n=1}^\infty
\lambda_ne_n$. This result was further extended by
Akemann and Pedersen in \cite{Ake}, when they verified that the operator
norm of $a-d$ can be arbitrarily small. Later, in 1978, Kaftal \cite{Kaftal}
proved that, for a self-adjoint operator $a$ in $\mathcal{N}$ and $\epsilon>0
$, there exists a diagonal operator $d$ in $\mathcal{N}$ such that $\Vert
a-d\Vert<\epsilon \mbox{
and } \Vert a-d\Vert_{2}<\epsilon. $ Here $\Vert
a-d\Vert_{2}\triangleq(\tau(|a-d|^2))^{1/2}$ and $\tau$ is a faithful normal
semifinite tracial weight  of $\mathcal{N}$.

The main result of the paper is the following generalization of the
Voiculescu Theorem for normal operators in $\mathcal{B}(\mathcal{H})$ to the
setting of von Neumann algebras.

\vspace{0.2cm}

{T{\Small HEOREM}} \ref{thm6.1} \emph{Let $\mathcal{N}$ be a
countably decomposable, properly infinite, semifinite factor with a
faithful  normal semifinite tracial weight $\tau$. Let $r\geq2$ be a
positive integer. Assume
$\{ a_{i} \}_{i=1}^r $ is a family of commuting self-adjoint operators in $%
\mathcal{N}$. Then, for any $\epsilon>0,$ there is a family $%
\{d_{i}\}_{i=1}^r $of commuting diagonal operators in $\mathcal{N}$ such
that
\begin{equation*}
\max_{1\le i\le r} \{\|a_i-d_i\|, \|a_i-d_i\|_r\} \le \epsilon,
\end{equation*}
where $\Vert  a_i-d_i\Vert_{r}=(\tau(|a_i-d_i|^r))^{1/r}$ for each $1\le
i\le r$.}

\vspace{0.2cm} {T{\Small HEOREM}} \ref{voiNormal2}. \ \emph{\ Let $\mathcal{M%
}$ be a semifinite von Neumann algebra with separable pre-dual and
let $\tau$ be a faithful
normal semifinite tracial weight of $\mathcal M$. Assume $a$ is a normal operator in $%
\mathcal{M}.$ Given an $\epsilon>0,$ there is a diagonal operator $d$ in $%
\mathcal{M}$ such that
\begin{equation*}
\max \{ \|a-d\|, \|a-d\|_2 \}\le \epsilon,
\end{equation*}
where $\Vert  a-d\Vert_{2}=(\tau(|a-d|^2))^{1/2}$.}

\

A key ingredient in the proofs of the preceding results is our Theorem \ref%
{VoiThm_2}, which partially extends Voiculescu's celebrated non-commutative
Weyl-von Neumann theorem from $\mathcal{B}(\mathcal{H})$ to a semifinite
factor.

Besides Theorem \ref{thm6.1} and Theorem \ref{voiNormal2}, we also
considered perturbations of a self-adjoint operator in a von Neumann
algebra and provided a generalization of Kuroda Theorem (see
\cite{Kuroda}) as follows.

\vspace{0.2cm}

{T{\Small HEOREM}} \ref{thm2.3.3}. \ \emph{Let $\mathcal{M}$ be a
countably decomposable, properly infinite von Neumann algebra with a
faithful  normal  semifinite tracial weight $\tau$ and let
$\mathcal{K}_{\Phi}(\mathcal{M},\tau)$ be a norm ideal of
$(\mathcal{M},\tau)$ (see Definition \ref{prelim_def1} for its
definition). Assume that
\begin{equation*}
\displaystyle
\lim_{\overset{\tau(e)\rightarrow\infty}{e\in \mathcal{P}\mathcal{F}(%
\mathcal{M},\tau)}}\frac{\Phi(e)}{\tau(e)} =0.
\end{equation*}
Let $a\in \mathcal{M} $ be a self-adjoint element. Then for every $\epsilon>0
$, there exists a diagonal operator $d$ in $\mathcal{M}$ such that: }

\begin{enumerate}
\item[(i)] \emph{$a-d\in{\mathcal{K}^{0}_{\Phi}(\mathcal{M},\tau)^{}}$; }

\item[(ii)] \emph{$\Phi(a-d)\le\epsilon$. }
\end{enumerate}

The paper is organized as follows. In Section $2$, we prepare
related notation, definitions, and lemmas. In Section $3$, we extend
Kuroda Theorem for self-adjoint operators in a countably
decomposable, properly infinite von Neumann algebra with a faithful
normal semifinite tracial weight. In Section $4$, a definition of
$\Phi$-well-behaved sets is introduced and some of its properties
are discussed. Examples of $\Phi$-well-behaved sets are also given.
In Section $5$, an extended Voiculescu's non-commutative Weyl-von
Neumann Theorem in semifinite factors is proved for separable nuclear $C^{*}$%
-algebras relative to norm ideals. In Section $6$, we prove that in
a countably decomposable, properly infinite, semifinite factor, for
$r\ge 2$, $r$ mutually commuting self-adjoint operators can be
diagonalized simultaneously up to arbitrarily small perturbations
with respect to a $(\max\{\|\cdot\|, \Vert\cdot\Vert_{r}\})$-norm.
As a corollary, a normal operator in a countably decomposable,
properly infinite, semifinite factor can be diagonalized up to an
arbitrarily small $(\max\{\|\cdot\|, \Vert\cdot\Vert_{2}\})$-norm
perturbation. By applying von Neumann's reduction theory
\cite{MurrayVonNeuman5}, a normal operator in a von Neumann algebra
with separable predual and a faithful normal semifinite tracial
weight $\tau$ can be diagonalized up to an arbitrarily small $%
(\max\{\|\cdot\|, \Vert\cdot\Vert_{2}\})$-norm perturbation.

\section{Preliminaries and Notation}

\label{sec2} In the paper we let $\mathcal{M}$ be a countably
decomposable, properly infinite von Neumann algebra with a faithful
normal semifinite tracial weight $\tau $. Let
\begin{equation*}
\begin{aligned}
  \mathcal P\mathcal F(\mathcal M,\tau) &=\{ e  \ : \ e=e^*=e^2\in \mathcal M \text { and } \tau(e)<\infty\}\\
  \mathcal F(\mathcal M,\tau) &= \{xey \ : \ e\in  \mathcal P\mathcal F(\mathcal M,\tau)  \text { and } x,y\in\mathcal M\}\\
    \mathcal K(\mathcal M,\tau)&= \text{$\|\cdot\|$-norm closure of }  \mathcal F(\mathcal M,\tau) \text{ in } \mathcal M
 \end{aligned}
\end{equation*}%
be the sets of finite rank projections, finite rank operators, and compact
operators respectively, in $(\mathcal{M},\tau )$.

\begin{remark}
\label{2.0.1} For a von Neumann algebra $\mathcal{M}$, the $\| \cdot\|$-norm
closed ideal generated by finite projections in $\mathcal{M}$ is denoted $%
\mathcal{K}(\mathcal{M})$. Generally, $\mathcal{K}(\mathcal{M}%
,\tau)\subseteq \mathcal{K}(\mathcal{M})$, because a finite projection might
not be a finite rank projection with respect to $\tau$. However, if $%
\mathcal{M}$ is a countably decomposable semifinite factor, then $\mathcal{K}%
(\mathcal{M},\tau) = \mathcal{K}(\mathcal{M})$ for a faithful, normal,
semifinite tracial weight $\tau$.
\end{remark}
For each $x$ in $\mathcal{M}$, we let $R(x)$ denote the range
projection of $x$ in $\mathcal{M}$.
 We will repeatedly use the following facts in the paper.
\begin{enumerate}
\item [(i)]  Let $(\mathcal{M})_1^+$ be the  unit ball of positive
operators in $\mathcal{M}$.
Suppose $x,y$ are in $(\mathcal{M})_1^+$. Then $xy=x=yx$ if and only if $%
x\le R(x)\le y$.
\item [(ii)]    Suppose $e$ is a nonzero projection in $\mathcal M$. Then, by Proposition 8.5.2 in \cite{Kadison}, there exists
a family $\{e_n\}_{n=1}^\infty$ of orthogonal projections in $\mathcal M$ such that $e=\sum_{n=1}^\infty e_n$ and $\tau(e_n)<\infty$ for all $n\in\mathbb N$.
\end{enumerate}

Recall that the weak$^*$-topology on $\mathcal{M}$ is the topology
on $\mathcal{M} $ induced from   the predual of $\mathcal{M}$.


\subsection{Norm ideals of semifinite von Neumann algebras}

\begin{definition}
\label{prelim_def1} A \emph{norm ideal} $\mathcal{K}_\Phi(\mathcal{M},\tau)$
of $(\mathcal{M},\tau)$ is a two sided ideal of $\mathcal{M}$ equipped with
a norm $\Phi : \mathcal{K}_\Phi(\mathcal{M},\tau)\rightarrow [0,\infty)$,
which satisfies

\begin{enumerate}
\item[(i)] $\displaystyle \Phi(uxv)=\Phi(x)$ for all $x\in \mathcal{K}_\Phi(%
\mathcal{M},\tau)$ and unitary elements $u,v$ in $\mathcal{M}$, i.e. $\Phi$
is unitarily invariant;

\item[(ii)] there exists $\lambda>0$ such that $\displaystyle \Phi(x) \ge
\lambda\|x\|$ for all $x\in \mathcal{K}_\Phi(\mathcal{M},\tau)$, i.e. $\Phi$
is $\|\cdot \|$-dominating;

\item[(iii)] $\mathcal{K}_\Phi(\mathcal{M},\tau)$ is a Banach space with
respect to the norm $\Phi$;

\item[(iv)] $\displaystyle \mathcal{F}(\mathcal{M},\tau)\subseteq \mathcal{K}%
_\Phi(\mathcal{M},\tau) \subseteq \mathcal{K }(\mathcal{M},\tau)$.
\end{enumerate}

The $\Phi$-norm closure of $\mathcal{F}(\mathcal{M},\tau)$ in $\mathcal{K}%
_\Phi(\mathcal{M},\tau)$ will be denoted by $\mathcal{K}_\Phi^0(\mathcal{M}%
,\tau)$, which is also a norm ideal of $(\mathcal{M},\tau)$. If $\mathcal{K}%
_\Phi^0(\mathcal{M},\tau)=\mathcal{K}_\Phi(\mathcal{M},\tau)$, then $%
\mathcal{K}_\Phi(\mathcal{M},\tau)$ is called a \emph{minimal} norm ideal of
$(\mathcal{M},\tau)$.
\end{definition}

\begin{remark}
For the purpose of convenience, if $x\notin \mathcal{K}_\Phi(\mathcal{M}%
,\tau)$, then we set $\Phi(x)=\infty.$
\end{remark}

\begin{example}
See \cite{Voi} for examples of norm ideals when $(\mathcal
M,\tau)=(\mathcal B(\mathcal H),Tr)$, where   $\mathcal H$ is a
separable complex Hilbert space and $Tr$ is a canonical trace of
$\mathcal B(\mathcal H)$.
\end{example}

\begin{example}
$\mathcal{K}(\mathcal{M},\tau)$ is a norm ideal of $(\mathcal{M},\tau)$ with
respect to the $\|\cdot\|$-norm.
\end{example}

We list some useful properties of a norm ideal in the next lemma.

\begin{lemma}
\label{prelim_lemma1} Suppose that $\mathcal{K}_\Phi(\mathcal{M},\tau)$ is a
norm ideal in $(\mathcal{M},\tau)$. Then the following statements are true.

\begin{enumerate}
\item[(i)] $\Phi(axb)\le \|a\|\Phi(x) \|b\|, $ for all $x\in \mathcal{K}%
_\Phi(\mathcal{M},\tau)$ and $a,b\in\mathcal{M}$.

\item[(ii)] For all $x\in \mathcal{K}_\Phi(\mathcal{M},\tau)$, $x^*\in
\mathcal{K}_\Phi(\mathcal{M},\tau)$ and $\Phi(x)=\Phi(x^*)=\Phi(|x|)$.

\item[(iii)] If $x, y\in \mathcal{K}_\Phi(\mathcal{M},\tau)$ such that $0\le
x\le y$, then $\Phi(x)\le \Phi(y)$.

\item[(iv)] If $x \in \mathcal{F}(\mathcal{M},\tau)$ and $R(x)$ is the range
projection of $x$ in $\mathcal{M}$, then $\Phi(x) \le \|x\|\Phi(R(x))$.

\item[(v)] Suppose that $\{x_n\}_{n=1}^\infty\subseteq \mathcal{K}_\Phi^0(\mathcal{M}%
,\tau)$ such that (1) $\sum_n x_n$ converges to $x\in\mathcal{M}$  in weak$^*
$-topology and (2) $\sum_n\Phi(x_n)<\infty$. Then $x\in \mathcal{K}_\Phi^0(%
\mathcal{M},\tau)$ and $\lim_k\Phi(x-\sum_{n=1}^k x_n)=0$.
\end{enumerate}
\end{lemma}

\begin{proof}
(i), (ii) , (iii) and (iv) are obvious. We will need only to show (v). Since
$\sum_n\Phi(x_n)<\infty$, it is clear that $\{\sum_{n=1}^kx_n\}_{k=1}^\infty$
is a Cauchy sequence in $\Phi$-norm. From the condition that $%
\{x_n\}_{n=1}^\infty\subseteq \mathcal{K}_\Phi^0(\mathcal{M},\tau)$,
we might assume that
$\sum_n x_n$ converges to $y \in \mathcal{K}_\Phi^0(\mathcal{M},\tau)$ in $%
\Phi$-norm. From the property that $\Phi$ is $\|\cdot \|$-dominating, it
follows that $\sum_n x_n$ converges to $y $ in $\|\cdot \|$-topology. Note
that $\sum_n x_n$ converges to $x\in\mathcal{M}$ in weak$^*$-topology. We
can conclude that $x=y$, which finishes the proof of the lemma.
\end{proof}

Examples of norm ideals of $(\mathcal{M},\tau)$ can come from the next
lemma. Recall that $L^r(\mathcal{M},\tau)$, for $1\le r<\infty$, is the
non-commutative $L^r$-spaces associated with $(\mathcal{M},\tau)$ and its
norm $\|\cdot\|_r$ is defined by
\begin{equation*}
\|x\|_r=(\tau(|x|^r)^{1/r}, \ \ \forall \ x\in L^r(\mathcal{M},\tau)
\end{equation*}
(see \cite{Pisier} for more details).

\begin{lemma}
\label{prelim_lemma2} Let $1\le r<\infty$ and $\mathcal{J}=L^r(\mathcal{M}%
,\tau)\cap \mathcal{M}$. Define a mapping $\Phi$ on $\mathcal{J}$ by
\begin{equation*}
\Phi(x) =\max \{ \|x\|_r,\|x\| \}, \ \ \forall \ x\in \mathcal{J}.
\end{equation*}
Then $\mathcal{J}$ is a norm ideal of $(\mathcal{M},\tau)$ with respect to
the norm $\Phi$.
\end{lemma}

\begin{proof}
It is not hard to check that $\mathcal{J}$ is a two sided ideal of $\mathcal{%
M}$ and $\Phi$ is a norm on $\mathcal{J}$. Moreover, $\Phi$ is unitarily
invariant and $\|\cdot \|$-dominating.

We will show that $\mathcal{J}$ is a Banach space with respect to $\Phi$.
Actually, assume that $\{x_n\}$ is a Cauchy sequence in $\Phi$-norm. Thus $%
\{x_n\}$ is a Cauchy sequence in both $\|\cdot\|_r$-norm and $\|\cdot\|$%
-norm. Therefore there exist $y_1\in L^r(\mathcal{M},\tau)$ and $y_2\in
\mathcal{M}$ such that $\lim_n\|y_1-x_n\|_r = 0$ and $\lim_n\|y_2-x_n\|= 0$.
From Theorem 5 in \cite{Nelson}, we get that $\{x_n\}$ converges to both $y_1
$ and $y_2$ in measure topology. This implies $y_1=y_2\in\mathcal{J}$. Hence
$\mathcal{J}$ is a Banach space with respect to $\Phi$.

It is easy to see that $\mathcal{F}(\mathcal{M},\tau)\subseteq \mathcal{J}$.
Assume that $0\le x\in \mathcal{M}$ such that $\|x\|_r<\infty$. For any $%
\lambda>0$, let $e_{(\lambda,\infty)}$ be the spectral projection of $x$
onto $\sigma(x)\cap (\lambda,\infty)$. Thus $0\le \lambda \cdot
e_{(\lambda,\infty)}\le x$. It follows that $\|\lambda \cdot
e_{(\lambda,\infty)}\|_r\le \|x\|_r<\infty$. Hence $e_{(\lambda,\infty)}\in
\mathcal{F}(\mathcal{M},\tau).$ Combining with the fact that $\|x-x \cdot
e_{(\lambda,\infty)}\|\le \lambda$, we conclude that $x\in \mathcal{K}(%
\mathcal{M},\tau).$ Hence, $\mathcal{J }\subseteq \mathcal{K}(\mathcal{M}%
,\tau).$

Therefore, $\mathcal{J}$ is a norm ideal of $(\mathcal{M},\tau)$ with
respect to the norm $\Phi$.
\end{proof}

\begin{remark}
In the proof of last lemma, from the fact that $\tau $ is a normal weight of
$\mathcal{M}$ and $\Vert x\Vert _{r}<\infty $, it follows that
\begin{equation*}
\lim_{\lambda \rightarrow 0^{+}}\Vert x-x\cdot e_{(\lambda ,\infty )}\Vert
_{r}=0.
\end{equation*}%
Thus such $\mathcal{J}$ is actually a minimal norm ideal of $(\mathcal{M}%
,\tau )$ with respect to the norm $\Phi $.
\end{remark}

\begin{definition}
\label{prelim_def2} Let $1\le r<\infty$. We define $\mathcal{K}_r(\mathcal{M}%
,\tau)$ to be $L^r(\mathcal{M},\tau)\cap \mathcal{M}$ equipped with the norm
$\Phi$ satisfying $\Phi(x) =\max \{ \|x\|_r,\|x\| \},$ for all $x\in L^r(%
\mathcal{M},\tau)\cap \mathcal{M}. $ Thus $\mathcal{K}_r(\mathcal{M},\tau)$
is a minimal norm ideal of $(\mathcal{M},\tau)$.
\end{definition}

\begin{remark}
More examples of norm ideals in $(\mathcal{M},\tau)$ can be found in
noncommutative Banach functional spaces (see \cite{Pagter} for details).
\end{remark}

\subsection{Extension of a norm ideal $\mathcal{K}_\Phi(\mathcal{M},\protect%
\tau)$ from $\mathcal{M}$ to $\mathcal{M}\otimes \mathcal{B}(\mathcal{H})$}

\label{section2.2}

\quad \newline

Let $\mathcal{H}$ be an infinite dimensional separable Hilbert space and let $%
\mathcal{B}(\mathcal{H})$ be the set of bounded linear operators on $%
\mathcal{H}$. Suppose that $\{f_{i,j}\}_{i,j= 1}^\infty$ is a system of
matrix units of $\mathcal{B}(\mathcal{H})$.

Recall that $\mathcal{M}$ is a countably decomposable, properly infinite von
Neumann algebra  with a faithful, normal, semifinite tracial weight $\tau$.
There exists a sequence $\{v_i\}_{i = 1}^\infty$ of partial isometries in $%
\mathcal{M}$ such that
\begin{equation*}
v_iv_i^*=I_{\mathcal{M}}, \ \ \ \ \sum_{i = 1}^\infty v_i^*v_i=I_{\mathcal{M}%
}, \ \ \ \ \text{ and } v_jv_i^*=0 \text { when } i\ne j.
\end{equation*}

Let $\mathcal{M}\otimes \mathcal{B}(\mathcal{H})$ be a von Neumann algebra
tensor product of $\mathcal{M}$ and $\mathcal{B}(\mathcal{H})$.

\begin{definition}
We introduce following two mappings:
\begin{equation*}
\phi: \mathcal{M}\rightarrow \mathcal{M}\otimes \mathcal{B}(\mathcal{H}) \ \
\ \text{ and } \ \ \psi: \mathcal{M}\otimes \mathcal{B}(\mathcal{H})
\rightarrow \mathcal{M }
\end{equation*}
defined by, for all $x\in \mathcal{M }$ and all \ $\sum_{i,j= 1}^\infty
x_{i,j}\otimes f_{i,j}\in \mathcal{M}\otimes \mathcal{B}(\mathcal{H})$,
\begin{equation*}
\phi(x) =\sum_{i,j= 1}^\infty (v_ixv^*_j)\otimes f_{i,j} \ \ \ \ \text{ and }
\ \ \ \ \psi( \sum_{i,j= 1}^\infty x_{i,j}\otimes f_{i,j} )= \sum_{i,j=
1}^\infty v_i^*x_{i,j}v_j.
\end{equation*}
\end{definition}

\begin{lemma}
\label{prelim_lemma2.6} Both $\phi$ and $\psi$ are normal $*$-homomorphisms
satisfying
\begin{equation*}
\text{ $\psi \circ \phi=id_{\mathcal{M}}$ \qquad and \qquad $\phi \circ
\psi=id_{\mathcal{M}\otimes \mathcal{B}(\mathcal{H})}.$}
\end{equation*}
\end{lemma}

\begin{proof}
It can be verified directly that $\phi$ and $\psi$ are $*$-homomorphisms
satisfying
\begin{equation*}
\text{ $\psi \circ \phi=id_{\mathcal{M}}$ \qquad and \qquad $\phi \circ
\psi=id_{\mathcal{M}\otimes \mathcal{B}(\mathcal{H})}.$}
\end{equation*}
From Corollary III.3.10 in \cite{Takesaki}, both $\phi$ and $\psi$ are
normal.
\end{proof}

\begin{definition}
\label{prelim_def2.7} We will further define a mapping $\tilde \tau :(%
\mathcal{M}\otimes \mathcal{B}(\mathcal{H}))^+\rightarrow
[0,\infty]$ to be
\begin{equation*}
\tilde \tau (y)=\tau(\psi(y)), \qquad \forall \ y\in
(\mathcal{M}\otimes \mathcal{B}(\mathcal{H}))^+.
\end{equation*}
\end{definition}

\noindent From Lemma \ref{prelim_lemma2.6} and Definition \ref{prelim_def2.7}%
, it follows our next result.

\begin{lemma}
\label{prelim_lemma2.7}

\begin{enumerate}
\item[(i)] $\displaystyle\tilde \tau$ is a faithful, normal, semifinite tracial weight
of $\mathcal{M}\otimes \mathcal{B}(\mathcal{H})$.

\item[(ii)] $\displaystyle\tilde \tau(\sum_{i,j=1}^\infty x_{i,j}\otimes
f_{i,j})=\sum_{i=1}^\infty \tau(x_{i,i})$ for all \ $\displaystyle%
\sum_{i,j=1}^\infty x_{i,j}\otimes f_{i,j}\in (\mathcal{M}\otimes \mathcal{B}%
(\mathcal{H}))^+$.

\item[(iii)]  We have
\begin{equation*}
\begin{aligned}
  \mathcal P\mathcal F(\mathcal M\otimes \mathcal B(\mathcal H), \tilde \tau)&=\phi(\mathcal P\mathcal F(\mathcal M,\tau)), \quad \\
   \mathcal F(\mathcal M\otimes \mathcal B(\mathcal H), \tilde \tau)&=\phi(\mathcal F(\mathcal M,\tau)), \quad \\
    \mathcal K(\mathcal M\otimes \mathcal B(\mathcal H), \tilde \tau)&=\phi(\mathcal K(\mathcal M,\tau)).\end{aligned}
\end{equation*}
\end{enumerate}
\end{lemma}

\begin{remark}
Preceding lemma shows that $\tilde \tau$ is a natural extension of $\tau$
from $\mathcal{M}$ to $\mathcal{M}\otimes \mathcal{B}(\mathcal{H})$. If no
confusion arises, $\tilde\tau$ will be also denoted by $\tau$.
\end{remark}

\begin{definition}
\label{prelim_def2.10}  Let $\mathcal{K}_{\Phi}(\mathcal{M},\tau)$ be a norm
ideal of $(\mathcal{M},\tau)$ equipped with a norm $\Phi$. We define
\begin{equation*}
\tilde {\mathcal{K}}_{\tilde \Phi} = \phi(\mathcal{K}_{\Phi}(\mathcal{M}%
,\tau)),
\end{equation*}
and
\begin{equation*}
\tilde \Phi(y)=\Phi(\psi(y)), \qquad \forall y\in \tilde {\mathcal{K}}%
_{\tilde \Phi}.
\end{equation*}
\end{definition}

\begin{lemma}
\label{prelim_lemma2.11} The following statements are true.

\begin{enumerate}
\item[(i)] $\tilde {\mathcal{K}}_{\tilde \Phi} $ is a norm ideal of $(%
\mathcal{M}\otimes \mathcal{B}(\mathcal{H}), \tau)$ with respect to the norm
$\tilde \Phi$.

\item[(ii)] If $\sum_{i,j=1}^\infty x_{i,j}\otimes f_{i,j}\in \tilde {%
\mathcal{K}}_{\tilde \Phi}$, then $x_{i,j}\in \mathcal{K}_{\Phi}(\mathcal{M}%
,\tau)$ for all $i,j\ge 1$.

\item[(iii)] If $x \in \mathcal{K}_{\Phi}(\mathcal{M},\tau)$, then $x\otimes
f_{i,j}\in \tilde {\mathcal{K}}_{\tilde \Phi}$ and $\Phi(x)=\tilde\Phi(x%
\otimes f_{i,j})$ for all $i,j\ge 1$.
\end{enumerate}
\end{lemma}

\begin{proof}
The results follow from definitions of $\tilde {\mathcal{K}}_{\tilde \Phi} $%
, $\tilde \Phi$, and Lemma \ref{prelim_lemma2.6}, Lemma \ref{prelim_lemma2.7}%
.
\end{proof}

\begin{remark}
Lemma \ref{prelim_lemma2.11} shows that if we identify $\mathcal{M}$ with $%
\mathcal{M}\otimes f_{1,1}$ in $\mathcal{M}\otimes \mathcal{B}(\mathcal{H})$%
, then $\tilde {\mathcal{K}}_{\tilde \Phi} $ and $\tilde \Phi$ can be viewed
as extensions of $\mathcal{K}_{\Phi}(\mathcal{M},\tau)$ and $\Phi$ from $%
\mathcal{M}$ to $\mathcal{M}\otimes \mathcal{B}(\mathcal{H})$.
\end{remark}

\begin{proposition}
\label{prelim_prop2.12}  The norm ideal $\tilde {\mathcal{K}}_{\tilde \Phi} $
equipped with the norm  $\tilde \Phi$ is independent of the choice of the
system of matrix units $\{f_{i,j}\}_{i,j=1}^\infty$  of $\mathcal{B}(%
\mathcal{H})$ and the choice of the family $\{v_i\}_{i =1}^\infty$ of
partial isometries in $\mathcal{M}$.
\end{proposition}

\begin{proof}
Assume that $\{\hat f_{i,j}\}_{i,j=1}^\infty$ is a family of system of
matrix units of $\mathcal{B}(\mathcal{H})$ and $\{\hat v_i\}_{i =1}^\infty$
is a family of partial isometries on $\mathcal{M}$ such that
\begin{equation*}
\hat v_i\hat v_i^*=I_{\mathcal{M}}, \ \ \ \ \sum_i \hat v_i^*\hat v_i=I_{%
\mathcal{M}}, \ \ \ \ \text{ and } \hat v_j\hat v_i^*=0 \text { when } i\ne
j.
\end{equation*}
Thus there exist unitary elements $u$ in $\mathcal{M}$ and $w\in \mathcal{B}(%
\mathcal{H})$ such that
\begin{equation*}
u^*v_ju=\hat v_j \qquad \text{ and } \qquad w^* f_{i,j}w=\hat f_{i,j}.
\qquad \forall \ i,j\ge 1.
\end{equation*}
Let $\hat \phi$, $\hat \psi$, $\hat {\mathcal{K}}_{\hat \Phi}$ and ${\hat
\Phi}$ be defined accordingly as $\phi$, $\psi$, ${\mathcal{K}}_{ \Phi}$ and
${\ \Phi}$.  We have
\begin{align}
\sum_{i,j=1}^\infty x_{i,j} \otimes f_{i,j}\in \tilde {\mathcal{K}}_{\tilde
\Phi} & \Leftrightarrow \sum_{i,j=1}^\infty u x_{i,j}u^*\otimes f_{i,j} \in
\tilde {\mathcal{K}}_{\tilde \Phi}  \tag{as $\tilde {\mathcal K}_{\tilde
\Phi}$ is a two-sided ideal} \\
& \Leftrightarrow \sum_{i,j=1}^\infty v_i^*ux_{i,j}u^*v_j \in \mathcal{K}%
_{\Phi}(\mathcal{M},\tau)  \tag{by definition of $\tilde {\mathcal
K}_{\tilde \Phi}$} \\
& \Leftrightarrow \sum_{i,j=1}^\infty u^*v_i^*u x_{i,j}u^*v_ju \in \mathcal{K%
}_{\Phi}(\mathcal{M},\tau)  \tag{as $\mathcal K_{\Phi}(\mathcal M,\tau)$ is
a two-sided ideal} \\
& \Leftrightarrow \sum_{i,j=1}^\infty \hat v_i^* x_{i,j} \hat v_j \in
\mathcal{K}_{\Phi}(\mathcal{M},\tau)  \tag{by the choice of $u$} \\
& \Leftrightarrow \sum_{i,j=1}^\infty x_{i,j} \otimes \hat f_{i,j} \in \hat {%
\mathcal{K}}_{\hat \Phi}  \tag{by definition of $\hat {\mathcal K}_{\hat
\Phi}$} \\
& \Leftrightarrow \sum_{i,j=1}^\infty x_{i,j} \otimes w^* \hat f_{i,j}w \in
\hat {\mathcal{K}}_{\hat \Phi}  \tag{as $\hat {\mathcal K}_{\hat \Phi}$ is a
two-sided ideal} \\
& \Leftrightarrow \sum_{i,j=1}^\infty x_{i,j} \otimes f_{i,j} \in \hat {%
\mathcal{K}}_{\hat \Phi}.  \tag{by the choice of $w$}
\end{align}
So
\begin{equation*}
\tilde {\mathcal{K}}_{\tilde \Phi}= \hat {\mathcal{K}}_{\hat \Phi}.
\end{equation*}
Similarly, it can be shown that
\begin{equation*}
\tilde \Phi(\sum_{i,j=1}^\infty x_{i,j} \otimes f_{i,j}) = \hat \Phi
(\sum_{i,j=1}^\infty x_{i,j}\otimes f_{i,j}),  \qquad \text { for all } \ \
\sum_{i,j=1}^\infty x_{i,j}\otimes f_{i,j}\in \tilde {\mathcal{K}}_{\tilde
\Phi} .
\end{equation*}
This completes the proof of the proposition.
\end{proof}

\begin{remark}
\label{prelim_rem2.2.10} Because of Lemma \ref{prelim_lemma2.11} and
Proposition \ref{prelim_prop2.12}, if no confusion arises, we will use $%
\mathcal{K}_{\Phi }(\mathcal{M}\otimes \mathcal{B}(\mathcal{H}),\tau )$ and $%
\Phi $ to denote $\tilde{\mathcal{K}}_{\tilde{\Phi}}$ and $\tilde{\Phi}$ in
Definition \ref{prelim_def2.10}.
\end{remark}

\subsection{Approximately unitarily equivalent mappings}

\quad \newline

Let $\mathcal{H}$ be an infinite dimensional separable Hilbert space and let $%
\mathcal{B}(\mathcal{H})$ be the set of bounded linear operators on $%
\mathcal{H}$. Recall that $\mathcal{M}$ is a countably decomposable, properly
infinite von Neumann algebra  with a faithful, normal, semifinite tracial
weight $\tau$. Let $\mathcal{M}\otimes \mathcal{B}(\mathcal{H})$ be a von
Neumann algebra tensor product of $\mathcal{M}$ with $\mathcal{B}(\mathcal{H}%
)$.

Let $\mathcal{K}_\Phi(\mathcal{M},\tau)$ be a norm ideal of $\mathcal{M}$.
Let $\mathcal{K}_\Phi(\mathcal{M}\otimes \mathcal{B}(\mathcal{H}),\tau)$ be
a norm ideal of $\mathcal{M}\otimes \mathcal{B}(\mathcal{H})$ extended from $%
\mathcal{K}_\Phi(\mathcal{M},\tau)$ (see Definition \ref{prelim_def2.10} and
Remark \ref{prelim_rem2.2.10}).

Let $\mathcal{A}$ be a separable $C^*$-subalgebra of $\mathcal{M}$ with an
identity $I_{\mathcal{A}}$ and $\mathcal{B}$ a $*$-subalgebra of $\mathcal{A}
$ such that $I_{\mathcal{A}}\in\mathcal{B}$.

Assume $\psi$ is a positive mapping from $\mathcal{A}$ into $\mathcal{M}$
such that $\psi(I_{\mathcal{A}})$ is a projection in $\mathcal{M}$. Then for
all $0\le x\in \mathcal{A}$, we have $0\le \psi(x) \le \|x\|\psi(I_{\mathcal{%
A}})$. Therefore, $\psi(x)\psi(I_{\mathcal{A}})=\psi(I_{\mathcal{A}%
})\psi(x)=\psi(x)$ for all positive $x\in \mathcal{A}$. In other words, $%
\psi(I_{\mathcal{A}})$ can be viewed as an identity of $\psi(\mathcal{A})$.
Or, $\psi(\mathcal{A})\subseteq \psi(I_{\mathcal{A}})\mathcal{M}\psi(I_{%
\mathcal{A}})$.

\begin{definition}
\label{prelim_def3.1.1} Suppose $\{f_{i,j}\}_{i,j\ge 1}$ is a system of
matrix units of $\mathcal{B}(\mathcal{H})$. Let $M,N\in \mathbb{N}%
\cup\{\infty\}$. Suppose that $\psi_1,\ldots, \psi_M$ and $\phi_1,\ldots,
\phi_N$ are positive mappings from $\mathcal{A}$ into $\mathcal{M}$ such
that $\psi_1(I_{\mathcal{A}}),\ldots, \psi_M(I_{\mathcal{A}})$, $\phi_1(I_{%
\mathcal{A}}),\ldots, \phi_N(I_{\mathcal{A}})$ are projections in $\mathcal{M%
}$.

\begin{enumerate}
\item[(a)] Let $F\subseteq \mathcal{B}$ be finite and $\epsilon>0$. Then we
call
\begin{equation*}
\text{ $\psi_1\oplus \cdots \oplus\psi_M$ is \emph{$(F,\epsilon)$%
-strongly-approximately-unitarily-equivalent} to $\phi_1\oplus\cdots\oplus
\phi_N$ over $\mathcal{B}$,}
\end{equation*}
denoted by
\begin{equation*}
\psi_1\oplus\psi_2\oplus\cdots \oplus\psi_M \sim_{\mathcal{B}%
}^{(F,\epsilon)} \phi_1\oplus \phi_2\oplus\cdots\oplus \phi_N, \qquad \mod (%
\mathcal{K}_{\Phi}(\mathcal{M},\tau) )
\end{equation*}
if there exists a partial isometry $v$ in $\mathcal{M}\otimes \mathcal{B}(%
\mathcal{H})$ such that

\begin{enumerate}
\item[(i)] $\displaystyle v^*v= \sum_{i=1}^M \psi_i(I_{\mathcal{A}})\otimes
f_{i,i}$ and $\displaystyle vv^*= \sum_{i=1}^N \phi_i(I_{\mathcal{A}%
})\otimes f_{i,i}$;

\item[(ii)] $\displaystyle  \sum_{i=1}^M \psi_i(x)\otimes f_{i,i} - v^*\left
(\sum_{i=1}^N \phi_i(x)\otimes f_{i,i}\right ) v \in \mathcal{K}_\Phi(%
\mathcal{M}\otimes \mathcal{B}(\mathcal{H}),\tau)$ for all $x\in \mathcal{B}$%
;

\item[(iii)] $\displaystyle \Phi\left ( \sum_{i=1}^M \psi_i(x)\otimes
f_{i,i} - v^*\left (\sum_{i=1}^N \phi_i(x)\otimes f_{i,i}\right ) v
\right)<\epsilon$ for all $x\in F$.
\end{enumerate}

\item[(b)] We call
\begin{equation*}
\text{ $\psi_1\oplus \cdots \oplus\psi_M$ is \emph{%
strongly-approximately-unitarily-equivalent } to $\phi_1\oplus\cdots\oplus
\phi_N$ over $\mathcal{B}$,}
\end{equation*}
denoted by
\begin{equation*}
\psi_1\oplus\psi_2\oplus\cdots \oplus\psi_M \sim_{\mathcal{B}} \phi_1\oplus
\phi_2\oplus\cdots\oplus \phi_N, \qquad \mod (\mathcal{K}_{\Phi}(\mathcal{M}%
,\tau))
\end{equation*}
if, for any finite subset $F\subseteq \mathcal{B}$ and $\epsilon>0$,
\begin{equation*}
\psi_1\oplus\psi_2\oplus\cdots \oplus\psi_M \sim_{\mathcal{B}%
}^{(F,\epsilon)} \phi_1\oplus \phi_2\oplus\cdots\oplus \phi_N, \qquad \mod (%
\mathcal{K}_{\Phi}(\mathcal{M},\tau)).
\end{equation*}
\end{enumerate}
\end{definition}

\begin{remark}
If no confusion arises, $\sum_{i=1}^M \psi_i(x)\otimes f_{i,i}$ and $%
\sum_{i=1}^N \phi_i(x)\otimes f_{i,i}$ are also denoted by $\psi_1(x)\oplus
\cdots \oplus \psi_M(x)$, and $\phi_1(x)\oplus \cdots \oplus \phi_N(x)$
respectively.
\end{remark}

\begin{remark}
It is clear from the definition that $\sim _{\mathcal{B}}$ is an equivalence
relation and $\sim _{\mathcal{B}}^{(F,\epsilon )}$ is a reflexive and
symmetric relation. By adding $\epsilon $'s, $\sim _{\mathcal{B}%
}^{(F,\epsilon )}$ may also be viewed as a \textquotedblleft \emph{transitive%
}\textquotedblright\ relation.
\end{remark}

\section{Perturbations of A Self-adjoint Element in Semifinite von Neumann
Algebras}

Let $\mathcal{M}$ be a countably decomposable, properly infinite von
Neumann algebra with a faithful  normal  semifinite tracial weight
$\tau$
and let  $\mathcal{P}\mathcal{F}(\mathcal{M},\tau),  \mathcal{F}(\mathcal{M}%
,\tau),  \mathcal{K}(\mathcal{M},\tau) $ be the sets of finite rank
projections, finite rank operators, and compact operators respectively, in $(%
\mathcal{M},\tau)$ (see Section \ref{sec2} for their definitions).

\subsection{Some lemmas}

\ \newline

In this section, we are interested in a class of norm ideals $\mathcal{K}%
_{\Phi}(\mathcal{M},\tau)$ in $(\mathcal{M},\tau)$ that satisfies
\begin{equation*}
\lim_{\overset{\tau(e)\rightarrow\infty}{e\in \mathcal{P}\mathcal{F}(%
\mathcal{M},\tau)}}\frac{\Phi(e)}{\tau(e)} =0.
\end{equation*}
As shown by next lemma, in the case of semifinite factors, this condition is
closely related to Kuroda's characterization of $\|\cdot\|_1$-norm of $%
\mathcal{M}$ (see \cite{Kuroda}).

\begin{lemma}
Let $\mathcal{N}$ be a countably decomposable,  properly infinite,
semifinite
factor with a faithful  normal  semifinite tracial weight $\tau$ and let $%
\mathcal{K}_{\Phi}(\mathcal{N},\tau)$ be a norm ideal of $(\mathcal{N},\tau)$%
. Then the following statements are equivalent.

\begin{enumerate}
\item[(i)] There exists a positive number $\lambda$ such that
\begin{equation*}
\Phi(e)\ge \lambda\tau(e), \qquad \forall \ e\in \mathcal{P}\mathcal{F}(%
\mathcal{N},\tau).
\end{equation*}

\item[(ii)] The limit
\begin{equation*}
\lim_{\overset{\tau(e)\rightarrow\infty}{e\in \mathcal{P}\mathcal{F}(%
\mathcal{N},\tau)}}\frac{\Phi(e)}{\tau(e)} \quad \text{ exists and is
positive}.
\end{equation*}
\end{enumerate}
\end{lemma}

\begin{proof}
To prove the equivalence between (i) and (ii), it suffices to show
the following statement: if $e$ and $f$ are projections in
$\mathcal{N}$ such that $0<\tau(e)\le \tau(f)<\infty$, then
$\displaystyle \frac{\Phi(f)}{\tau(f)}\le \frac{\Phi(e)}{\tau(e)}.$

Since $\mathcal{N}$ is a factor, we might assume that $e\le f$. Note that $f%
\mathcal{N}f$ is a finite factor with a tracial state $\tau_f$, defined by $%
\displaystyle \tau_f(x)=\frac {\tau(x)}{\tau(f)}$ for all $x\in f\mathcal{N }%
f.$ Let $\Phi_f$ be a mapping on $f\mathcal{N}f$ such that $\displaystyle %
\Phi_f(x)= \frac {\Phi(x)}{\Phi(f)}$ for all $x\in f\mathcal{N }f.$ Thus $%
\Phi_f$ is a normalized unitarily invariant norm on a finite factor $f%
\mathcal{N}f$. By Corollary 3.31 in \cite{Fang}, we know that $\tau_f(e) \le
\Phi_f(e)$, whence
\begin{equation*}
\frac{\Phi(f)}{\tau(f)}\le \frac{\Phi(e)}{\tau(e)}.
\end{equation*}
This ends the proof of the lemma.
\end{proof}

The following easy result is sometime useful.

\begin{lemma}
\label{Lemma 3.1.2} Let $\mathcal{K}_{\Phi}(\mathcal{M},\tau)$ be a norm
ideal of $(\mathcal{M},\tau)$. Assume $\delta$ is a positive number such
that
\begin{equation*}
\sup\{\tau(e):e\mbox{ is a minimal projection in } \mathcal{M}\ \}\leq
\delta .
\end{equation*}
If $\lambda>0$ and $e_1$ is projection in $\mathcal{M}$ such that $%
\tau(e_1)\le \lambda$, then there exists a projection $f_1$ in $\mathcal{M}$
such that
\begin{equation*}
\lambda\le \tau(f_1)\le\lambda+\delta \qquad \text { and } \Phi(e_1)\le
\Phi(f_1).
\end{equation*}
\end{lemma}

\begin{proof}
From the assumption on minimal projections in $\mathcal{M}$, it induces that
there exists a projection $f_1$ in $\mathcal{M}$ such that $e_1\le f_1$ and $%
\lambda\le \tau(f_1)\le \lambda+\delta$. Hence the result of the lemma holds.
\end{proof}

Recall that, for each $x\in\mathcal{M}$, we let $R(x)$ be the range
projection of $x$ in $\mathcal{M}$.

\begin{lemma}
\label{Lemma 3.1.3} Let $a\in \mathcal{M} $ be a self-adjoint element. For
each $x\in{\mathcal{F}(\mathcal{M},\tau)}$ and $m\in{\mathbb{N}}$, there
exist a projection $q$ in ${\mathcal{F}(\mathcal{M},\tau)}$ and self-adjoint
elements $d$, $b$ in ${\mathcal{F}(\mathcal{M},\tau)}$ such that:

\begin{enumerate}
\item[(i)] $(I-q)x=0$;

\item[(ii)] $d=dq=qd$ is diagonal and $\|d\|\le \|a\|$;

\item[(iii)] $a=d+b+(I-q)a(I-q)$;

\item[(iv)] $\Vert b\Vert\leq 3\Vert a\Vert/m$, and $\tau(R(b))\leq
2m\cdot\tau(R(x^{}))$.
\end{enumerate}
\end{lemma}

\begin{proof}
The result is trivial when $x=0$. Therefore, we will assume that $x\ne 0$.

Note $\sigma(a)$, the spectrum of $a$ in $\mathcal{M}$, is contained in the
interval $[-\|a\|,\|a\|]$. We partition $[-\|a\|,\|a\|]$ into $m$ equal
subintervals, $\Delta_1,\ldots, \Delta_m$, having length $2\|a\|/m$. For $%
1\le j\le m$, let $e_j$ be the spectral projection of $a$ onto $%
\sigma(a)\cap \Delta_j$ and let $q_{j}=R(e_{{j}}x)$, the range projection of
$e_jx$ in $\mathcal{M}$. Let $\lambda_j$ be a midpoint of $\Delta_j$.

Define
\begin{equation*}
d\triangleq\sum^{m}_{j=1}\lambda_{j}q_{j}, \quad
q\triangleq\sum^{m}_{j=1}q_{j} \quad \text{ and } \quad  b\triangleq
a-d-(I-q)a(I-q).
\end{equation*}
Thus, the following statements hold:

\begin{enumerate}
\item[(a)] $q_{j}\leq e_{{j}}$, whence $q_{i}\perp q_{j}$ for $i\neq j$;

\item[(b)] $\tau(q_{j})=\tau(R(e_{{j}}x))=\tau(R(x^{\ast}e_{{j}%
}))\leq\tau(R(x^{\ast}))=\tau(R(x^{}))< \infty$;

\item[(c)] $\tau(q)\leq m\cdot\tau(R(x^{}))$, by definition of $q$.
\end{enumerate}

Now, we will verify that $q$, $d$ and $b$ have the stated properties.

(i) The equality
\begin{equation*}
qx=\sum^{m}_{j=1}q_{j}x=\sum^{m}_{j=1}q_{j}e_{{j}}x=\sum^{m}_{j=1}e_{{j}}x=x
\end{equation*}
shows that $(I-q)x=0$.

(ii) From (a) and the definition of $d$, it is clear that $d=dq=qd$ is
diagonal and $\|d\|\le \|a\|$.

(iii) By the definition of $b$, we know that $a=d+b+(I-q)a(I-q)$.

(iv) From the definition of $b$, it follows that
\begin{equation*}
\begin{array}{rcl}
b & = & qaq-\sum^{m}_{j=1}\lambda_{j}q_{j}+qa(I-q)+(qa(I-q))^{\ast} \\
& = & q\sum^{m}_{j=1}(a-\lambda_{j})q_{j}+qa(I-q)+(qa(I-q))^{\ast}.%
\end{array}
\end{equation*}
Combining it with (c), one obtains that $\tau(R(b))\leq 2m\cdot\tau(R(x^{}))$%
. Note that
\begin{equation*}
\begin{array}{rcl}
\Vert(I-q)aq\Vert & = & \Vert(I-q)a\sum^{m}_{j=1}q_{j}\Vert \\
& = & \Vert(I-q)a\sum^{m}_{j=1}q_{j}-(I-q)\sum^{m}_{j=1}\lambda_{j}q_{j}\Vert
\\
& \leq & \Vert\sum^{m}_{j=1}(a-\lambda_{j})q_{j}\Vert. \\
\end{array}
\end{equation*}
Since $q_{j}$ is a subprojection of $e_{{j}}$, it follows that
\begin{equation*}
\Vert\sum^{m}_{j=1}(a-\lambda_{j})q_{j}\Vert=\Vert\sum^{m}_{j=1}e_{{j}%
}(a-\lambda_{j})q_{j}e_{{j}}\Vert.
\end{equation*}
Therefore $\sum^{m}_{j=1}(a-\lambda_{j})q_{j}$ is block diagonal and
\begin{equation*}
\Vert\sum^{m}_{j=1}(a-\lambda_{j})q_{j}\Vert=\max^{}_{1\leq j\leq m}\Vert e_{%
{j}}(a-\lambda_{j})q_{j}e_{{j}}\Vert.
\end{equation*}
The fact that $e_{{j}}$'s are spectral projections of $a$ and our choice of $%
\lambda_j$ ensure that
\begin{equation*}
\Vert(a-\lambda_{j})q_{j}\Vert\leq\Vert e_{{j}}(a-\lambda_{j})e_{{j}%
}\Vert\leq \Vert a\Vert/m.
\end{equation*}
Thus, it follows that $\Vert b\Vert\leq 3\Vert a\Vert/m$.

This ends the proof of the lemma.
\end{proof}

\begin{proposition}
\label{Lemma 3.1.4} Let $\mathcal{K}_{\Phi}(\mathcal{M},\tau)$ be a norm
ideal of $(\mathcal{M},\tau)$. Assume that

\begin{enumerate}
\item[(a)] there exists a positive number $\delta$ such that
\begin{equation*}
\sup\{\tau(e):e \mbox{ is a minimal projection in } \mathcal{M}\ \}\leq
\delta;
\end{equation*}

\item[(b)] $\displaystyle
\lim_{\overset{\tau(e)\rightarrow\infty}{e\in \mathcal{P}\mathcal{F}(%
\mathcal{M},\tau)}}\frac{\Phi(e)}{\tau(e)} =0. $
\end{enumerate}

Let $a\in\mathcal{M}$ be self-adjoint.  Then, for every $x\in{\mathcal{F}(%
\mathcal{M},\tau)}$ and $\epsilon>0$, there exist a projection $q$ in ${%
\mathcal{F}(\mathcal{M},\tau)}$ and self-adjoint elements $d$ and $b$ in ${%
\mathcal{F}(\mathcal{M},\tau)}$ such that:

\begin{enumerate}
\item[(i)] $(I-q)x=0$;

\item[(ii)] $d=dq=qd$ is diagonal and $\|d\|\le \|a\|$;

\item[(iii)] $a=d+b+(I-q)a(I-q)$;

\item[(iv)] $\Phi(b)\leq \epsilon$.
\end{enumerate}
\end{proposition}

\begin{proof}
The result is trivial when $x=0$. In the following, we will assume that $%
x\ne 0$. By using Lemma \ref{Lemma 3.1.3}, given $x\in{\mathcal{F}(\mathcal{M%
},\tau)}$, there exist a sequence $\{q_m\}_{m=1}^\infty$ of projections and
two sequences $\{d_m\}_{m=1}^\infty$, $\{b_m\}_{m=1}^\infty$ of self-adjoint
elements in ${\mathcal{F}(\mathcal{M},\tau)}$, such that:

\begin{enumerate}
\item[(i$_1$)] $(I-q_m)x=0$;

\item[(ii$_1$)] $d_m=d_mq_m=q_md_m$ is diagonal and $\|d_m\|\le \|a\|$;

\item[(iii$_1$)] $a=d_m+b_m+(I-q_m)a(I-q_m)$;

\item[(iv$_1$)] $\Vert b_m\Vert\leq 3\Vert a\Vert/m$ and $\tau(R(b_m))\leq
2m\cdot\tau(R(x^{}))$.
\end{enumerate}

For each $b_{m}$, Lemma \ref{prelim_lemma1} yields that
\begin{equation*}
\Phi(b_{m})\le\Vert b_{m}\Vert\Phi(R(b_{m})).
\end{equation*}
By the Assumption (a), Lemma \ref{Lemma 3.1.2} and (iv$_1$), for each $R(b_m)
$ there exists a projection $f_m$ in $\mathcal{M}$ such that
\begin{equation*}
2m\cdot\tau(R(x^{}))  \leq\tau(f_m)\leq 2m\cdot\tau(R(x^{}))+\delta
\end{equation*}
and
\begin{equation*}
\Phi(R(b_{m}))\le\Phi(f_m).
\end{equation*}
Therefore,
\begin{equation*}
\Phi(b_{m})\leq\Vert b_{m}\Vert\Phi(R(b_{m}))\leq\frac{3\Vert a\Vert}{m}%
\Phi(f_{m} )= {3\Vert a\Vert} \frac {\Phi(f_{m} )}{\tau(f_m)} \frac {%
\tau(f_m)}{m} .
\end{equation*}
Note that $\tau(f_m)\rightarrow\infty$ when $m\rightarrow\infty$. By
applying the Assumption (b), for $m$ large enough we can find self-adjoint
elements $d_m$, $b_m$ and a projection $q_m$, all in $\mathcal{F}(\mathcal{M}%
, \tau)$, with desired properties.
\end{proof}

\begin{proposition}
\label{prop3.1.5} Let $\mathcal{K}_{\Phi}(\mathcal{M},\tau)$ be a norm ideal
of $(\mathcal{M},\tau)$. Assume that

\begin{enumerate}
\item[(a)] there exists a positive number $\delta$ such that
\begin{equation*}
\sup\{\tau(e):e \mbox{ is a minimal projection in } \mathcal{M}\ \}\leq
\delta;
\end{equation*}

\item[(b)] $\displaystyle
\lim_{\overset{\tau(e)\rightarrow\infty}{e\in \mathcal{P}\mathcal{F}(%
\mathcal{M},\tau)}}\frac{\Phi(e)}{\tau(e)} =0. $
\end{enumerate}

Let $a\in \mathcal{M}$ be self-adjoint and let $\{e_n\}_{n=1}^\infty$ be a
family of projections in $\mathcal{F}(\mathcal{M},\tau)$. Let $\epsilon>0$.
Then there exist a sequence $\{q_n\}_{n=1}^\infty$ of projections in $%
\mathcal{F}(\mathcal{M},\tau)$ and two sequences $\{d_n\}_{n=1}^\infty$, $%
\{b_n\}_{n=1}^\infty$ of self-adjoint elements in $\mathcal{F}(\mathcal{M}%
,\tau)$ such that, for each $n\ge 1$,

\begin{enumerate}
\item[(i)] $q_iq_n=0$ and $q_ib_n=0$ for all $1\le i<n;$

\item[(ii)] $(\sum_{i=1}^n q_i)e_k=e_k$ for all $1\le k\le n$;

\item[(iii)] $d_n=q_nd_n=d_nq_n$ is diagonal and $\|d_n\|\le \|a\|$;

\item[(iv)] $a=(\sum_{i=1}^n d_i)+ (\sum_{i=1}^n b_i) + (I-\sum_{i=1}^n
q_i)a(I-\sum_{i=1}^n q_i)$;

\item[(v)] $\Phi(b_n)\le \epsilon/2^{n}$.
\end{enumerate}
\end{proposition}

\begin{proof}
We will use an inductive construction to find sequences $\{q_n\}_{n=1}^\infty
$, $\{d_n\}_{n=1}^\infty$ and $\{b_n\}_{n=1}^\infty$ with desired properties.

When $n=1$, applying Proposition \ref{Lemma 3.1.4} to a self-adjoint element
$a$ and a finite rank operator $e_1$ in $\mathcal{M}$, we obtain  a
projection $q_1$ in ${\mathcal{F}(\mathcal{M},\tau)}$ and self-adjoint
elements $d_1$, $b_1$ in ${\mathcal{F}(\mathcal{M},\tau)}$ satisfying

\begin{enumerate}
\item[(ii)] $(I-q_1)e_1=0$;

\item[(iii)] $d_1=d_1q_1=q_1d_1$ is diagonal and $\|d_1\|\le \|a\|$;

\item[(iv)] $a=d_1+b_1+(I-q_1)a(I-q_1)$;

\item[(v)] $\Phi(b_1)\leq \epsilon/2$.
\end{enumerate}

Assume $n$ is a positive integer and we have obtained $\{q_i\}_{i=1}^n$, $%
\{d_i\}_{i=1}^n$ and $\{b_i\}_{i=1}^n$ with desired properties. We apply
Proposition \ref{Lemma 3.1.4} to a self-adjoint element $(I-\sum_{i=1}^n
q_i)a(I-\sum_{i=1}^n q_i)$ and a finite rank operator $(I-\sum_{i=1}^n
q_i)e_{n+1}(I-\sum_{i=1}^n q_i)$ in $(I-\sum_{i=1}^n q_i)\mathcal{M}%
(I-\sum_{i=1}^n q_i)$ and obtain  a projection $q_{n+1}$ in ${\mathcal{F}(%
\mathcal{M},\tau)}$ and self-adjoint elements $d_{n+1}$, $b_{n+1}$ in ${%
\mathcal{F}(\mathcal{M},\tau)}$ satisfying

\begin{enumerate}
\item[(i$_1$)] $q_{n+1}\le (I-\sum_{i=1}^n q_i) $ and $R(b_{n+1})\le
(I-\sum_{i=1}^n q_i)$;

\item[(ii$_1$)] $((I-\sum_{i=1}^n q_i)-q_{n+1})e_{n+1}(I-\sum_{i=1}^n q_i)=0$%
;

\item[(iii)] $d_{n+1}=d_{n+1}q_{n+1}=q_{n+1}d_{n+1}$ is diagonal and $%
\|d_{n+1}\|\le \|a\|$;

\item[(iv$_1$)] $(I-\sum_{i=1}^n q_i)a(I-\sum_{i=1}^n
q_i)=d_{n+1}+b_{n+1}+(I-\sum_{i=1}^{n+1} q_i )a(I-\sum_{i=1}^{n+1} q_i )$;

\item[(v)] $\Phi(b_{n+1})\leq \epsilon/2^{n+1}$.
\end{enumerate}

We will further verify that (i), (ii) and (iv) are also true. Actually, from
(i$_1$), it induces that $q_iq_{n+1}=0$ and  $q_ib_{n+1}=0$ for all $1\le
i<n+1$, whence (i) is true. As $(I-\sum_{i=1}^n q_i)-q_{n+1}$ is a
sub-projection of $I-\sum_{i=1}^n q_i$, from  (ii$_1$),  one concludes that $%
(I-\sum_{i=1}^{n+1} q_i)e_{n+1}=0$. Furthermore, for $1\le i\le n$,  $%
(I-\sum_{i=1}^{n+1} q_i)e_{i}= (I-\sum_{i=1}^{n+1} q_i)(I-\sum_{i=1}^n q_i)
e_{i}=0$ by the induction hypothesis. Thus (ii) is also true.  From (iv$_1$)
and induction hypothesis,
\begin{equation*}
\begin{aligned}
a = \sum_{i=1}^n d_i +  \sum_{i=1}^n b_i  + (I-\sum_{i=1}^n q_i)a(I-\sum_{i=1}^n q_i)  = \sum_{i=1}^{n+1} d_i +  \sum_{i=1}^{n+1} b_i  + (I-\sum_{i=1}^{n+1} q_i)a(I-\sum_{i=1}^{n+1} q_i).\end{aligned}
\end{equation*}
Hence (iv) is satisfied.

This completes the inductive construction of the sequences and finishes the
proof of the result.
\end{proof}

\subsection{Kuroda Theorem for a self-adjoint operator in semifinite von
Neumann algebras}

\ \newline

In the following theorem, we prove that a self-adjoint element in $\mathcal{M%
} $ can be written as a diagonal element up to a small perturbation in ${%
\mathcal{K}^{0}_{\Phi}(\mathcal{M},\tau)^{}}$ when $\Phi$ satisfies a
natural condition.

\begin{proposition}
\label{theorem3.2.1} Let $\mathcal{M}$ be a countably decomposable,
properly infinite von Neumann algebra with a faithful  normal
semifinite tracial
weight $\tau$ and let $\mathcal{K}_{\Phi}(\mathcal{M},\tau)$ be a norm ideal of $%
(\mathcal{M},\tau)$. Assume that

\begin{enumerate}
\item[(a)] there exists a positive number $\delta$ such that
\begin{equation*}
\sup\{\tau(e):e \mbox{ is a minimal projection in } \mathcal{M}\ \}\leq
\delta ;
\end{equation*}

\item[(b)] $\displaystyle
\lim_{\overset{\tau(e)\rightarrow\infty}{e\in \mathcal{P}\mathcal{F}(%
\mathcal{M},\tau)}}\frac{\Phi(e)}{\tau(e)} =0. $
\end{enumerate}

Let $a\in \mathcal{M}$ be self-adjoint. Then for every $\epsilon>0$, there
exists a diagonal operator $d$ in $\mathcal{M}$ such that:

\begin{enumerate}
\item[(i)] $a-d\in{\mathcal{K}^{0}_{\Phi}(\mathcal{M},\tau)^{}}$;

\item[(ii)] $\Phi(a-d)\le\epsilon$ and $\|d\|\le \|a\|$.
\end{enumerate}
\end{proposition}

\begin{proof}
By the fact that $\mathcal{M}$ is a countably decomposable,
properly infinite von Neumann algebra with a faithful normal
semifinite tracial weight $\tau$ and the Assumption (a), one
concludes that there exists a
sequence $\{e_{n}\}^{\infty}_{n=1}$ of orthogonal projections  in $\mathcal{F%
}(\mathcal{M},\tau)$ such that
\begin{equation*}
\text{$\sum^{\infty}_{n=1}e_{n}=I$ \qquad and \qquad $\tau(e_{n})\leq \delta$%
.}
\end{equation*}

By Proposition \ref{prop3.1.5}, we obtain a sequence $\{q_n\}_{n=1}^\infty$
of projections in $\mathcal{F}(\mathcal{M},\tau)$ and two sequences $%
\{d_n\}_{n=1}^\infty$ and $\{b_n\}_{n=1}^\infty$ of self-adjoint elements in
$\mathcal{F}(\mathcal{M},\tau)$ such that, for each $n\ge 1$,

\begin{enumerate}
\item[(i)] $q_iq_n=0$ and $q_ib_n=0$ for all $1\le i<n;$

\item[(ii)] $(\sum_{i=1}^n q_i)e_k=e_k$ for all $1\le k\le n$;

\item[(iii)] $d_n=q_nd_n=d_nq_n$ is diagonal and $\|d_n\|\le \|x\|$;

\item[(iv)] $a=(\sum_{i=1}^n d_i)+ (\sum_{i=1}^n b_i) + (I-\sum_{i=1}^n
q_i)a(I-\sum_{i=1}^n q_i)$;

\item[(v)] $\Phi(b_n)\le \epsilon/2^{n}$.
\end{enumerate}

From (i), it follows that $\{q_n\}_{n=1}^\infty$ is a sequence of orthogonal
projections in $\mathcal{F}(\mathcal{M},\tau)$. Thus $\sum_{n=1}^\infty q_n$
converges to a projection $q \in\mathcal{M}$ in weak$^*$ topology. By (ii)
and the fact that $\sum^{\infty}_{n=1}e_{n}=I$, one knows that $q=I$.

From (iii) and the fact that $\{q_n\}_{n=1}^\infty$ is a sequence of
orthogonal projections, we let  $d=\sum_{n=1}^\infty d_n$ be a diagonal
operator in $\mathcal{M}$ satisfying $\|d\|\le \|a\|$. From (v) and Lemma %
\ref{prelim_lemma1}, it follows that $\sum_{n=1}^\infty b_n$ converges to an
element $b\in \mathcal{K}_{\Phi}^0(\mathcal{M},\tau)$ in $\Phi$-norm and $%
\Phi(b)\le \epsilon.$

Now we need only to show that $a=d+b$. In viewing of $\sum_{n=1}^\infty q_n=I
$, it suffice to show that $q_n(a-d-b)=0$ for each $n\ge 1$, which follows
directly from (i), (iii), (iv) and definitions of $d$, $b$. This ends the
proof of the proposition.
\end{proof}

Now we are ready to present the main result of this section, which gives an
extension of Kuroda's result (see \cite{Kuroda}) in a type I$_\infty$ factor.

\begin{theorem}
\label{thm2.3.3} Let $\mathcal{M}$ be a countably decomposable,
properly infinite von Neumann algebra with a faithful  normal
semifinite tracial
weight $\tau$ and let $\mathcal{K}_{\Phi}(\mathcal{M},\tau)$ be a norm ideal of $%
(\mathcal{M},\tau)$. Assume that
\begin{equation*}
\displaystyle
\lim_{\overset{\tau(e)\rightarrow\infty}{e\in \mathcal{P}\mathcal{F}(%
\mathcal{M},\tau)}}\frac{\Phi(e)}{\tau(e)} =0.
\end{equation*}
Let $a\in \mathcal{M} $ be a self-adjoint element. Then for every $\epsilon>0
$, there exists a diagonal operator $d$ in $\mathcal{M}$ such that:

\begin{enumerate}
\item[(i)] $a-d\in{\mathcal{K}^{0}_{\Phi}(\mathcal{M},\tau)^{}}$;

\item[(ii)] $\Phi(a-d)\le\epsilon$.
\end{enumerate}
\end{theorem}

\begin{proof}
Note that $\mathcal{M}$ is a countably decomposable, properly
infinite von Neumann algebra  with a faithful  normal  semifinite
tracial weight $\tau$. There  exists a countable family
$\{q_n\}_{n\ge 0}$ of orthogonal
projections in the center of  $\mathcal{M}$ such that $\sum_{n } q_n=I$, $q_0%
\mathcal{M}$ is diffused  (or $0$) and $q_n\mathcal{M}$ is a type I$_\infty$
factor (or $0$) for each  $n\ge 1$.

For each $n\ge 0$, from Proposition \ref{theorem3.2.1}, there exists a
diagonal operator $d_n$ in $q_n\mathcal{M}$ such that

\begin{enumerate}
\item[(i$_1$)] $q_na-d_n\in{\mathcal{K}^{0}_{\Phi}(\mathcal{M},\tau)^{}}$;

\item[(ii$_1$)] $\Phi(q_na-d_n)\le\epsilon/2^{n+1}$ and $\|d_n\|\le
\|q_na\|\le \|a\|$.
\end{enumerate}

Let $d=\sum_n d_n$. Then from Lemma \ref{prelim_lemma1} and the choice of $%
\{d_n\}$ we deduce that $d$ is a diagonal operator in $\mathcal{M}$
satisfying

\begin{enumerate}
\item[(i)] $a-d\in{\mathcal{K}^{0}_{\Phi}(\mathcal{M},\tau)^{}}$;

\item[(ii)] $\Phi(a-d)\le\epsilon$.
\end{enumerate}

This completes the proof of the theorem.
\end{proof}

The next result follows directly from preceding theorem and Definition \ref%
{prelim_def2}. The result has been obtained in \cite{Kaftal} when $\mathcal{M%
}$ is a semifinite factor.

\begin{corollary}
Let $\mathcal{M}$ be a countably decomposable,  properly infinite
von Neumann algebra  with a faithful  normal  semifinite tracial
weight $\tau$.  Let $r>
1$ and $a\in \mathcal{M} $ be a self-adjoint element. Then for every $%
\epsilon>0$, there exists a diagonal operator $d$ in $\mathcal{M}$ such that

\begin{enumerate}
\item[(i)] $a-d\in L^r(\mathcal{M},\tau)\cap \mathcal{M}$;

\item[(ii)] $\max\{\|a-d\|, \|a-d\|_r\}\le\epsilon$.
\end{enumerate}
\end{corollary}

From Lemma II.2.8 in \cite{Davidson} and Theorem \ref{thm2.3.3}, we can
quickly conclude the following result. The result has been obtained in \cite%
{Zaido} (see also in \cite{Ake}) when $\mathcal{M}$ is a semifinite factor.

\begin{corollary}
Let $\mathcal{M}$ be a countably decomposable,  properly infinite
von Neumann
algebra  with a faithful  normal  semifinite tracial weight $\tau$. Let  $%
x_1,\ldots, x_r$ be a commuting family of self-adjoint operators  in $%
\mathcal{M}$. Then for every $\epsilon>0$, there exists a commuting family
of diagonal operators $d_1,\ldots, d_r$ in $\mathcal{M}$ such that

\begin{enumerate}
\item[(i)] $a_i-d_i\in{\mathcal{K}(\mathcal{M},\tau)^{}}$ for all $1\le i\le
r$;

\item[(ii)] $\|a_i-d_i\|\le\epsilon$ for all $1\le i\le r$.
\end{enumerate}
\end{corollary}

\section{$\Phi$-well-behaved Sets in Semifinite von Neumann Algebras}

In this section,   we let $\mathcal{M}$ be a countably decomposable,
properly infinite von Neumann algebra  with a faithful  normal
semifinite tracial weight $\tau$.
Let  $\mathcal{P}\mathcal{F}(\mathcal{M},\tau)$, $\mathcal{F}(\mathcal{M}%
,\tau)$ and   $\mathcal{K}(\mathcal{M},\tau)$ be  the sets of finite rank
projections, finite rank operators, and compact operators  respectively, in $%
(\mathcal{M},\tau)$.

For each $x$ in $\mathcal{M}$, let $R(x)$ be the range projection of $x$ in $%
\mathcal{M}$. Note that, \emph{for $x,y\in (\mathcal{M})_1^+$, $xy=x=yx$
if and only if $x\le R(x)\le y$.}

The concept of quasi-central approximate units plays an important role in
the study of $C^*$-algebras. For von Neumann algebras, we introduce a
concept of \emph{$\Phi$-well-behaved sets}, following a notation from
Voiculescu in \cite{Voi}, which will be a replacement of quasi-central
approximate units for $C^*$-algebras.

\subsection{Definition of $\Phi$-well-behaved sets}

\begin{definition}
\label{def4.1.1}  Let $\mathcal{K}_{\Phi}(\mathcal{M},\tau)$ be a norm ideal
of $(\mathcal{M},\tau)$. Let $\mathcal{B}$ be a countably generated $*$%
-subalgebra of $\mathcal{M}$ with an identity $I_{\mathcal{B}}$. We call
that $\mathcal{B}$ is $\Phi$-well-behaved in $\mathcal{M}$ if there exists a
sequence $\{f_n\}_{n=1}^\infty$ of operators in $\mathcal{F}(\mathcal{M}%
,\tau)$ such that

\begin{enumerate}
\item[(i)] $0\le f_1\le R(f_1)\le f_2 \le R(f_2)\le \cdots \le f_n\le
R(f_n)\le \cdots \le I_{\mathcal{B}}$;

\item[(ii)] As $n$ goes to infinity, $f_n$ converges to $I_{\mathcal{B}}$ in
weak$^*$-topology of $\mathcal{M}$;

\item[(iii)] $\displaystyle lim_n \Phi(bf_n-f_nb)=0$ for each $b\in \mathcal{%
B}.$
\end{enumerate}
\end{definition}

\begin{example}
Examples of $\Phi$-well-behaved sets will be given in next two subsections.
\end{example}

A direct computation shows the following result. (See Lemma 2.2 in \cite{DavidsonNormal2} for a proof.)

\begin{lemma}
\label{lemma4.1.2} Let $\mathcal{K}_{\Phi}(\mathcal{M},\tau)$ be a norm
ideal of $(\mathcal{M},\tau)$. Let $0\le x\le I_{\mathcal M}$ be a positive operator in $%
\mathcal{F}(\mathcal{M},\tau)$. Then, for every $y\in \mathcal{M}$,
\begin{equation*}
\Phi(\sin(\frac {\pi x}2)y-y\sin(\frac {\pi x}2))\le 4\Phi(xy-yx)
\end{equation*}
and
\begin{equation*}
\Phi(\cos(\frac {\pi x}2)y-y\cos(\frac {\pi x}2))\le 4\Phi(xy-yx).
\end{equation*}
\end{lemma}

The following lemma is an extension of Lemma 2.2 in \cite{Voi}.

\begin{lemma}
\label{lemma2.4}  Let $\mathcal{K}_{\Phi}(\mathcal{M},\tau)$ be a norm ideal
of $(\mathcal{M},\tau)$. Let $\mathcal{B}$ be a countably generated $*$%
-subalgebra of $\mathcal{M}$ with an identity $I_{\mathcal{B}}$. If $%
\mathcal{B}$ is $\Phi$-well-behaved in $\mathcal{M}$, then for every finite
subset $F$ of $\mathcal{B}$ and $\epsilon>0$ there exists a sequence $%
\{e_n\}_{n=1}^\infty$ of positive operators in $\mathcal{F}(\mathcal{M},\tau)
$ such that

\begin{enumerate}
\item[(i)] $\displaystyle \sum_{n=1}^\infty e_n^2=I_{\mathcal{B}}$,

\item[(ii)] $\displaystyle \sum_{n=1}^\infty \Phi( b e_n-e_nb)\le \epsilon$
for all  $b\in F$,

\item[(iii)] $\displaystyle \sum_{n=1}^\infty \Phi( b e_n-e_nb)<\infty $ for
all  $b\in \mathcal{B}$.
\end{enumerate}
\end{lemma}

\begin{proof}
Since $\mathcal{B}$ is a countably generated $*$-algebra and $F$ is a
finite subset of $\mathcal{B}$, we might assume that $\{b_k\}_{k=1}^\infty$
is a base of $\mathcal{B}$ (as a linear space) and $F=\{b_1,\ldots, b_m\}$.
From the definition of $\Phi$-well-behaved sets, we assume that $%
\{p_n\}_{n=1}^\infty$ is a sequence of finite rank operators satisfying

\begin{enumerate}
\item[(a)] $0\le p_1\le R(p_1)\le p_2 \le R(p_2)\le \cdots \le p_n\le
R(p_n)\le \cdots \le I_{\mathcal{B}}$;

\item[(b)] As $n$ goes to infinity, $p_n$ converges to $I_{\mathcal{B}}$ in
weak$^*$-topology of $\mathcal{M}$.

\item[(c)] $\displaystyle  \Phi(b_kp_n-p_nb_k)\le \frac {\epsilon}{2^{n+4}}$
for each $1\le k\le m+n.$
\end{enumerate}

Let $p_0=0$. Define
\begin{equation*}
f_n=\sin^2(\frac {\pi p_n} 2), \qquad \text{ for each } n\ge 0
\end{equation*}
and
\begin{equation*}
e_{n}=(f_{n}-f_{n-1})^{1/2}, \qquad \text{ for each } n\ge 1.
\end{equation*}
From (a) and (b), we know that $f_n$ is an increasing sequence that
converges to $I_{\mathcal{B}}$ in weak$^*$-topology. This means that
\begin{equation*}
\sum_{n=1}^\infty e_n^2 =I_{\mathcal{B}}.
\end{equation*}
Again from (a), we know that $p_np_{n-1}=p_{n-1}=p_{n-1}p_n$ for $n\ge 1$,
whence
\begin{equation*}
e_{n}= (f_{n}-f_{n-1})^{1/2}=(\sin^2(\frac {\pi p_{n}} 2)-\sin^2(\frac {\pi
p_{n-1}} 2))^{1/2}= \sin(\frac {\pi p_{n}} 2)\cos(\frac {\pi p_{n-1 }} 2).
\end{equation*}
Therefore, from Lemma \ref{lemma4.1.2} and (c),
\begin{equation*}
\begin{aligned}
\Phi(e_{n}b_k-b_ke_{n})&=\Phi\left(\sin(\frac {\pi p_{n}} 2)\cos(\frac {\pi p_{n-1}} 2)b_k -b_k\sin(\frac {\pi p_{n}} 2)\cos(\frac {\pi p_{n-1}} 2)\right )\\
  &\le \Phi\left(\sin(\frac {\pi p_{n}} 2)b_k-b_k\sin(\frac {\pi p_{n}} 2)\right )+ \Phi\left (\cos(\frac {\pi p_{n-1}} 2)b_k-b_k\cos(\frac {\pi p_{n-1}} 2)\right )\\
  &\le 4\Phi(p_{n}b_k-b_kp_{n})+4\Phi(p_{n-1}b_k-b_kp_{n-1})\\
  &\le   \frac {4\epsilon}{2^{n+4}}+   \frac {4\epsilon}{2^{n+3}}\le \frac {\epsilon}{2^{n}} \qquad \text{ for } 1\le k\le m+n.
\end{aligned}
\end{equation*}
It follows that, for $k>m$,
\begin{equation*}
\begin{aligned}
\sum_{n=1}^\infty  \Phi( b_k  e_n-e_nb_k) &= \sum_{n=1}^{k-m-1}  \Phi( b_k  e_n-e_nb_k)+ \sum_{n= {k-m}}^\infty  \Phi( b_k  e_n-e_nb_k)\\
& \le \sum_{n=1}^{k-m-1}  \Phi( b_k  e_n-e_nb_k) + \sum_{n= {k-m}}^\infty  \frac {\epsilon}{2^{n}}\\ & <\infty.\end{aligned}
\end{equation*}
When $1\le k\le m$,
\begin{equation*}
\sum_{n=1}^\infty \Phi( b_k e_n-e_nb_k) \le \sum_{n= 1}^\infty \frac {%
\epsilon}{2^{n}}\le \epsilon.
\end{equation*}
This ends the proof of the lemma.
\end{proof}

\subsection{$\|\cdot\|$-well-behaved sets}

\begin{lemma}
\label{lemma4.2.1} Let $\mathcal{W}$ be a von Neumann subalgebra of $%
\mathcal{M}$. For any $\epsilon>0$ and any $0\le x\in \mathcal{K}(\mathcal{M}%
,\tau)\cap \mathcal{W}$, there exists an element $y\in \mathcal{F}(\mathcal{M%
},\tau)\cap \mathcal{W}$ such that $0\le y\le x$ and $\|x-y\|\le \epsilon$.
\end{lemma}

\begin{proof}
Let $\sigma(x)$ be the spectrum of $x$ in $\mathcal M$ and $e_\epsilon$ be the spectral
projection of $x$ onto $\sigma(x) \cap (\epsilon,\infty)$. Thus $%
e_\epsilon\in \mathcal{W}$. Note that $x\in \mathcal{K}(\mathcal{M},\tau)$.
We know that $e_\epsilon \in \mathcal{F}(\mathcal{M},\tau)$. Thus $y=x\cdot
e_\epsilon$ is an element in $\mathcal{F}(\mathcal{M},\tau)\cap \mathcal{W}$
such that $0\le y\le x$ and $\|x-y\|\le \epsilon$.
\end{proof}

Recall that $\mathcal{K}(\mathcal{M},\tau)$ is a norm ideal of $(\mathcal{M},\tau)
$ with respect to $\|\cdot \|$-norm.

\begin{proposition}
\label{lemma4.2.2} Let $\mathcal{B}$ be a countably generated $*$-subalgebra
of $\mathcal{M}$ with an identity $I_{\mathcal{B}}$. Then $\mathcal{B}$ is $%
\|\cdot \|$-well-behaved in $\mathcal{M}$.
\end{proposition}

\begin{proof}
We should define a partial order ``$\prec$'' on $(\mathcal{M})_1^+$, the
unit ball of positive operators in $\mathcal{M}$, as follows: $x\prec y$ in $%
(\mathcal{M})_1^+$ if and only if $0\le x\le R(x)\le y\le I_{\mathcal{M}}$,
where $R(x)$ is the range projection of $x$ in $\mathcal{M}$.

Replacing $\mathcal{M}$ by $I_{\mathcal{B}}\mathcal{M }I_{\mathcal{B}}$ if
necessary, we will assume that $I_{\mathcal{B}}=I_{\mathcal{M}}$. Note that $%
\mathcal{M}$ is a countably decomposable von Neumann algebra with a normal,
faithful and semifinite tracial weight $\tau$. There exists a sequence $%
\{q_n\}_{n=1}^\infty$ of orthogonal finite rank projections in $\mathcal{M}$
such that $\sum_{n=1}^\infty q_n=I_{\mathcal{M}}$.

Let $\mathcal{A}$ be the unital separable $C^{\ast }$-subalgebra generated by $\mathcal{B}$ and the projections $\{q_{n}\}_{n=1}^{\infty }$ in $\mathcal{M}$. Denote by $\mathcal{I}$ the closed two sided ideal $\mathcal{A}\cap \mathcal{K}(\mathcal{M},\tau)$ of $\mathcal{A}$.
By Corollary 1.5.11 in \cite{Lin}, there exists an
approximate identity $\{e_{n}\}_{n=1}^{\infty }$ in $\mathcal{I}$ such that $%
e_{n}e_{n+1}=e_{n}=e_{n+1}e_{n}$ for all $n\geq 1$. Because $%
\{q_{n}\}_{n=1}^{\infty }\subseteq \mathcal{I}$, we know that $e_{n}$
converges to $I_{\mathcal{M}}$ in weak$^{\ast }$-topology and $e_{1}\prec
e_{2}\prec \cdots \prec e_{n}\prec e_{n+1}\prec \cdots $.

Since $\mathcal{B}$ is a countably generated $*$-subalgebra of $\mathcal{M}$%
, there exists a sequence $\{b_n\}_{n=1}^\infty$ in $\mathcal{B}$ that forms
a base for $\mathcal{B}$ (as a linear space). In the proof of Theorem 1 in \cite{Ave}, replacing the natural order by
``$\prec$'' on positive operators,
we can find a sequence $\{f_n\}_{n=1}^\infty$ in $\mathcal{I}$ such that

\begin{enumerate}
\item[(i)] $0\prec f_1\prec f_2\prec \cdots \prec f_n\prec f_{n+1} \prec
\cdots \prec I_{\mathcal{M}}$;

\item[(ii)] As $n$ approaches infinity, $f_n$ goes to $I_{\mathcal{M}}$ in
weak$^*$-topology;

\item[(iii)] $\lim_n \|f_n b_k-b_kf_n\|=0$ for all $k\ge 1$.
\end{enumerate}

Let 
$\mathcal{W}$ be the commutative von Neumann algebra
generated by $\{f_n\}_{n=1}^\infty$ in $\mathcal{M}$. 

\

\noindent\emph{Claim \ref{lemma4.2.2}.1. \ There exists a sequence $%
\{p_n\}_{n=1}^\infty$ in $\mathcal{F}(\mathcal{M},\tau)$ such that  $%
\|f_n-p_n\|\le \frac 1 {2^n}$ for all $n\ge 1$ and $0\prec p_1\prec p_2\prec
\cdots \prec p_n\prec p_{n+1} \prec \cdots\prec I_{\mathcal{M}}$.}

\

\noindent {Proof of the claim: } We will use an induction process to find a
sequence with desired properties.

When $n=1$, by Lemma \ref{lemma4.2.1}, there exists a $p_1\in \mathcal{F}(%
\mathcal{M},\tau)\cap \mathcal{W}$ such that $0\le p_1\le f_1$ and $%
\|p_1-f_1\|\le \frac 1 2.$

Let $n$ be a positive integer and assume we have chosen $\{p_1,\ldots, p_n\}$
in $\mathcal{F}(\mathcal{M},\tau)\cap \mathcal{W}$ satisfying

\begin{enumerate}
\item[(a)] $p_i\in \mathcal{F}(\mathcal{M},\tau)\cap \mathcal{W}$, $0\le
p_i\le f_i$ and $\|f_i-p_i\|\le 1/ {2^i}$ for $1\le i\le n$;

\item[(b)] $p_1\prec p_2\prec \cdots \prec p_n$.
\end{enumerate}

From the fact that $f_n\prec f_{n+1}$ and $0\le p_n\le f_n$, it follows that
$p_n\prec f_{n+1}$. Hence $f_{n+1}=f_{n+1}(I_{\mathcal{M}}-R(p_n))+R(p_n)$,
where $R(p_n)$ is the range projection of $p_n$ in $\mathcal{W}$. By Lemma %
\ref{lemma4.2.1}, there exists a $\hat p_{n+1} \in \mathcal{F}(\mathcal{M}%
,\tau)\cap \mathcal{W}$ such that $0\le \hat p_{n+1}\le f_{n+1}(I_{\mathcal{M%
}}-R(p_n))$ and $\|\hat p_{n+1}- f_{n+1}(I_{\mathcal{M}}-R(p_n))\|\le \frac
1 {2^{n+1}}.$ Let $p_{n+1}= \hat p_{n+1}+R(p_n)$. Then $p_{n+1}\in \mathcal{F%
}(\mathcal{M},\tau)\cap \mathcal{W}$ satisfying $0\le p_{n+1}\le f_{n+1}$, $%
\|p_{n+1}-f_{n+1}\|\le \frac 1 {2^{n+1}}$ and $p_n\prec p_{n+1}$. This ends
the construction of $\{p_n\}_{n=1}^\infty$ and the proof of the claim.

\

\noindent (\emph{Continue the proof of the result}) From Claim \ref%
{lemma4.2.2}.1, $\{p_n\}_{n=1}^\infty$ is a sequence in $\mathcal{F}(%
\mathcal{M},\tau)$ such that

\begin{enumerate}
\item[(i)] $0\le p_1\le R(p_1)\le p_2 \le R(p_2)\le \cdots \le p_n\le
R(p_n)\le \cdots \le I_{\mathcal{B}}$;

\item[(ii)] As $n$ goes to infinity, $p_n$ converges to $I_{\mathcal{B}}$ in
weak$^*$-topology of $\mathcal{M}$;

\item[(iii)] $\displaystyle lim_n \Vert b_kp_n-p_nb_k\Vert=0$ for each $k\ge
1.$
\end{enumerate}

Thus $\mathcal{B}$ is $\|\cdot \|$-well-behaved in $\mathcal{M}$.
\end{proof}

\begin{lemma}
{\label{Averson'sLemma}} Suppose $\mathcal{M}$ is a countably
decomposable, properly infinite von Neumann algebra  with a faithful
normal  semifinite tracial weight $\tau $ and $\mathcal{A}$ is a
separable $C^*$-subalgebra of $\mathcal{M}$ with an identity
$I_{\mathcal{A}}$. For any finite set $F\subseteq \mathcal{A}$ and
any positive number $\epsilon$, there exists a sequence
$\{e_n\}_{n=1}^\infty $ of positive operators in
$\mathcal{F}(\mathcal{M},\tau)$ such that

\begin{enumerate}
\item[(i)] $\displaystyle \sum_{n=1}^\infty e_n^2=I_{\mathcal{A}}$;

\item[(ii)] $\displaystyle x-\sum_{n=1}^\infty e_nxe_n \in \mathcal{K}(%
\mathcal{M},\tau)$ for all $x$ in $\mathcal{A}$; and

\item[(iii)] $\displaystyle \|x-\sum_{n=1}^\infty e_nxe_n \|\le \epsilon, $
for all $x$ in $F$.
\end{enumerate}
\end{lemma}

\begin{proof}
Let $\mathcal{B}$ be a norm dense, countably generated $*$-subalgebra of $%
\mathcal{A}$ such that $I_{\mathcal{A}}\in\mathcal{B}$ and $F\subseteq
\mathcal{B}$. Now the result follows directly from Lemma \ref{lemma2.4} and
Proposition \ref{lemma4.2.2}.
\end{proof}

\subsection{$(\max\{\|\cdot\|, \|\cdot\|_r\})$-well-behaved sets}

\ \newline

The main result of this subsection is the following proposition.

\begin{proposition}
\label{prop4.3.1} Let $r\ge 2$. Suppose $\mathcal{M}$ is a countably
decomposable,  properly infinite   von Neumann algebra  with a
faithful  normal  semifinite
tracial weight $\tau$ and $\mathcal{K}_r(\mathcal{M},\tau)=L^r(\mathcal{M}%
,\tau)\cap \mathcal{M}$ is a norm ideal of $(\mathcal{M},\tau)$ equipped
with a norm $\Phi$ satisfying
\begin{equation*}
\Phi(x) = \max\{\|x\|, \|x\|_r\}, \qquad \forall \ x\in L^r(\mathcal{M}%
,\tau)\cap \mathcal{M}.
\end{equation*}
Let $\mathcal{B}$ be a unital $*$-subalgebra generated by a commuting family
of self-adjoint elements $x_1, \ldots, x_r$ in $\mathcal{M}$ with an
identity $I_{\mathcal{B}}$. Then $\mathcal{B}$ is $\Phi$-well-behaved in $%
\mathcal{M}$.
\end{proposition}

Before presenting the proof of Proposition \ref{prop4.3.1}, we need
to show some lemmas first.



\begin{lemma}
\label{lemma4.3.1}Let $r\ge 2$ be a positive integer. Let $x_1,\ldots, x_r$
be a commuting family of self-adjoint elements in $\mathcal{M}$. For every
projection $q$ in $\mathcal{M}$ with $\tau(q)<\infty$ and $\epsilon>0$,
there exists an operator $e$ in $\mathcal{M}$ such that

\begin{enumerate}
\item[(i)] $q\le e\le I_{\mathcal M} $ and $\tau(R(e))<\infty$;

\item[(ii)] $\displaystyle \max_{1\le i\le r} \|x_ie-ex_i\|_r\le\epsilon.$
\end{enumerate}
\end{lemma}

\begin{proof}
We might assume that $\|x_i\|\le 1$ for all $1\le i\le r$. For each $k\ge 0$%
, define
\begin{equation*}
p_k=\vee \{R(x_1^{i_1}\cdots x_r^{i_r} q) \ : \ 0\le i_1,\ldots, i_r \text {
satisfy } i_1+ \cdots + i_r\le k\}.
\end{equation*}
Thus
\begin{equation*}
q=p_0\le p_1\le \cdots \le p_k\le \cdots
\end{equation*}
is an increasing sequence of finite rank projections. Define
\begin{equation*}
f_k=p_k-p_{k-1} \qquad \text{ for all } k\ge 1.
\end{equation*}
From the definitions of $p_k$ and $f_k$, it follows that

\begin{enumerate}
\item[(a)] $R(x_ip_k)\le p_{k+1}$ for $1\le i\le r$ and $k\ge 1$; and

\item[(b)] $(I_{\mathcal M}-p_{k+1})x_ip_k=0$ for $1\le i\le r$ and $k\ge 1$;

\item[(c)] $\displaystyle \tau(f_k)\le \sum_{i_1+\cdots +i_r=k}
\tau(R(x_1^{i_1}\cdots x_r^{i_r} q)) \le \frac{(k+r-1) !}{k !(r-1) !}\tau(q).
$
\end{enumerate}

Moreover, for all $1\le i\le r$ and $k\ge 1$,
\begin{align}
x_ip_k-p_kx_i &= p_{k+1}x_ip_k-p_kx_ip_{k+1}  \tag{by (b)} \\
&=(p_{k+1}-p_k)x_i(p_k-p_{k-1})-(p_k-p_{k-1})x_i(p_{k+1}-p_k)  \notag \\
& \quad +(p_kx_ip_k+(p_{k+1}
-p_k)x_ip_{k-1})-(p_kx_ip_k+p_{k-1}x_i(p_{k+1}-p_k))  \notag \\
&=(p_{k+1}-p_k)x_i(p_k-p_{k-1})-(p_k-p_{k-1})x_i(p_{k+1}-p_k)  \tag{by (b)}
\\
&= f_{k+1}x_if_k-f_kx_if_{k+1}.  \label{equ4.2}
\end{align}
Now it follows from (c) that
\begin{equation}
\tau(R(x_ip_k-p_kx_i))\le 2 \tau(f_k)\le 2\tau(q)\cdot \frac{(k+r-1) !}{k
!(r-1) !}.  \label{equ4.1}
\end{equation}
Furthermore, from equation (\ref{equ4.2}) we conclude that, for each $1\le
i\le r$,
\begin{equation*}
\{ x_ip_{_{2k }}-p_{_{2k }}x_i\}_{k=1}^\infty
\end{equation*}
is a family of finite rank operators that satisfies
\begin{equation*}
(x_ip_{_{2k_1}}-p_{_{2k_1}}x_i)(x_ip_{_{2k_2}}-p_{_{2k_2}}x_i)^*=0\qquad
\qquad \text{ for all } k_1\ne k_2.
\end{equation*}
Or,
\begin{equation}
|x_ip_{_{2k_1}}-p_{_{2k_1}}x_i| \cdot |x_ip_{_{2k_2}}-p_{_{2k_2}}x_i
|=0\qquad \qquad \text{ for all } k_1\ne k_2.  \label{equ4.3}
\end{equation}

For each $n\ge 1$, we let
\begin{equation*}
\lambda_n=\sum_{k=n}^{n^2} \frac 1 k \qquad \text{ and }\qquad e_n=\frac 1
{\lambda_n}\sum_{k=n}^{n^2}\frac {p_{_{2k}}}{k}.
\end{equation*}
Then
\begin{equation*}
q\le e_n \le I_{\mathcal M}\qquad \text{ for each } n\ge 1.
\end{equation*}
And
\begin{align}
\|x_ie_n-&e_nx_i\|_r^r =\tau(|x_ie_n-e_nx_i|^{r})  \notag \\
&=\tau(\left |\frac 1 {\lambda_n}\sum_{k=n}^{n^2}\frac {x_ip_{_{2k}}-
p_{_{2k}} x_i}{k}\right|^{r}) = \tau(\frac 1 {\lambda_n^r}\sum_{k=n}^{n^2}%
\frac {\left |x_ip_{_{2k}}- p_{_{2k}}x_i\right|^{r}}{k^r})  \tag{by
\ref{equ4.3}} \\
&\le \frac 1 {\lambda_n^r}\sum_{k=n}^{n^2}\frac {2^r\tau(R(x_ip_{_{2k}}-
p_{_{2k}}x_i))}{k^r} \le \frac 1 {\lambda_n^r}\sum_{k=n}^{n^2}\frac {%
2^r\cdot 2\tau(q)\cdot \frac{(k+r-1) !}{k !(r-1) !}}{k^r} \qquad  \tag{ by
\ref{equ4.1} } \\
&\le \frac 1 {\lambda_n^r}\sum_{k=n}^{n^2}\frac {2^r\cdot 2\tau(q)\cdot
2^{r-1}\cdot {k ^{r-1}}}{k^r\cdot {\ (r-1) !}} \quad  \tag{ when  $ n\ge r $
} \\
&\quad \longrightarrow 0. \qquad  \tag{ as $ n\rightarrow \infty $}
\end{align}
Thus, when $n$ is large enough, there exists an operator $e=e_n$ in $%
\mathcal{M}$ such that

\begin{enumerate}
\item[(i)] $q\le e \le I_{\mathcal M}$ and $\tau(R(e))<\infty$;

\item[(ii)] $\displaystyle \max_{1\le i\le r} \|x_ie-ex_i\|_r\le\epsilon.$
\end{enumerate}
\end{proof}

\begin{lemma}
\label{lemma4.3.2}Let $r\ge 2$ be a positive integer. Let $x_1,\ldots, x_r$
be a commuting family of self-adjoint elements in $\mathcal{M}$. For every
projection $q$ in $\mathcal{M}$ with $\tau(q)<\infty$ and $\epsilon>0$,
there exists an operator $e$ in $\mathcal{M}$ such that

\begin{enumerate}
\item[(i)] $q\le e\le I_{\mathcal M}$ and $\tau(R(e))<\infty$;

\item[(ii)] $\displaystyle \max_{1\le i\le r}\{ \|x_ie-ex_i\|_r,
\|x_ie-ex_i\|\}\le\epsilon.$
\end{enumerate}
\end{lemma}

\begin{proof}
Let $\mathcal{W}$ be an abelian von Neumann subalgebra generated by $I_{\mathcal M}, %
x_1,\ldots, x_r$ in $\mathcal{M}$. For $\epsilon>0$, there exist a
positive
integer $n$, a family of orthogonal projections $p_1,\ldots,p_n$ in $%
\mathcal{W}$ and a family of real numbers $\{\lambda_{i,j}\}_{1\le
i\le r, 1\le j\le n}$ such that $\sum_{j=1}^n p_j=I_{\mathcal M}$
and
\begin{equation*}
\max\{\|x_ip_j-\lambda_{i,j}p_j \| \ : \ 1\le i\le r, 1\le j\le
n\}<\epsilon/2.
\end{equation*}
Let
\begin{equation*}
q_j=\vee\{R(p_jqp_kqp_j) \ : \ 1\le k\le n\} \qquad \text{ for $1\le j\le n$
}.
\end{equation*}
For each $1\le j\le n$, applying Lemma \ref{lemma4.3.1} to a commuting
family of self-adjoint elements $p_j x_1, p_j x_2,\ldots, p_j x_n$ and a
finite rank projection $q_j$ in $p_j\mathcal{M }p_j$, we obtain an operator $%
e_j$ in $\mathcal{M}$ such that

\begin{enumerate}
\item[(i')] $q_j\le e_j\le p_j$ and $\tau(R(e_j))<\infty$;

\item[(ii')] $\displaystyle \max_{1\le i\le r}
\|p_jx_ie_j-e_jp_jx_i\|_r<\epsilon/n.$
\end{enumerate}

Let $e=\sum_{j=1}^n e_j$. Then $\tau(R(e))<\infty$ and
\begin{equation*}
\begin{aligned}
q\le &\vee \{R(p_jq p_k \ : \ 1\le j,k\le n\} \\&=\bigvee_{1\le j\le
n} \left (\vee\{R(p_jq p_kqp_j \ : \ 1\le  k\le n\}  \right ) =
\bigvee_{1\le j\le n} q_j\\&\le \sum_{j=1}^n e_j=e  \\&\le
\sum_{j=1}^n p_j=I_{\mathcal M}. \end{aligned}
\end{equation*}%
Therefore (i) is true. Furthermore,
\begin{equation*}
\begin{aligned}
\|ex_i-x_ie\|&=\max_{1\le j\le n} \|p_j(ex_i-x_ie)\| \le \max_{1\le j\le n} \|e_jp_j x_i-p_jx_ie_j \| \\&\le \max_{1\le j\le n}\|e_j (\lambda_{i,j}p_j)-(\lambda_{i,j}p_j)e_j\|+\epsilon=\epsilon, \end{aligned}
\end{equation*}
and
\begin{equation*}
\|ex_i-x_ie\|_r\le \sum_{1\le j\le n} \|e_jx_i-x_ie_j\|_r \le \epsilon.
\end{equation*}
Thus (ii) is true. This completes the proof of the lemma.
\end{proof}

\begin{lemma}
\label{lemma4.3.3} Let $\mathcal{B}$ be a unital $*$-subalgebra generated by a commuting family
of self-adjoint elements $x_1, \ldots, x_r$ in $\mathcal{M}$ with an
identity $I_{\mathcal{B}}$.  Then there exists a sequence $%
\{e_n\}_{n=1}^\infty$ of positive operators in $\mathcal{F}(\mathcal{M},\tau)
$ such that

\begin{enumerate}
\item[(i)] $0\le e_1\le R(e_1)\le e_2 \le R(e_2)\le \cdots \le e_n\le
R(e_n)\le \cdots \le I_{\mathcal{B}}$;

\item[(ii)] As $n$ goes to infinity, $e_n$ converges to $I_{\mathcal{B}}$ in
weak$^*$-topology of $\mathcal{M}$;

\item[(iii)] $\displaystyle lim_n \|e_nx_i-x_ie_n\| =0$ and $\displaystyle %
lim_n \|e_nx_i-x_ie_n\|_r =0$ for each $1\le i\le r.$
\end{enumerate}
\end{lemma}

\begin{proof}
By considering von Neumann subalgebra $I_{\mathcal{B}}\mathcal{M}I_{\mathcal{%
B}}$ instead, we might assume that $I_{\mathcal{B}}=I_{\mathcal M}
$. Note that $\mathcal{M}$ is a countably decomposable, properly
infinite von Neumann algebra with a faithful normal semifinite
tracial weight $\tau$. There exists a family of
orthogonal projections $\{f_n\}_{n=1}^\infty$ in $\mathcal{F}(\mathcal{M}%
,\tau)$ such that $\sum_{n=1}^\infty f_n=I_{\mathcal M} $.

Using Lemma \ref{lemma4.3.2} repeatedly, we can construct a sequence of $%
\{e_n\}_{n=1}^\infty$ of positive operators in $\mathcal{F}(\mathcal{M},\tau)
$ such that

\begin{enumerate}
\item[(a)] $\left (R(e_{n-1})\vee f_n\right ) \le e_{n} \le I_{\mathcal M} $
(we let $e_0=0$);

\item[(b)] $\max_{1\le i\le r} \{\|x_i e_n-e_nx_i\|, \|x_i e_n-e_nx_i\|_r
\}\le 1/n.$
\end{enumerate}

From (a) and (b), we can directly verify that $\{e_n\}_{n=1}^\infty$ is a
sequence with stated properties.
\end{proof}

Now Proposition \ref{prop4.3.1} is a direct consequence of Lemma \ref%
{lemma4.3.3} and Definition \ref{def4.1.1} and its proof is thus skipped.

\section{Voiculescu's Noncommutative Weyl-von Neumann Theorem in Semifinite
Factors}

The section is devoted to prove a version of Voiculescu's noncommutative
Weyl-Neumann theorem in semifinite factors.

Let $\mathcal{N}$ be a countably decomposable,  properly infinite,
semifinite
factor with a faithful normal semifinite tracial weight $\tau$. Let $%
\mathcal{F}(\mathcal{N},\tau)$ and $\mathcal{K}(\mathcal{N},\tau)$ be the
sets of finite rank operators, and compact operators respectively, in $(%
\mathcal{N},\tau)$.  Recall that $\mathcal{K}(\mathcal{N })$ is the $\|\cdot \|$%
-norm closed ideal generated by finite projections in $\mathcal{N}$.  As $%
\mathcal{N}$ is a countably decomposable, properly infinite, semifinite
factor, we know that $\mathcal{K}(\mathcal{N})=\mathcal{K}(\mathcal{N},\tau)$%
. Thus, in all results proved in this section, $\mathcal{K}(\mathcal{N},\tau)
$ can be replaced by $\mathcal{K}(\mathcal{N})$.

\subsection{Voiculescu Theorem for nuclear $C^*$-subalgebras in semifinite
factors with respect to compact operators}

\ \newline

We will need the following lemma from \cite{Niu}.

\begin{lemma}
{\label{Niu'sLemma}} Let $\mathcal{N}$ be a countably decomposable,
properly infinite,  semifinite factor with a faithful  normal
semifinite tracial
weight $\tau$. Let $\mathcal{F}(\mathcal{N},\tau)$ and $\mathcal{K}(\mathcal{%
N},\tau)$ be the sets of finite rank operators, and compact operators
respectively, in $(\mathcal{N},\tau)$. Let $\mathcal{A}$ be a nuclear
separable $C^*$-subalgebra with an identity $I_{\mathcal{A}}$. Let $q\in
\mathcal{K}(\mathcal{N},\tau)$ be a finite rank projection in $\mathcal{F}(%
\mathcal{N},\tau)$. Suppose $\phi:\mathcal{A}\rightarrow \mathcal{N}$ is a
completely positive mapping satisfying

\begin{enumerate}
\item[(i)] $\displaystyle \phi(I_{\mathcal{A}})$ is a projection in $%
\mathcal{N}$; and

\item[(ii)] $\displaystyle \phi(\mathcal{A}\cap \mathcal{K}(\mathcal{N}%
,\tau))=0$.
\end{enumerate}

Then for any finite set $F\subseteq \mathcal{A}$ and any positive number $%
\epsilon$, there exists a partial isometry $v$ in $\mathcal{N}$ such that $%
qv=0$, $v^*v=\phi(I_{\mathcal{A}}) $, $vv^*\le I_{\mathcal{A}}$ and
\begin{equation*}
\|\phi(x)-v^*xv\|\le \epsilon, \ \ \forall \ x\in F.
\end{equation*}
\end{lemma}

\begin{proof}
The result is a direct consequence of the definition of nuclear $C^*$%
-algebra and Lemma 3.4 in \cite{Niu}.
\end{proof}

Now we are ready to prove the following main result of this subsection.

\begin{theorem}
{\label{Voiculescu'sTheorem}} Let $\mathcal{N}$ be a countably
decomposable,  properly infinite,  semifinite factor acting on a
Hilbert space and let $\tau$ be a
faithful  normal  semifinite tracial weight of $\mathcal N$. Let $\mathcal{F}(\mathcal{N}%
,\tau)$ and $\mathcal{K}(\mathcal{N},\tau)$ be the sets of finite rank
operators, and compact operators respectively, in $(\mathcal{N},\tau)$.

Suppose that $\mathcal{A}$ is a nuclear separable $C^*$-subalgebra of $\mathcal{N}$ with an
identity $I_{\mathcal{A}}$. Suppose $\psi:\mathcal{A}\rightarrow \mathcal{N}$
is a $*$-homomorphism satisfying $\displaystyle \psi(\mathcal{A}\cap
\mathcal{K}(\mathcal{N},\tau))=0$. Then there exists a sequence $%
\{v_j\}_{j=1}^\infty$ of partial isometries in $\mathcal{N}$ such that

\begin{enumerate}
\item[(i)] $\displaystyle  v_jv_j^*\le I_{\mathcal{A}}, v^*_jv_j=\psi(I_{%
\mathcal{A}})$ and $v^*_jv_k=0$ for all $j,k\ge 1$ with $j\ne k$;

\item[(ii)] $\displaystyle \psi(x)-v_j^*xv_j$ is in $\mathcal{K}(\mathcal{N}%
,\tau)$ for all $x\in \mathcal{A}$ and $j\ge 1$;

\item[(iii)] $\displaystyle \lim_j\| \psi(x)-v_j^*xv_j\|=0$ for all $x\in
\mathcal{A}$.

\item[(iv)] $\displaystyle v_j\psi(x)-xv_j$ is in $\mathcal{K}(\mathcal{N}%
,\tau)$ for all $x\in \mathcal{A}$ and $j\ge 1$;

\item[(v)] $\displaystyle \lim_j\|v_j\psi(x)-xv_j\|=0$ for all $x\in
\mathcal{A}$;
\end{enumerate}
\end{theorem}

\begin{proof}
Let $\mathcal{B}=\psi(\mathcal{A})$ be a separable $C^*$-subalgebra of $%
\mathcal{N}$ with an identity $\psi(I_{\mathcal{A}})$. Let $\{x_i\}_{i= 1}^\infty$
be a norm dense subset in the unit ball of $\mathcal{A}$. For each $j\ge 1$,
applying Lemma \ref{Averson'sLemma} to $\mathcal{B}$, there exists a family $%
\{e_{j,n}\}_{n= 1}^\infty$ of positive finite rank operators in $\mathcal{F}(%
\mathcal{N},\tau)$ such that

\begin{enumerate}
\item[(a)] $\displaystyle \sum_{n= 1}^\infty e_{j,n}^2=\psi(I_{\mathcal{A}})$%
, for $j\ge 1$;

\item[(b)] $\displaystyle \psi(x)-\sum_{n= 1}^\infty e_{j,n}\psi(x)e_{j,n}
\in \mathcal{K}(\mathcal{N},\tau)$ for all $x$ in $\mathcal{A}$ and $j\ge 1$%
;

\item[(c)] $\displaystyle \|\psi(x_i)-\sum_{n= 1}^\infty
e_{j,n}\psi(x_i)e_{j,n} \|\le \frac 1 {2^j}, $ for $1\le i\le j$.
\end{enumerate}

Let $p_{j,n}=R(e_{j,n})$ be the range projection of $e_{j,n}$ in $\mathcal{N}
$. It is obvious that $p_{j,n}$ is a finite rank projection in $\mathcal{F}(%
\mathcal{N},\tau)$ such that $p_{j,n}\le \psi(I_{\mathcal{A}})$. Note that
each $p_{j,n}\psi(\cdot) p_{j,n}$ is a completely positive mapping from $%
\mathcal{A}$ into $\mathcal{N}$ such that (1) $p_{j,n}\psi (I_{\mathcal{A}})
p_{j,n}$ is a projection and (2) $p_{j,n}\psi(\mathcal{A}\cap \mathcal{K}(%
\mathcal{N},\tau)) p_{j,n}=0$. Now we are ready to prove the following claim.

\vspace{0.2cm}

\noindent {\ {\em Claim \ref{Voiculescu'sTheorem}.1. \ }}\emph{There exists a
family of partial isometries $\{v_{j,n}\}_{j,n= 1}^\infty$ in $\mathcal{N}$
such that, for each $j,n\ge 1$, }

\begin{enumerate}
\item[(d)] \emph{%
\begin{equation}
\|p_{j,n}\psi(x_i)p_{j,n}-v_{j,n}^*x_iv_{j,n}\|\le \frac 1 {2^{j+n}}, \ \
\text{for $1\le i\le j+n$,}  \label{equ2.1}
\end{equation}
}

\item[(e)] \emph{%
\begin{equation}
v_{j,n}v_{j,n}^*\le I_{\mathcal{A}}, \ \ \ \ v_{j,n}^* v_{j,n} =p_{j,n},
\label{equ2.2}
\end{equation}
}

\item[(f)] \emph{%
\begin{equation}
v^*_{j,n} v_{s,t}=0, \ \ v^*_{j,n} x_i v_{s,t}=0, \text{ and } \
v^*_{j,n}x_i^* v_{s,t}=0 ,  \label{equ2.25}
\end{equation}
when $i, s,t\in \mathbb{N}$ satisfy $1\le i\le j+n, s+t<j+n $, or we have $%
s+t=j+n $ and $s<j$. }
\end{enumerate}

\vspace{0.2cm}

\noindent {Proof of Claim: \ } Define an order ``$\prec$'' on $\mathbb{N}%
\times \mathbb{N}$ as follows:
\begin{equation*}
(s,t)\prec (j,n)
\end{equation*}
if and only if
\begin{equation*}
\text{$s+t<j+n $, or we have $s+t=j+n $ and $s<j$}.
\end{equation*}
Next we will use an inductive process on $\mathbb{N}\times \mathbb{N}$, with
respect to the order $\prec$, to construct a family $\{v_{j,n}\}_{j,n=
1}^\infty$ of partial isometries in $\mathcal{N}$ with stated properties.

When $(j,n)=(1,1)$, applying Lemma \ref{Niu'sLemma} to $p_{1,1}\psi(%
\cdot)p_{1,1}$ and $q_{1,1}=0$, we know that there exists a partial isometry
$v_{1,1}$ in $\mathcal{N}$ such that
\begin{equation*}
\|p_{1,1}\psi(x_i)p_{1,1}-v_{1,1}^*x_iv_{1,1}\|\le \frac 1 {2^{2}}, \
\end{equation*}
and
\begin{equation*}
v_{1,1}v_{1,1}^*\le I_{\mathcal{A}}, \ \ \ \ v_{1,1}^* v_{1,1} =p_{1,1}.
\end{equation*}

Let $j,n\ge 1$ be such that $(1,1)\prec (j,n)$ and assume the family $%
\{v_{j^{\prime },n^{\prime }}\}_{(j^{\prime },n^{\prime })\prec (j,n)}$ of
partial isometries in $\mathcal{N}$ with desired properties has been chosen.
Let
\begin{equation*}
q_{j,n}= \vee \{ R(v_{s,t}), R( x_i v_{s,t}), R(x_i^* v_{s,t}) \ :  \ 1\le
i\le j+n, \text{ and } (s,t)\prec (j,n)\},
\end{equation*}
where $R(x)$ is the range projection of $x$ in $\mathcal{N}$. Thus $q_{j,n}$
is a finite rank projection in $\mathcal{F}(\mathcal{N},\tau)$. Applying
Lemma \ref{Niu'sLemma} to $p_{j,n}\psi(\cdot)p_{j,n}$ and $q_{j,n}$, we can
find a partial isometry $v_{j,n}$ in $\mathcal{N}$ such that
\begin{equation*}
\|p_{j,n}\psi(x_i)p_{j,n}-v_{j,n}^*x_iv_{j,n}\|\le \frac 1 {2^{j+n}}, \
\forall \ 1\le i\le j+n,
\end{equation*}
\begin{equation*}
v_{j,n}v_{j,n}^*\le I_{\mathcal{A}}, \ \ \ \ v_{j,n}^* v_{j,n} =p_{j,n},
\end{equation*}
and
\begin{equation*}
v^*_{j,n} v_{s,t}=0, \ \ v^*_{j,n} x_i v_{s,t}=0, \ \text{ and } \
v^*_{j,n}x_i^* v_{s,t}=0, \ \forall 1\le i\le j+n, \text { and } (s,t)\prec
(j,n).
\end{equation*}
This finishes the inductive construction of $\{v_{j,n}\}_{j,n= 1}^\infty$
with stated properties. And the proof of the Claim is completed.

\vspace{0.2cm}

\noindent (\emph{Continue the proof of the theorem}:)
From inequality (\ref{equ2.1}), we have, for all $1\le i\le j$,
\begin{equation}
\sum_{n=1}^\infty\|p_{j,n}\psi(x_i)p_{j,n}-v_{j,n}^*x_iv_{j,n}\|\le \frac 1
{2^{j}},  \label{equ2.3}
\end{equation}
and, for all $i, j\ge 1$
\begin{equation}
\sum_{n=1}^\infty\|p_{j,n}\psi(x_i)p_{j,n}-v_{j,n}^*x_iv_{j,n}\|<\infty.
\label{equ2.4}
\end{equation}
From equation (\ref{equ2.25}),
\begin{equation}
v^*_{s,t}v_{j,n}=0 \qquad \text { if $(s,t)\ne (j,n)$.}  \label{equ2.0}
\end{equation}
Define, for each $j\ge 1$,
\begin{equation}
v_j= \sum_{n=1}^\infty v_{j,n}e_{j,n}, \qquad  \tag{\text{convergence is
in
strong operator topology}}
\end{equation}
which is a partial isometry in $\mathcal{N}$. We will verify that $%
\{v_j\}_{j= 1}^\infty$ satisfies the condition (i), (ii), (iii), (iv) and
(v).

(i): From equation (\ref{equ2.0}) and (\ref{equ2.2}), we have for $j,k\ge 1$
with $j\ne k$
\begin{equation}
v_j^*v_k=0, \ \ \ v_j^*v_j=\sum_{n=1}^\infty e_{j,n} p_{j,n} e_{j,n}
=\sum_{n=1}^\infty e_{j,n}^2=\psi(I_{\mathcal{A}}) \ \text { and }
v_jv_j^*\le I_{\mathcal{A}} .  \label{equ2.5}
\end{equation}

(ii) and (iii): Again from equation (\ref{equ2.25}), 
\begin{equation}
v^*_{j,n} x_i v_{s,t}=0 \ \text{ and } \ v^*_{j,n}x_i^* v_{s,t}=0,
\label{equ2.6}
\end{equation}
when $i,j,n, s,t\in \mathbb{N}$ satisfy $1\le i\le \max\{j+n, s+t\}$ and $%
(j,n)\ne (s,t).$ Hence, for each $i, j\ge 1$,
\begin{align}
\psi(x_i)-v_j^*x_iv_j &= \psi(x_i)-\sum_n e_{j,n}\psi(x_i) e_{j,n} + \sum_n
e_{j,n}\psi(x_i) e_{j,n} - \sum_{m,n\ge 1} e_{j,m}v_{j,m}^*x_i v_{j,n}e_{j,n}
\notag \\
&= (\psi(x_i)-\sum_n e_{j,n}\psi(x_i) e_{j,n})  \notag \\
&\quad + \sum_{n\ge 1} e_{j,n}(p_{j,n}\psi(x_i)p_{j,n}-v_{j,n}^*x_i
v_{j,n})e_{j,n}  \notag \\
&\quad - \sum_{m\ne n <i-j} e_{j,m}v_{j,m}^*x_i v_{j,n}e_{j,n}  \tag{by
(\ref{equ2.6})}
\end{align}
By the Condition (b) on the choice of $\{e_{j,n}\}^{\infty}_{j,n=1}$, $\psi(x_i)-\sum_n
e_{j,n}\psi(x_i) e_{j,n}\in \mathcal{K}(\mathcal{N},\tau)$ for $i, j\ge 1$.
Note that each $e_{j,n}(p_{j,n}\psi(x_i)p_{j,n}-v_{j,n}^*x_i v_{j,n})e_{j,n}$
is a finite rank operator in $\mathcal{F}(\mathcal{N},\tau)$ and
\begin{equation}
\sum_n \| e_{j,n}(p_{j,n}\psi(x_i)p_{j,n}-v_{j,n}^*x_i v_{j,n})e_{j,n}\| \le
\sum_n \| (p_{j,n}\psi(x_i)p_{j,n}-v_{j,n}^*x_i v_{j,n}) \| <\infty.
\tag{by (\ref{equ2.3})}
\end{equation}%
Thus $\sum_{n\ge 1} e_{j,n}(p_{j,n}\psi(x_i)p_{j,n}-v_{j,n}^*x_i
v_{j,n})e_{j,n}$ converges in norm to a compact operator in $\mathcal{K}(%
\mathcal{N},\tau)$. Finally $\sum_{m\ne n <i-j} e_{j,m}v_{j,m}^*x_i
v_{j,n}e_{j,n}$ is a finite sum, whence it is in $\mathcal{K}(\mathcal{N}%
,\tau)$. Therefore, for all $i, j\ge 1$,
\begin{equation*}
\psi(x_i)-v_j^*x_iv_j\in \mathcal{K}(\mathcal{N},\tau).
\end{equation*}
Since $\{x_i\}^{\infty}_{i=1}$ is norm dense in the unit ball of $\mathcal{A}$,
\begin{equation}
\psi(x )-v_j^*x v_j\in \mathcal{K}(\mathcal{N},\tau), \ \ \forall \ x\in
\mathcal{A}.  \label{equ2.75}
\end{equation}
Thus (ii) is satisfied.  Moreover, when $1\le i\le j$,
\begin{align}
\|\psi(x_i)-v_j^*x_iv_j \|&= \|\psi(x_i)-\sum_n e_{j,n}\psi(x_i) e_{j,n} +
\sum_n e_{j,n}\psi(x_i) e_{j,n} - \sum_{m,n\ge 1} e_{j,m}v_{j,m}^*x_i
v_{j,n}e_{j,n} \|  \notag \\
&\le \|\psi(x_i)-\sum_n e_{j,n}\psi(x_i) e_{j,n}\|  \notag \\
&\quad + \|\sum_{n\ge 1}
e_{j,n}(p_{j,n}\psi(x_i)p_{j,n}-v_{j,n}^*x_i v_{i,n})e_{j,n}\|
\tag{by
(\ref{equ2.6})} \\
&\quad \le \frac 1 {2^j}+ \frac 1 {2^j} =\frac 1 {2^{j-1}}.  \tag{by
Condition (c) and (\ref{equ2.3})}
\end{align}
So
\begin{equation*}
\lim_j \|\psi(x_i)-v_j^*x_iv_j \| =0, \qquad \forall \ i\ge 1,
\end{equation*}
whence
\begin{equation}
\lim_j \|\psi(x )-v_j^*x v_j \| =0, \qquad \forall \ x\in\mathcal{A}.
\label{equ2.76}
\end{equation}
Thus (iii) is satisfied.

(iv) and (v): For each $j\ge 1$ and $x\in \mathcal{A}$, as $v_j^*v_j=\psi(I_{\mathcal A})$, we have
\begin{align}
(v_j\psi&(x )-x v_j)^*(v_j\psi(x )-x v_j)  \notag \\
& = \psi(x^*)v_j^*v_j\psi(x )-\psi(x ^*)v_j^*x v_j - v_j^*x^*v_j\psi(x)
+v_j^*x^*xv_j  \notag \\
&=\psi(x^*)(\psi(x)-\ v_j^*x  v_j ) +(\psi(x^*)-v_j^*x^*v_j) \psi(x)-
(\psi(x^*x)-v_j^*x^*xv_j)  \label{equ2.9.5}
\end{align}
By (\ref{equ2.75}) and (\ref{equ2.9.5}),
\begin{equation}
(v_j\psi (x)-xv_j)^*(v_j\psi(x)-xv_j) \in \mathcal{K}(\mathcal{N},\tau), \ \
\forall \ j\ge 1.   \notag
\end{equation}%
Hence
\begin{equation}
v_j\psi(x)-xv_j \in \mathcal{K}(\mathcal{N},\tau), \ \ \forall \ j\ge 1.
\label{equ2.10}
\end{equation}%
Thus (iv) is satisfied. Furthermore, from (\ref{equ2.9.5}) and (\ref{equ2.76}%
) we have, for all $x\in \mathcal{A}$,
\begin{align}
\lim_j\|v_j\psi(x)-xv_j \|=\lim_j\| (v_j\psi&(x )-x v_j)^*(v_j\psi(x )-x
v_j) \|^{1/2}=0.  \label{equ2.11}
\end{align}
Thus (iv) is satisfied. This completes the proof of the theorem.
\end{proof}

\subsection{Voiculescu Theorem for nuclear $C^*$-subalgebras in semifinite
factors with respect to norm ideals}

\begin{theorem}
\label{VoiThm_1} Let $\mathcal{N}$ be a countably decomposable,
properly infinite,  semifinite factor with a faithful  normal
semifinite tracial
weight $\tau$. Let $\mathcal{F}(\mathcal{N},\tau)$ and $\mathcal{K}(\mathcal{%
N},\tau)$ be the sets of finite rank operators, and compact operators
respectively, in $(\mathcal{N},\tau)$. Let $\mathcal{K}_{\Phi}(\mathcal{N}%
,\tau)$ be a norm ideal of $(\mathcal{N},\tau)$.

Suppose that $\mathcal{A}$ is a separable nuclear $C^*$-subalgebra of $%
\mathcal{N}$ with an identity $I_{\mathcal{A}}$ and $\mathcal{B}$ is a
countably generated $*$-subalgebra of $\mathcal{A}$ such that $I_{\mathcal{A%
}}\in \mathcal{B}$. Suppose that $\psi:\mathcal{A}\rightarrow
\mathcal{N}$ is a $*$-homomorphism satisfying (i) $\displaystyle
\psi(\mathcal{A}\cap
\mathcal{K}(\mathcal{N},\tau))=0$ and (ii) $\psi(\mathcal{B})$ is $\Phi$%
-well-behaved in $\mathcal{N}$.

Then for every finite subset $F$ of $\mathcal{B}$ and any positive number $%
\epsilon$, there exists a partial isometry $v$ in $\mathcal{N}$ such that

\begin{enumerate}
\item[(a)] $v^*v=\psi(I_{\mathcal{A}})$, and $vv^*\le
I_{\mathcal{A}}$;

\item[(b)] $v\psi(b)-bv\in \mathcal{K}_{\Phi}(\mathcal{N},\tau)$, for all $%
b\in \mathcal{B}$;

\item[(c)] $\Phi(v\psi(b)-bv)\le \epsilon$, for all $b\in F$.
\end{enumerate}
Moreover,
\begin{enumerate}
\item[(d)] $\psi(b)-v^*bv\in \mathcal{K}_{\Phi}(\mathcal{N},\tau)$, for all $%
b\in \mathcal{B}$;

\item[(e)] $\Phi( \psi(b)-v^*bv)\le \epsilon$, for all $b\in F$.
\end{enumerate}
\end{theorem}

\begin{proof}
Let $F$ be a finite subset of $\mathcal{B}$ and $\epsilon>0$. We will need
only to find a partial isometry $v$ in $\mathcal{N}$ satisfying (a), (b) and
(c), as (d) and (e) follow directly from (a), (b) and (c).

Since $\mathcal{B}$ is a countably generated $*$-algebra and $F$ is a
finite subset of $\mathcal{B}$, we might assume that $\{b_k\}_{k=1}^\infty$
is a base of $\mathcal{B}$ (as a linear space) and $F=\{b_1,\ldots, b_m\}$.
From the condition that $\psi(\mathcal{B})$ is $\Phi$-well-behaved in $%
\mathcal{N}$, by Lemma \ref{lemma2.4} there exists a sequence $%
\{e_n\}_{n=1}^\infty$ of finite rank positive operators in $\mathcal{F}(%
\mathcal{N},\tau)$ such that

\begin{enumerate}
\item $\displaystyle \sum_{n=1}^\infty e_n^2=\psi(I_{\mathcal{A}})$,

\item $\displaystyle \sum_{n=1}^\infty \Phi( \psi(b_k) e_n - e_n\psi(b_k)
)\le \epsilon/3$ for all  $1\le k\le m$,

\item $\displaystyle \sum_{n=1}^\infty \Phi(\psi(b_k) e_n -
e_n\psi(b_k))<\infty $  for  $k\ge 1$.
\end{enumerate}
From Theorem \ref{Voiculescu'sTheorem}, there exists a sequence $%
\{v_n\}_{n=1}^\infty$ of partial isometries in $\mathcal{N}$ satisfying

\begin{enumerate}
\item[(4)] $\displaystyle  v_nv^*_n\le I_{\mathcal{A}} , \ v^*_nv_n=\psi(I_{%
\mathcal{A}})$ and $v^*_nv_k=0$ for all $n,k\ge 1$ with $n\ne  k$;

\item[(5)] $\displaystyle  \Phi((v_n\psi(b_k)-b_kv_n)e_n)\le
\|v_n\psi(b_k)-b_kv_n\|\Phi(e_n) \le \frac {\epsilon} {3\cdot 2^{n }}$ for
all $1\le k\le m+n$;
\end{enumerate}
Let $$ v=\sum_{n=1}^\infty v_ne_n.$$ Then from (4) and (1) we
conclude that $v$ is a partial isometry in $\mathcal{N}$ such that
\begin{equation}
vv^*\le I_{\mathcal{A}} \ \ \text{ and } \ \ v^*v= \psi(I_{\mathcal{A}}).
\label{equ2.17}
\end{equation}
Furthermore, for each $k\ge 1$,
\begin{align}
\Phi(v\psi(b_k)-b_kv) &= \Phi(( \sum_{n=1}^\infty v_ne_n) \psi(b_k)-b_k
(\sum_n v_ne_n))  \notag \\
& \le \sum_{n<k} \Phi( v_ne_n \psi(b_k)-b_k v_ne_n )+ \sum_{n\ge k} \Phi(
v_ne_n \psi(b_k)-b_k v_ne_n )  \notag \\
&\le \sum_{n<k} \Phi( v_ne_n \psi(b_k)-b_k v_ne_n )+ \sum_{n\ge k} \Phi(
v_ne_n \psi(b_k)- v_n \psi(b_k)e_n )  \notag \\
& \qquad \qquad \qquad + \sum_{n\ge k} \Phi(v_n \psi(b_k)e_n - b_k v_ne_n )
\notag \\
& \le \sum_{n<k} \Phi( v_ne_n \psi(b_k)-b_k v_ne_n )+ \frac \epsilon 3 +
\sum_{n\ge k} \frac {\epsilon} {3\cdot 2^{n }}  \tag{by  (2) and (5) } \\
& <\infty.  \notag
\end{align}
Since $\{b_k\}_k$ is a base of $\mathcal{B}$ (as a linear space),
\begin{equation}
\text{$v\psi(b)-bv\in \mathcal{K}_{\Phi}(\mathcal{N},\tau)$, for all
$b\in \mathcal{B}$} . \label{equ2.18}
\end{equation}%
When $1\le k\le m$,
\begin{align}
\Phi(v\psi(b_k)-b_kv) &= \Phi(( \sum_{n=1}^\infty v_ne_n) \psi(b_k)-b_k
(\sum_{n=1}^\infty v_ne_n))  \notag \\
& \le \sum_{n=1}^\infty \Phi( v_ne_n \psi(b_k)- v_n \psi(b_k)e_n )+
\sum_{n=1}^\infty \Phi( v_n \psi(b_k)e_n -b_k v_ne_n )  \notag \\
&\le \frac \epsilon 3 + \sum_{n=1}^\infty \frac {\epsilon} {3\cdot 2^{n }}
\tag{by  (2) and (5) } \\
& \le\epsilon.  \label{equ2.19}
\end{align}
From (\ref{equ2.17}), (\ref{equ2.18}) and (\ref{equ2.19}), we complete the
proof of the theorem.
\end{proof}

The following is the main result of the section, which can be viewed as an
analogue of Voiculescu's noncommutative Weyl-von Neumann Theorem (see Theorem 2.4 in \cite{Voi}) in
semifinite factors.

\begin{theorem}
\label{VoiThm_2}  Let $\mathcal{N}$ be a countably decomposable,
properly infinite,   semifinite factor with a faithful  normal
semifinite tracial weight  $\tau$ and let
$\mathcal{K}(\mathcal{N},\tau)$ be the set of compact operators in
$(\mathcal{N},\tau)$. Let $\mathcal{K}_{\Phi}(\mathcal{N},\tau)$ be
a norm ideal of $(\mathcal{N},\tau)$.

Suppose $\mathcal{A}$ is a separable nuclear $C^*$-subalgebra of $\mathcal{N}
$ with an identity $I_{\mathcal{A}}$ and $\mathcal{B}$ is a countably
generated $*$-subalgebra of $\mathcal{A}$ such that $I_{\mathcal{A}}\in
\mathcal{B}$. If $\psi:\mathcal{A}\rightarrow \mathcal{N}$ is a $*$%
-homomorphism satisfying (i) $\displaystyle \psi(\mathcal{A}\cap \mathcal{K}(%
\mathcal{N},\tau))=0$ and (ii) $\psi(\mathcal{B})$ is $\Phi$-well-behaved in
$\mathcal{N}$, then
\begin{equation*}
id_{\mathcal{A}} \sim_{\mathcal{B}} id_{\mathcal{A}} \oplus \psi, \qquad %
\mod (\mathcal{K}_{\Phi}(\mathcal{N},\tau)).
\end{equation*}
\end{theorem}

\begin{proof}
It suffices to show that, for every finite $F\subseteq \mathcal{B}$ and $%
\epsilon>0$,
\begin{equation*}
id_{\mathcal{A}} \sim_{\mathcal{B}}^{(F,\epsilon)} id_{\mathcal{A}} \oplus
\psi , \qquad \mod (\mathcal{K}_{\Phi}(\mathcal{N},\tau)).
\end{equation*}
Suppose that a finite subset $F\subseteq \mathcal{B}$ and $\epsilon>0$ are
given. We might also assume that $F$ is a self-adjoint set, i.e. $b\in F$
implies $b^*\in F$.

Let $\mathcal{H}$ be an infinite dimensional separable Hilbert space and let $%
\mathcal{B}(\mathcal{H})$ be the set of bounded linear operators on $%
\mathcal{H}$. Suppose that $\{f_{i,j}\}_{i,j= 1}^\infty$ is a system of
matrix units of $\mathcal{B}(\mathcal{H})$.

Let $\mathcal{N}\otimes \mathcal{B}(\mathcal{H})$ be a von Neumann algebra
tensor product of $\mathcal{N}$ and $\mathcal{B}(\mathcal{H})$. Recall that $%
\mathcal{K}_{\Phi}(\mathcal{N}\otimes \mathcal{B}(\mathcal{H}),\tau)$ is an
extension of $\mathcal{K}_{\Phi}(\mathcal{N },\tau)$ from $\mathcal{N}$ to $%
\mathcal{N}\otimes \mathcal{B}(\mathcal{H})$ (see Section \ref{section2.2}
for details). We should identify $\mathcal{N}$ with $\mathcal{N}\otimes
f_{1,1}$ in $\mathcal{N}\otimes \mathcal{B}(\mathcal{H})$. 

Define a mapping
\begin{equation*}
\psi^\infty: \mathcal{A}\otimes f_{1,1} \rightarrow \mathcal{N}\otimes
\mathcal{B}(\mathcal{H})
\end{equation*}
by
\begin{equation*}
\psi^\infty ( x\otimes f_{1,1}) = \sum_{i=2}^\infty \psi(x)\otimes f_{i,i},
\qquad \text{ for all } x\in \mathcal{A}.
\end{equation*}
By Condition (i) and Lemma \ref{prelim_lemma2.7}, we know that $\psi^\infty %
\Big( \left( \mathcal{A}\otimes f_{1,1})\cap \mathcal{K}(\mathcal{N}\otimes
\mathcal{B}(\mathcal{H}) ,\tau \right )\Big)=0$. From Condition (ii),
Definition \ref{def4.1.1} and Lemma \ref{prelim_lemma2.11}, it induces that $%
\psi^\infty(\mathcal{B}\otimes f_{1,1})$ is a $\Phi$-well-behaved set in $%
\mathcal{N}\otimes \mathcal{B}(\mathcal{H})$. Therefore, by Theorem \ref%
{VoiThm_1}, there exists a partial isometry $v$ in $\mathcal{N}\otimes
\mathcal{B}(\mathcal{H})$ such that

\begin{enumerate}
\item[(a)] $v^*v=\psi^\infty(I_{\mathcal{A}}\otimes
f_{1,1})=\sum_{i=2}^\infty \psi (I_{\mathcal{A}})\otimes f_{i,i}$ and $%
vv^*\le I_{\mathcal{A}}\otimes f_{1,1}$;

\item[(b)] $v \psi^\infty(b\otimes f_{1,1})-(b\otimes f_{1,1})v\in \mathcal{K%
}_{\Phi}(\mathcal{N}\otimes \mathcal{B}(\mathcal{H}),\tau)$, for all $b\in
\mathcal{B}$;

\item[(c)] $\Phi\big(v \psi^\infty(b\otimes f_{1,1})-(b\otimes f_{1,1})v\big)%
\le \epsilon/8$, for all $b\in F$.
\end{enumerate}
Note that $\mathcal{B}$ is a $*$-algebra. It follows directly from (b) that
\begin{enumerate}
\item[(d)] $vv^*(b\otimes f_{1,1})-(b\otimes f_{1,1})vv^* \in \mathcal{K}%
_{\Phi}(\mathcal{N}\otimes \mathcal{B}(\mathcal{H}),\tau),$ for all $b\in
\mathcal{B}$.
\end{enumerate}
From (c) and the fact that $F$ is a self-adjoint set, one gets that
\begin{enumerate}
\item[(e)] $\Phi( \psi^\infty(b\otimes f_{1,1})v^*-v^*(b\otimes f_{1,1}))\le
\epsilon/8$, for all $b\in F$,
\end{enumerate}
and
\begin{enumerate}
\item[(f)] $\Phi( (b\otimes f_{1,1})vv^*-vv^*(b\otimes f_{1,1}))\le
\epsilon/4$, for all $b\in F$.
\end{enumerate}

By (a), if we denote $p=vv^*\le I_{\mathcal{A}}\otimes f_{1,1} $, then $%
p=q\otimes f_{1,1}$ for some projection $q\le I_{\mathcal{A}}$ in $\mathcal{N%
}$. Thus, from (d), (f) and Lemma \ref{prelim_lemma2.11},
\begin{equation}
bq-qb \in \mathcal{K}_{\Phi}(\mathcal{N },\tau), \qquad \forall \ b\in
\mathcal{B},  \label{equ5.16.4}
\end{equation}
and
\begin{equation}
\Phi(bq-qb)\le \epsilon/4, \qquad \forall \ b\in F.  \label{equ5.16.5}
\end{equation}
Define
\begin{equation*}
\psi_q: \mathcal{A}\rightarrow \mathcal{N}
\end{equation*}
by%
\begin{equation*}
\psi_q (x) =(I_{\mathcal{A}}-q)x (I_{\mathcal{A}}-q), \ \text{ for all }
x\in \mathcal{A }.
\end{equation*}
Let
\begin{equation*}
w=((I_{\mathcal{A}}-q)\otimes f_{1,1})+v=(\psi_q(I_{\mathcal{A}})\otimes
f_{1,1})+v \in \mathcal{N}\otimes \mathcal{B}(\mathcal{H}).
\end{equation*}
Then, from (a),
\begin{equation}
w^*w=\psi_q(I_{\mathcal{A}})\otimes f_{1,1}+ \sum_{i=2}^\infty \psi (I_{%
\mathcal{A}})\otimes f_{i,i}\quad \text{ and } \qquad ww^*= I_{\mathcal{A}%
}\otimes f_{1,1}.  \label{equ5.20}
\end{equation}
Now, for all $b\in \mathcal{B}$,
\begin{align}
w \Big(\psi_q(b)& \otimes f_{1,1} + \sum_{i=2}^\infty \psi (b)\otimes
f_{i,i} \Big )- \Big(b\otimes f_{1,1}\Big)w  \label{equ5.22} \\
&=\Big(\psi_q(b)\otimes f_{1,1}+ v\psi^\infty(b\otimes f_{1,1})\Big)- \Big(%
(b(I_{\mathcal{A}}-q))\otimes f_{1,1}+(b\otimes f_{1,1})v\Big)  \notag \\
&= \Big(\psi_q(b)\otimes f_{1,1}-(b(I_{\mathcal{A}}-q))\otimes f_{1,1}\Big)+%
\Big(v\psi^\infty(b\otimes f_{1,1})-(b\otimes f_{1,1})v\Big)  \notag \\
&= \Big( ( -q b(I_{\mathcal{A}}-q))\otimes f_{1,1}\Big)+\Big(%
v\psi^\infty(b\otimes f_{1,1})-(b\otimes f_{1,1})v\Big)  \tag{by definition
of $\psi_q$} \\
& \in \mathcal{K}_{\Phi}(\mathcal{N}\otimes \mathcal{B}(\mathcal{H}),\tau).
\tag{by  Lemma \ref{prelim_lemma2.11}, (\ref{equ5.16.4}) and (b)}
\end{align}
Moreover, for all $b\in F$,
\begin{align}
&\Phi(w \Big(\psi_q(b) \otimes f_{1,1} + \sum_{i=2}^\infty \psi (b)\otimes
f_{i,i} \Big )- \Big(b\otimes f_{1,1}\Big)w)  \label{equ5.21} \\
& \qquad\qquad \le \Phi \Big( ( -q b(I_{\mathcal{A}}-q))\otimes f_{1,1}\Big)%
+ \Phi\Big(v\psi^\infty(b\otimes f_{1,1})-(b\otimes f_{1,1})v\Big)  \notag \\
&\qquad\qquad \le \epsilon/2,  \tag{by Lemma \ref{prelim_lemma2.11},
(\ref{equ5.16.5}) and (c)}
\end{align}
From (\ref{equ5.20}), (\ref{equ5.22}), (\ref{equ5.21}) and Definition \ref%
{prelim_def3.1.1},
\begin{equation}
id_{\mathcal{A}} \sim_{\mathcal{B}}^{(F,\epsilon/2)} \psi_q\oplus
\psi\oplus\psi\oplus \cdots , \qquad \mod (\mathcal{K}_{\Phi}(\mathcal{N}%
,\tau)).  \label{equ5.23}
\end{equation}
It is a direct consequence of Definition \ref{prelim_def3.1.1} that
\begin{equation*}
\psi_q\oplus \psi\oplus\psi\oplus \cdots \sim_{\mathcal{B}} \psi_q\oplus
0\oplus\psi\oplus \psi\oplus\cdots, \qquad \mod (\mathcal{K}_{\Phi}(\mathcal{%
N},\tau))
\end{equation*}
which implies that
\begin{equation*}
id_{\mathcal{A}} \sim_{\mathcal{B}}^{(F,\epsilon/2)} \psi_q\oplus
0\oplus\psi\oplus \psi\oplus\cdots, \qquad \mod (\mathcal{K}_{\Phi}(\mathcal{%
N},\tau)).
\end{equation*}
Hence,
\begin{equation}
id_{\mathcal{A}} \oplus \psi \sim_{\mathcal{B}}^{(F,\epsilon/2)}
\psi_q\oplus \psi\oplus\psi\oplus \psi\oplus\cdots, \qquad \mod (\mathcal{K}%
_{\Phi}(\mathcal{N},\tau)).  \label{equ5.24}
\end{equation}
Combining (\ref{equ5.23}) and (\ref{equ5.24}), we have
\begin{equation*}
id_{\mathcal{A}} \oplus \psi \sim_{\mathcal{B}}^{(F, \epsilon)} id_{\mathcal{%
A}}, \qquad \mod (\mathcal{K}_{\Phi}(\mathcal{N},\tau)).
\end{equation*}
This completes the proof of the theorem.
\end{proof}

\subsection{Some consequences}

\ \newline

The following are quick consequences of Theorem \ref{VoiThm_2}, Proposition %
\ref{lemma4.2.2} and Proposition \ref{prop4.3.1}.

\begin{theorem}
\label{VoiThmNuclear}  Let $\mathcal{N}$ be a countably
decomposable, properly infinite,   semifinite factor with a faithful
normal semifinite tracial weight  $\tau$ and let
$\mathcal{K}(\mathcal{N},\tau)$ be the set of compact operators in
$(\mathcal{N},\tau)$.

Suppose $\mathcal{A}$ is a separable nuclear $C^*$-subalgebra of $\mathcal{N}
$ with an identity $I_{\mathcal{A}}$. If $\psi:\mathcal{A}\rightarrow
\mathcal{N}$ is a $*$-homomorphism satisfying $\displaystyle \psi(\mathcal{A}%
\cap \mathcal{K}(\mathcal{N},\tau))=0$, then
\begin{equation*}
id_{\mathcal{A}} \sim_{\mathcal{A}} id_{\mathcal{A}} \oplus \psi, \qquad %
\mod (\mathcal{K }(\mathcal{N},\tau)).
\end{equation*}
\end{theorem}

\begin{theorem}
\label{VoiThmComm}  Let $\mathcal{N}$ be a countably decomposable,
properly infinite,  semifinite factor with a faithful  normal
semifinite tracial weight  $\tau$ and let
$\mathcal{K}(\mathcal{N},\tau)$ be the set of compact
operators in $(\mathcal{N},\tau)$. Let $r\ge 2$ be a positive integer and let $%
\mathcal{K}_r(\mathcal{N},\tau)=L^r(\mathcal{N},\tau)\cap \mathcal{N}$ be a
norm ideal of $(\mathcal{N},\tau)$ equipped with a norm $\Phi$ satisfying $%
\Phi(x) = \max\{\|x\|, \|x\|_r\}, \ \forall \ x\in L^r(\mathcal{N},\tau)\cap
\mathcal{N}. $

Suppose $\mathcal{A}$ is a separable abelian $C^*$-subalgebra of $\mathcal{N}
$ with an identity $I_{\mathcal{A}}$ and $\mathcal{B}$ is a $*$-subalgebra
generated by a family of self-adjoint elements $I_{\mathcal A}, x_1,\ldots, x_r$ in $%
\mathcal{A}$. If $\psi:\mathcal{A}\rightarrow \mathcal{N}$ is a $*$%
-homomorphism satisfying $\displaystyle \psi(\mathcal{A}\cap \mathcal{K}(%
\mathcal{N},\tau))=0$, then
\begin{equation*}
id_{\mathcal{A}} \sim_{\mathcal{B}} id_{\mathcal{A}} \oplus \psi, \qquad %
\mod (\mathcal{K}_r (\mathcal{N},\tau)).
\end{equation*}
\end{theorem}

\section{Perturbations of Normal Operators in Semifinite von Neumann
Algebras}

The section is devoted to prove a version of Voiculescu's Theorem for normal
operators in a semifinite von Neumann algebra.


\subsection{Voiculescu Theorem for normal operators in semifinite factors}

\begin{lemma}
\label{k1} Let $\mathcal{N}$ be a countably decomposable,  properly
infinite,  semifinite factor with a faithful  normal  semifinite
tracial weight $\tau$. Let $\mathcal{F}(\mathcal{N},\tau)$ and
$\mathcal{K}(\mathcal{N},\tau)$ be
the sets of finite rank operators, and compact operators respectively, in $(%
\mathcal{N},\tau)$. Suppose $\mathcal{W}$ is a   von Neumann subalgebra of $%
\mathcal{N}$ containing $I$, an identity of $\mathcal N$. Then there
is a sequence $\{ p_{n} \} _{n\in\mathbb{N}}$ of orthogonal
projections in $\mathcal{W}$ such that

\begin{enumerate}
\item $0\leq\tau (p_{n} )<\mathcal{1}$;

\item For any $x\in \mathcal{W}\cap \mathcal{K}(\mathcal{N},\tau ),$ $( I-p
) x=0$ where $p=\sum_{n\in \mathbb{N}}p_{n}.$
\end{enumerate}
\end{lemma}

\begin{proof}Without loss of generality we assume that $\mathcal{W}\cap \mathcal{K}(\mathcal{N},\tau )\neq 0$%
. Since $\mathcal{N}$ is countably decomposable, any family of nonzero
mutually orthogonal projections in $\mathcal{N}$ must be countable. By
Zorn's Lemma, there exists a maximal family $\{ p_{n} \} _{n=1}^{N} $ (here $%
N$ might be $\infty $) of nonzero orthogonal projections in $\mathcal{W}$
such that $0\leq \tau ( p_{n} ) <\infty $ for each $n$. Let $%
p=\sum_{n=1}^{N}p_{n}$ be a projection in $\mathcal{W}$.

We assert that $(I-p )x=0$ for every $x$ in $\mathcal{W}\cap \mathcal{K}(%
\mathcal{N},\tau)$. In fact, assume contrarily that there exists an $x$ in $%
\mathcal{W}\cap\mathcal{K}(\mathcal{N},\tau)$ such that $(I-p)x\neq0$. Then
\begin{equation*}
0\neq(I-p)xx^{\ast}(I-p)\in\mathcal{W}\cap\mathcal{K}(\mathcal{N},\tau).
\end{equation*}
By Lemma 6.8.1 in \cite{Kadison}, there exists a nonzero spectral projection
$e$ of $(I-p)xx^{\ast}(I-p)$ in $(I-p){\mathcal{W}}(I-p)$ such that $e\in%
\mathcal{W}\cap\mathcal{K}(\mathcal{N},\tau)$. This further implies that $%
e\in\mathcal{F}(\mathcal{N},\tau)$ and $0\leq\tau (e )<\infty$ and $e\leq
(I-p ).$ It contradicts the maximality of $\{ p_{n} \}^{N}_{n=1}$.

If $N<\infty$, then, by allowing $p_{n}=0$ for all $n>N$, we get a sequence $
\{ p_{n} \}^{\infty}_{n=1}$ with desired properties. This completes the
proof of the lemma.
\end{proof}

\begin{theorem}
\label{thm6.1} \label{vo}\bigskip Let $\mathcal{N}$ be a countably
decomposable,  properly infinite,  semifinite factor with a faithful
normal  semifinite tracial weight $\tau$. Let $r\geq2$ be a positive
integer. Assume
$\{ a_{i} \}_{i=1}^r $ is a family of commuting self-adjoint operators in $%
\mathcal{N}$. Then, for any $\epsilon>0,$ there is a family $%
\{d_{i}\}_{i=1}^r $of commuting diagonal operators in $\mathcal{N}$ such
that
\begin{equation*}
\max_{1\le i\le r} \{\|a_i-d_i\|, \|a_i-d_i\|_r\} \le \epsilon.
\end{equation*}
\end{theorem}

\begin{proof}
Recall that $\mathcal{K}(\mathcal{N},\tau)$ is the set of compact operators in $(%
\mathcal{N},\tau)$ and $\mathcal{K}_{r}(\mathcal{N},\tau)=L^{r}(\mathcal{N}%
,\tau)\cap\mathcal{N}$ is a norm ideal of $(\mathcal{N},\tau)$ equipped with
a norm $\Phi$ satisfying
\begin{equation*}
\Phi (x)=\max\{\Vert x\Vert,\Vert x\Vert_{r}\},\ \forall\ x\in L^{r}(%
\mathcal{N},\tau)\cap\mathcal{N}.
\end{equation*}

Let $\mathcal{W}$ be an abelian von Neumann subalgebra generated by
$I, a_{1},\ldots ,a_{r} $ in $\mathcal{N}$, where $I$ is an identity
of $\mathcal N$. By Lemma \ref{k1}, there exists a sequence $ \{
p_{n} \} _{n=1}^{\infty }$ of mutually orthogonal projections in
$\mathcal{W} $ such that the following are true.

\begin{enumerate}
\item[(i)] $0\leq \tau ( p_{n} ) <\mathcal{1}$.

\item[(ii)] Let $p=\sum_{n\in \mathbb{N}}p_{n}$. Then $( I-p ) x=0$ for all $%
x\in \mathcal{W}\cap \mathcal{K} ( \mathcal{N},\tau ) $.
\end{enumerate}

From (ii), it follows
\begin{equation*}
\tau ( I-p ) =\infty \text{ or }0.
\end{equation*}%
We will proceed the proof according to either $\tau ( I-p ) =\infty $ or $%
\tau(I-p)=0$.

Case (1): Suppose $\tau ( I-p ) =\mathcal{1}$.

We have
\begin{equation}
a_{i}=pa_{i}+ ( I-p ) a_{i}\qquad \text{ for }1\leq i\leq r.
\label{equ6.0.1}
\end{equation}
Since $\mathcal{N}$ is a countably decomposable, properly infinite
factor, there exists a sequence of mutually orthogonal projections
$\{ q_{n} \}
_{n\in \mathbb{N}}$ in $\mathcal{N}$ such that $I-p=\sum_{n\in \mathbb{N}%
}q_{n}$ and $\tau ( q_{n} ) =\mathcal{1}$ for all $n\in \mathbb{N}.$

Let $\mathcal{A}_{1}$ and $\mathcal{B}_{1}$ be an abelian $C^*$-subalgebra,
and a $*$-subalgebra respectively, generated by $I, a_{1},\ldots ,a_{r} $ in
$\mathcal{W}$. Note $( I-p ) \mathcal{A}_{1}$ is also an abelian and
separable $C^*$-algebra. Thus there exists a sequence $\{ \rho _{k} \}
_{k\in \mathbb{N}}$ of one-dimensional $*$-representations of $( I-p )
\mathcal{A}_{1}$ such that $\oplus _{k}\rho _{k}$ is faithful on $( I-p )
\mathcal{A}_{1}.$

Note $I-p=\sum_{n\in \mathbb{N}}q_{n}$ and $\rho _{k}( ( I-p ) \mathcal{A}%
_{1})$ is a set of  scalars for each $k\ge 1$. We define
\begin{equation}
y_{i}\triangleq pa_{i}+\sum_{k=1}^{\mathcal{1}}\rho _{k} ( ( I-p ) a_{i} )
q_{k},\qquad \text{ for }1\leq i\leq r.  \label{equ6.0.2}
\end{equation}
Let $\mathcal{A}_{2}$ and $\mathcal{B}_{2}$ be an abelian
$C^*$-subalgebra, and a $*$-algebra respectively, generated by $\{I,
y_{1},\ldots ,y_{r} \} $ in $\mathcal{N}$. Note $p$ and $I-p$
commute with $\mathcal{A}_2$.

\vspace{0.2cm}

\textit{Claim \ref{thm6.1}.1. }\emph{For each $1 \le i\le r $, there is a
diagonal operator  $\widetilde{d_{i}}$ in $\mathcal{N}$ such that
\begin{equation*}
\Phi ( y_{i}-\widetilde{d_{i}} ) \le\frac{\epsilon }{2} .
\end{equation*}%
}

\vspace{0.2cm}

\noindent\emph{Proof of the Claim \ref{thm6.1}.1:} As each $\rho_k$ is a one
dimensional $*$-representation of $( I-p ) \mathcal{A}_{1}$,
\begin{equation*}
( I-p ) y_{i} =\sum_{k=1}^{\mathcal{1}}\rho _{k} ( ( I-p ) a_{i} ) q_{k}
\end{equation*}
is a diagonal operator in $\mathcal{N}$ for each $i\in \{ 1,\ldots ,r \}$.
From (ii) and (\ref{equ6.0.2}),
\begin{equation*}
py_{i}=pa_{i}=\sum^{\infty}_{n=1} p_{n}a_{i}.
\end{equation*}
For each $n\in \mathbb{N}$, let ${\mathcal{W}}_{n}$ be an abelian von
Neumann subalgebra generated by $\{p_n, p_{n}a_{1}  ,\ldots ,p_{n}a_{r} \}$
in $\mathcal{W}$. By spectral theorem, there are $k_n$ in $%
\mathbb{N}$, a family $\{ e_{1}^{ ( n ) },\ldots ,e_{k_{n}}^{ ( n ) } \} $
of mutually orthogonal projections in ${\mathcal{W}}_{n}$  and a family $\{
\lambda _{i,1}^{ ( n ) },\ldots ,\lambda _{i,k_{n}}^{ ( n ) } \} $ of real
numbers, such that
\begin{equation*}
\sum_{l=1}^{k_{n}} e_{l}^{ ( n ) }=p_n, \quad \text{ and } \quad \|
p_{n}a_{i} -\sum_{l=1}^{k_{n}}\lambda _{i,l}^{ ( n ) }e_{l}^{ ( n ) }\| \le%
\frac{\epsilon }{ 2^{n+1}(\Phi ( p_{n} ) +1) }, \qquad \forall \ 1\le i\le r.
\end{equation*}%
By Lemma \ref{prelim_lemma1},
\begin{equation*}
\Phi ( p_{n}a_{i} -\sum_{l=1}^{k_{n}}\lambda _{i,l}^{ ( n ) }e_{l}^{ ( n ) }
) \leq \Phi ( p_{n} ) \| p_{n}a_{i} -\sum_{l=1}^{k_{n}}\lambda _{i,l}^{ ( n
) }e_{l}^{ ( n ) }\| \le\frac{\epsilon }{2^{n+1}}\text{ } , \qquad \forall \
n\in \mathbb{N}, \ \forall 1\le i\le r.
\end{equation*}%
Then
\begin{equation*}
\widetilde{d}_{i}=\sum_{n=1}^\infty\Big( \sum_{l=1}^{k_{n}}\lambda _{i,l}^{
( n ) }e_{l}^{ ( n ) }\Big)+ \sum_{k=1}^{\mathcal{1}}\rho _{k} ((I-p) a_{i}
) q_{k}
\end{equation*}%
is a diagonal operator in $\mathcal{N}$ satisfying
\begin{equation*}
\Phi ( y_{i}-\widetilde{d_{i}} ) \le\frac{\epsilon }{2} , \qquad \forall \
i\in \{ 1,\ldots,r \}.
\end{equation*}
This ends the proof of the claim.

\vspace{0.2cm}

\noindent\textit{(Continue the proof of theorem:) }Define a $*$-homomorphism
$\psi _{1}:\mathcal{A}_{1}\rightarrow \mathcal{N}$ by
\begin{equation}
\psi_{1}(x)=\sum_{k=1}^{\mathcal{1}}\rho_{k} ( (I-p ) x ) q_{k} \qquad
\text { for all } x\in \mathcal{A}_1.  \label{equ6.1}
\end{equation}
From (ii) in the choice of $\{p_n\}_{n=1}^\infty$, it is clear that
\begin{equation*}
\psi _{1} ( \mathcal{A}_{1}\cap \mathcal{K} ( \mathcal{N},\tau ) ) =0\qquad
\text{ and } \qquad \psi_1 (I)=I-p.
\end{equation*}%

Since $\oplus _{k}\rho _{k}$ is a faithful $*$-representation of $( I-p )
\mathcal{A}_{1},$ the map
\begin{equation*}
\rho:  ( I-p ) \mathcal{A}_{1}\rightarrow ( I-p )\mathcal{A}_2,
\end{equation*}%
defined by
\begin{equation*}
\rho(x)=\sum_{k\in \mathbb{N}}\rho _{k}(x)q_k, \qquad \forall \ x\in ( I-p )
\mathcal{A}_{1}
\end{equation*}
is a $*$-isomorphism from $( I-p ) \mathcal{A}_{1}$ onto $( I-p )\mathcal{A}%
_2$. Thus the map
\begin{equation*}
\psi _{2}:\mathcal{A}_{2}\rightarrow \mathcal{N},
\end{equation*}
by
\begin{equation}
\psi_{2}(y)= \rho^{-1} ( (I-p ) y ) \qquad \text { for all } y\in \mathcal{A}%
_2,  \label{equ6.2}
\end{equation}
is well-defined. By the fact that $\tau(q_k)=\infty$ for all $k\ge 1$, we
conclude that $\psi_2$ is a $*$-homomorphism such that
\begin{equation*}
\psi _{2} ( \mathcal{A}_{2}\cap \mathcal{K} ( \mathcal{N},\tau ) ) =0,
\qquad \text{ and } \qquad \psi_2(I)=I-p.
\end{equation*}
As a summary, we have proven the following claim.

\vspace{0.2cm}

\textit{Claim \ref{thm6.1}.2. } \textit{\ For }$1\le j\le 2$, $\psi _{j}:%
\mathcal{A}_{j}\rightarrow \mathcal{N}$ \textit{\ is a $*$-homomorphism
satisfying }%
\begin{equation*}
\psi _{j} ( \mathcal{A}_{j}\cap \mathcal{K} ( \mathcal{N},\tau ) ) =0,
\qquad \text{ and } \qquad \psi_j(I)=I-p.
\end{equation*}

\vspace{0.2cm}

\textit{Claim \ref{thm6.1}.3.} \textit{With }$\psi _{1}$\textit{\ and }$\psi
_{2}$\textit{\ as above, we have }%
\begin{equation*}
id_{\mathcal{A}_{1}}\sim _{\mathcal{B}_{1}}id_{\mathcal{A}_{1}}\oplus \psi
_{1}\text{ \ }\mod ( \mathcal{K}_{r} ( \mathcal{N},\tau ) ),
\end{equation*}%
\textit{and }%
\begin{equation*}
id_{\mathcal{A}_{2}}\sim _{\mathcal{B}_{2}}id_{\mathcal{A}_{2}}\oplus \psi
_{2}\text{ \ }\mod ( \mathcal{K}_{r} ( \mathcal{N},\tau ) ) .
\end{equation*}%
\textit{Moreover, for $1\le i\le r$, }%
\begin{equation}
( id_{\mathcal{A}_{1}}\oplus \psi _{1} ) ( a_{i} ) =pa_{i}\oplus
(I-p)a_{i}\oplus \sum_{k=1}^{\mathcal{1}}\rho _{k} ( ( I-p ) a_{i} ) q_{k},
\label{equ6.5}
\end{equation}%
\emph{and}
\begin{equation}
( id_{\mathcal{A}_{2}}\oplus \psi _{2} ) ( y_{i} )=pa_{i}\oplus \sum_{k=1}^{%
\mathcal{1}}\rho _{k} ( (I-p ) a_{i} ) q_{k}\oplus (I-p)a_{i}.
\label{equ6.6}
\end{equation}

\vspace{0.2cm}

\noindent \emph{Proof of the Claim \ref{thm6.1}.3:} By Theorem \ref{VoiThm_2}
and Claim 6.1.2.2,
\begin{equation*}
id_{\mathcal{A}_{1}}\sim _{\mathcal{B}_{1}}id_{\mathcal{A}_{1}}\oplus \psi
_{1}\text{ \ }\mod ( \mathcal{K}_{r} ( \mathcal{N},\tau ) ),
\end{equation*}%
and
\begin{equation*}
id_{\mathcal{A}_{2}}\sim_{\mathcal{B}_{2}}id_{\mathcal{A}_{2}}\oplus \psi
_{2}\text{ \ }\mod ( \mathcal{K}_{r} ( \mathcal{N},\tau ) ).
\end{equation*}%
From (\ref{equ6.0.1}) and (\ref{equ6.1}),
\begin{equation*}
( id_{\mathcal{A}_{1}}\oplus \psi _{1} ) ( a_{i} ) =pa_{i}\oplus
(I-p)a_{i}\oplus \sum_{k=1}^{\mathcal{1}}\rho _{k} ( ( I-p ) a_{i} ) q_{k}.
\end{equation*}%
From (\ref{equ6.0.2}) and (\ref{equ6.2}),
\begin{equation*}
( id_{\mathcal{A}_{2}}\oplus \psi _{2} ) ( y_{i} ) =pa_{i}\oplus \sum_{k=1}^{%
\mathcal{1}}\rho _{k} ( (I-p ) a_{i} ) q_{k}\oplus (I-p)a_{i}
\end{equation*}
for $i\in \{ 1,\ldots ,r \} .$ This ends the proof of the claim.

\vspace{0.2cm}

\noindent\textit{(Continue the proof of theorem:) }By Claim \ref{thm6.1}.3
and Definition \ref{prelim_def3.1.1}, there exist partial isometries $w_1$
and $w_2$ in $\mathcal{N}\otimes \mathcal{B}(\mathcal{H})$, where $\mathcal{H%
}$ is a separable Hilbert space, such that

\begin{enumerate}
\item[(iii)] $w_1w_1^*=I$ and $w_1^*w_1=p\oplus (I-p)\oplus (I-p) $;

\item[(iv)] $\displaystyle \Phi\Big(a_i-w_1 \Big( pa_{i}\oplus
(I-p)a_{i}\oplus \sum_{k=1}^{\mathcal{1}}\rho _{k} ( ( I-p ) a_{i} ) q_{k} %
\Big)w_1^*\Big)\le \epsilon/4,$ for all $1\le i\le r$.

\item[(v)] $w_2w_2^*=I$ and $w_2^*w_2=p\oplus (I-p)\oplus (I-p) $;

\item[( vi)] $\displaystyle \Phi\Big(y_i-w_2 \Big( pa_{i}\oplus \sum_{k=1}^{%
\mathcal{1}}\rho _{k} ( ( I-p ) a_{i} ) q_{k}\oplus (I-p)a_{i} \Big)w_2^*%
\Big)\le \epsilon/4,$ for all $1\le i\le r$.
\end{enumerate}
From (\ref{equ6.5}) and (\ref{equ6.6}), there exists a partial isometry $w_3$
in $\mathcal{N}\otimes \mathcal{B}(\mathcal{H})$ such that

\begin{enumerate}
\item[(vii)] $w_3w_3^*=w_3^*w_3=p\oplus (I-p)\oplus (I-p) $;

\item[( viii)] $\displaystyle pa_{i}\oplus \sum_{k=1}^{\mathcal{1}}\rho _{k}
( ( I-p ) a_{i} ) q_{k}\oplus (I-p)a_{i}=w_3 \Big( pa_{i}\oplus
(I-p)a_{i}\oplus \sum_{k=1}^{\mathcal{1}}\rho _{k} ( ( I-p ) a_{i} ) q_{k} %
\Big)w_3^*,$ for all $1\le i\le r$.
\end{enumerate}

Let $u=w_2w_3w_1$. Then from (iii), (v) and (viii), we know that $u$ is a
unitary element in $\mathcal{N}$. Now (iv) and (vi) imply that
\begin{equation*}
\Phi ( a_{i}-u^{\ast }y_{i}u )\le\frac{\epsilon }{2}.
\end{equation*}
By Claim \ref{thm6.1}.1, there exist diagonal operators $\{\widetilde{d_{i}}%
\}_{i=1}^r$ in $\mathcal{N}$ such that, for each $1\le i\le r$,
\begin{equation*}
\Phi ( y_{i}-\widetilde{d_{i}} ) \le\frac{\epsilon }{2} .
\end{equation*}
Let $d_{i}=u^{\ast }\widetilde{d_{i}}u$ for each $1\le i\le r$. Then $d_i$
is a diagonal operator in $\mathcal{N}$ satisfying
\begin{equation*}
\Phi ( a_{i}-d_{i} ) \leq \Phi ( a_{i}-u^{\ast }y_{i}u ) +\Phi ( u^{\ast
}y_{i}u-u^{\ast }\widetilde{d}_{i}u ) \le\epsilon .\text{ }
\end{equation*}

Case (2): Suppose that $\tau (I-p ) =0$. Then $a_{i}=pa_{i}$ for $1\leq
i\leq r.$ So Claim \ref{thm6.1}.1 finishes the proof of the result.
\end{proof}

Applying Theorem \ref{thm6.1}, we obtain the following
generalization of Voiculescu's Theorem for normal operators in a
type I$_\infty$ factor.

\begin{corollary}
\label{vo2} Let $\mathcal{N}$ be a countably decomposable,  properly
infinite,  semifinite factor with a faithful  normal  semifinite
tracial
weight $\tau$. Assume $a\in\mathcal{N}$ is a normal operator. Given any $%
\epsilon>0,$ there is a diagonal operator $d$ in $\mathcal{N} $ such that $%
\|a-d\|\le \epsilon$ and $\|a-d\|_2\le \epsilon.$
\end{corollary}

\begin{proof}
Suppose $a=a_{1}+ia_{2}$ where $a_{1}$ and $a_{2}$ are self-adjoint and
commutative$.$ Then the statement is a direct consequence of Theorem \ref{vo}%
.
\end{proof}


\begin{lemma}
\label{lemma6.1.4} Let $\mathcal{N}$ be a countably decomposable,
properly infinite, semifinite factor with a faithful  normal
tracial weight $\tau$. Suppose $q$ is a projection in $\mathcal{N}$
and $\{\lambda_n\}_{n=1}^\infty$ is a family of complex numbers with
a limit point $\lambda_0$. Then, for each $\epsilon>0$, there exists
a family $\{q_n\}_{n=1}^\infty$ of orthogonal projections in
$\mathcal{N}$ such that $\sum_{n=1}^\infty q_n=q$ and
\begin{equation*}
\max\{ \|\lambda_0q-\sum_{n=1}^\infty \lambda_nq_n\|,
\|\lambda_0q-\sum_{n=1}^\infty \lambda_nq_n\|_2 \}\le \epsilon.
\end{equation*}
\end{lemma}

\begin{proof}
Since $\tau$ is a faithful, normal,  semifinite tracial weight of a countably
decomposable factor $\mathcal{N}$, there exists a family $%
\{e_k\}_{k=1}^\infty$ of orthogonal projections in $\mathcal{N}$ such that $%
\sum_{k=1}^\infty e_k=q$ and $\tau(e_k)<\infty$ for all $k\ge 1$. For each $%
k\ge 1$, there exists a positive integer ${m_k}$ such that
\begin{equation*}
\max\{ \|\lambda_0e_k- \lambda_{m_k}e_k\|, \|\lambda_0e_k-
\lambda_{m_k}e_k\|_2 \}\le \frac {\epsilon}{2^k},
\end{equation*}
as $\lambda_0$ is a limit point of $\{\lambda_k\}_{k=1}^\infty$ and $%
\tau(e_k)<\infty.$ It follows that
\begin{equation*}
\max\{ \|\lambda_0q-\sum_{k=1}^\infty \lambda_{m_k}e_k\|,
\|\lambda_0q-\sum_{k=1}^\infty \lambda_{m_k}e_k\|_2 \}\le \epsilon.
\end{equation*}%
Let, for each $n\ge 1$,
\begin{equation*}
q_n=\sum_{m_k=n}e_{k},
\end{equation*}
where some $q_n$ could be $0$. Then, $\sum_{n=1}^\infty q_n=q$ and
\begin{equation*}
\max\{ \|\lambda_0q-\sum_{n=1}^\infty \lambda_nq_n\|,
\|\lambda_0q-\sum_{n=1}^\infty \lambda_nq_n\|_2 \}\le \epsilon.
\end{equation*}
\end{proof}

Combining Corollary \ref{vo2} and Lemma \ref{lemma6.1.4}, we have the
following result.

\begin{lemma}
\label{lemma6.1.5}  Let $\mathcal{N}$ be a countably decomposable,
properly infinite,  semifinite factor with a faithful  normal
semifinite tracial
weight $\tau$. Assume $a\in\mathcal{N}$ is a normal operator. Suppose that $%
\{\lambda_n\}_{n=1}^\infty$ is a family of complex numbers such that the
ball  $\{\lambda\in\mathbb{C }\ : \ |\lambda|\le 2\|a\|\}$ is contained in
the closure of $\{\lambda_n\}_{n=1}^\infty$. Then, for any $0<\epsilon<\|a\|,
$ there is a family $\{q_n\}_{n=1}^\infty$ of orthogonal projections in $%
\mathcal{N}$ such that
\begin{equation*}
\text{$\|a-\sum_{n=1}^\infty \lambda_nq_n\|\le \epsilon$ \ \ \ and \ \ \ $%
\|a-\sum_{n=1}^\infty \lambda_nq_n\|_2\le \epsilon.$}
\end{equation*}
\end{lemma}

\subsection{Voiculescu Theorem for normal operators in semifinite von
Neumann algebras}

\ \newline

In this subsection, we are going to use the theory of direct integral of
separable Hilbert spaces and von Neumann algebras acting on separable
Hilbert spaces (see \cite{Kadison} for general knowledge of direct
integral). 
In  order to give a clear discussion, we list some lemmas which will be
needed in this part of the paper.

\begin{lemma}
\label{k2}(\cite{Kadison}) Suppose $\mathcal{M}$ is a von Neumann algebra
acting on a separable Hilbert space $\mathcal{H}$ and $\mathcal{Z}$ is the
center of $\mathcal{M}$. Then there is a direct integral decomposition of $%
\mathcal{M}$ relative to $\mathcal{Z}$, i.e., there exists a locally compact
complete separable metric (Borel measure) space $(X,\mu)$ such that

\begin{enumerate}
\item $\mathcal{H}$ is (unitarily equivalent to) the direct integral of $ \{
\mathcal{H}_{s}:s\in X \} $ over $(X,\mu),$ where each $\mathcal{H}_{s}$ is
a separable Hilbert space, $s\in X;$

\item $\mathcal{M}$ is (unitarily equivalent to) the direct integral of $ \{
\mathcal{M}_{s}:s\in X \} $ over $(X,\mu),$ where $\mathcal{M}_{s}$ is a
factor in $\mathcal{B} ( \mathcal{H}_{s} ) $ almost everywhere. Also, if $%
\mathcal{M}$ is of type I$_{n}$ ($n$ could be infinite), II$_{1}$,  II$_{%
\mathcal{1}}$ or III, then the components $\mathcal{M}_{s}$ are, almost
everywhere, of type I$_{n}$, II$_{1}$, II$_{\mathcal{1}}$ or III
respectively.
\end{enumerate}
\end{lemma}

\begin{remark}
\label{k3}Let $\mathcal{M=}\int_{X}\oplus\mathcal{M}_{s}d\mu$ and $\mathcal{%
H=}\int_{X}\oplus\mathcal{H}_{s}d\mu$ be the direct integral decompositions
of $( \mathcal{M},\mathcal{H} ) $ relative to the center $\mathcal{Z}$ of $%
\mathcal{M}$. By the argument in section 14.1 in \cite{Kadison}, we can find
a separable Hilbert space ${\tilde{\mathcal{H}}}$ and a family of unitaries $
\{ U_{s}:\mathcal{H}_{s}\rightarrow{\tilde{\mathcal{H}}}\text{ as }s\in X \}
$ such that $s\rightarrow U_{s}\eta_1(s)$ is measurable (i.e., $s\rightarrow
\langle U_{s}\eta_1(s),\eta_2 \rangle $ is measurable for any vector $\eta_2$
in ${\tilde{\mathcal{H}}}$) for every $\eta_1\in\mathcal{H}$ and $%
s\rightarrow U_{s}x ( s ) U_{s}^{\ast}$ is measurable (i.e., $s\rightarrow
\langle U_{s}x ( s ) U_{s}^{\ast}\eta_1,\eta_2 \rangle $ is measurable for
any vector $\eta_1, \eta_2$ in ${\tilde{\mathcal{H}}}$) for every
decomposable operator $x\in\mathcal{B} ( \mathcal{H} ) .$
\end{remark}

\begin{lemma}
\label{k5}Let $\mathcal{M}$ be a properly infinite von Neumann algebra
acting on a separable Hilbert space $\mathcal{H}$ with a faithful, normal,
semifinite tracial weight $\tau$ and let  $\mathcal{Z}$ be the center of $\mathcal{M}%
.$ Suppose $\mathcal{M=}\int_{X}\oplus\mathcal{M}_{s}d\mu$ and $\mathcal{H=}%
\int_{X}\oplus\mathcal{H}_{s}d\mu$ are direct integral decompositions of $%
\mathcal{M}$ and $\mathcal{H}$ over $( X,\mu ) $ relative to $\mathcal{Z}$.
Then there exist a $\mu$-null set $N$ and a family $\{ \xi_{n} \} _{n\in%
\mathbb{N}}$ of vectors in ${\mathcal{H}}$ such that

\begin{enumerate}
\item[(i)] there is a family of faithful, normal, semifinite tracial weights
$\tau_{s}$ on $\mathcal{M}_{s}$ for $s\in X\setminus N$ satisfying,
for every positive $x\in\mathcal{M}$,
\begin{equation}  \label{equ6.6.4}
\tau ( x ) =\int_{X}\tau_{s} ( x ( s ) ) d\mu ;
\end{equation}

\item[(ii)] moreover, for all $s\in X\setminus N$,
\begin{equation*}
\tau_s(x(s))=\sum_{n\in\mathbb{N}} \langle x(s)\xi_n(s),
\xi_n(s)\rangle, \qquad   \ \text{ for all positive }
x(s)\in\mathcal{M}_s.
\end{equation*}
\end{enumerate}
\end{lemma}

\begin{proof}
By Definition 7.5.1 and Theorem 7.1.12 in \cite{Kadison}, there is a family $%
\{ \xi_{n} \} _{n\in\mathbb{N}}$ of vectors in ${\mathcal{H}}$ such that $%
\tau=\sum_{n=1}^{\mathcal{1}}\omega_{\xi_{n}}$ where $\omega_{\xi_{n}} (
\cdot ) = \langle \ \cdot \ \xi_{n},\xi_{n} \rangle $ is a vector state for
each $n$. Take $\tau_{s}=\sum_{n=1}^{\mathcal{1}}\omega _{ \xi_{n} ( s ) }$
for every $s\in X.$ 
So for every positive $x\in\mathcal{M}$,
\begin{align*}
\int_{X}\tau_{s} ( x ( s ) ) d\mu & =\int_{X}
\sum_{n=1}^{\mathcal{1}} \langle x ( s ) \xi_{n} ( s ) , \xi_{n} ( s
) \rangle d\mu   =\sum_{n=1}^{\mathcal{1}}\int_{X} \langle x ( s )
\xi_{n} ( s ) , \xi_{n}
( s ) \rangle d\mu \\
& =\sum_{n=1}^{\mathcal{1}} \langle x\xi_{n},\xi_{n} \rangle =\tau ( x ) .
\end{align*}
Thus proof of Lemma 14.1.19 in \cite{Kadison} shows that $\tau_{s}$ is
faithful, normal, semifinite tracial weight of $\mathcal{M}_{s}$ almost
everywhere. By the construction of $\tau_s$,
\begin{equation*}
\tau_s(x(s))=\sum_{n\in\mathbb{N}} \langle x(s)\xi_n(s),
\xi_n(s)\rangle, \qquad\ \text{ for all positive } \
x(s)\in\mathcal{M}_s \ \text{almost everywhere.}
\end{equation*}
\end{proof}

The following is the main result of the section.

\begin{theorem}
\label{voiNormal}  Let $\mathcal{M%
}$ be a properly infinite semifinite von Neumann algebra with
separable pre-dual and let $\tau$ be a faithful
normal semifinite tracial weight of $\mathcal M$.  Assume $a$ is a normal operator in $\mathcal{M}.$ Given any $%
\epsilon>0,$ there is a diagonal operator $d$ in $\mathcal{M}$ such
that
\begin{equation*}
\max \{ \|a-d\|, \|a-d\|_2 \}\le \epsilon.
\end{equation*}
\end{theorem}

\begin{proof}
We might assume that $\mathcal{M}$ acts on a separable Hilbert space $%
\mathcal{H}$. Assume that a countable dense subset $\{ \lambda_{i} \} _{i\in%
\mathbb{N}}$ of the ball $ \{ \lambda\in \mathbb{C}: \vert \lambda \vert \leq2 \Vert
a \Vert  \} $ is fixed.

Let $\mathcal{M=}\int_{X}\oplus\mathcal{M}_{s}d\mu$,
$\mathcal{H=}\int _{X}\oplus\mathcal{H}_{s}d\mu$ and
$\tau=\int_{X}\oplus\tau_{s}d\mu$ be direct integral decompositions
of $\mathcal{M}$, $\mathcal{H}$ and $\tau$ over $( X,\mu ) $
relative to the center of $\mathcal{M}$. We might assume that, for
$s\in X$ almost everywhere, $\mathcal{M}_s$ is a properly infinite
factor with a faithful normal semifinite tracial weight $\tau_s$.

As $a$ is a normal operator in $\mathcal{M},$ $a=\int_{X}a(s)d\mu$ where $%
a(s)$ is a normal operator in $\mathcal{M}_s$ almost everywhere.
Thus we can find a Borel $\mu$-null set $N_{1}$ such that
\begin{equation}
\text{$a ( s ) $ is normal \qquad and \qquad $\Vert a ( s ) \Vert \leq \Vert
a \Vert $ \ \ for $s\in X\backslash N_{1}$.}  \label{0}
\end{equation}

Note $X$ is a $\sigma$-compact locally compact space and $\mu$ is the
completion of a Borel measure on $X$. We might assume that $X=\cup_{m\in%
\mathbb{N}}K_{m}$, where $\{K_{m}\}_{m\in\mathbb{N}}$ is a family of
disjoint Borel subsets of $X$ such that $\mu ( K_{m} ) <\mathcal{1}$
for all $m\ge 1$.

From Remark \ref{k3}, there are a separable Hilbert space ${\tilde{\mathcal{H%
}}}$ and a family of unitary elements
\begin{equation*}
\{ U_{s}:\mathcal{H}_{s}\mapsto{\tilde{\mathcal{H}}};s\in X \}
\end{equation*}
such that $s\mapsto U_{s}\zeta(s)$ and $s\mapsto U_{s}x(s)U_{s}^{\ast}$ are
measurable for any $\zeta\in\mathcal{H}$ and any decomposable operator $x\in%
\mathcal{B}(\mathcal{H}).$  Let ${B}$ be the unit ball of
$\mathcal{B} ( {\tilde{\mathcal{H}}} ) $
equipped with the $*$-SOT. Then it is metrizable by setting $%
\rho(x,y)=\sum_{n=1}^{\mathcal{1}}\frac{1}{2^{n}} ( \Vert  ( x-y )
\zeta_{n} \Vert + \Vert ( x^{\ast}-y^{\ast } ) \zeta_{n} \Vert ) $
for any $x,y\in{B}$
where $ \{ \zeta_{n} \}_{n\in\mathbb{N}} $ is an orthonormal basis of ${%
\tilde{\mathcal{H}}}$. The metric space $( {B},\rho) $ is complete
and
separable. Now let ${B}_{1}=\cdots={B}_{i}=\cdots={B}$ and $\mathcal{C=}%
\prod _{i\in\emph{N}}{B}_{i}$ provided with the product topology of the $*$%
-SOT on each ${B}_{i}.$ It implies that $\mathcal{C}$ is metrizable and in
fact it is a complete separable metric space. We denote by
\begin{equation*}
( s,q_{1},\ldots ,q_{i},\ldots )
\end{equation*}
an element in $X\times\mathcal{C}$. Since $x\mapsto x^{\ast}$ and $x\mapsto
x^{2}$ are $*$-SOT continuous from ${B}$ to ${B}$, the maps
\begin{equation}
( s,q_{1},\ldots ) \mapsto q_{i}; \quad ( s,q_{1},\ldots ) \mapsto
q_{i}^{2}; \quad ( s,q_{1},\ldots ) \mapsto q_{i}^{\ast} \text{\quad and
\quad } ( s,q_{1},\ldots ) \mapsto q_{i}q_j  \label{1}
\end{equation}
are Borel measurable from $X\times\mathcal{C}$ to ${B}$.

From Definition 14.1.14 and Lemma 14.1.15 in \cite{Kadison}, there exists a
family $\{a_j^{\prime }\}_{j\in\mathbb{N}}$ of operator in $\mathcal{M}%
^{\prime }$ such that $\{a_j^{\prime }(s)\}_{j\in\mathbb{N}}$ is SOT dense
in the unit ball of $\mathcal{M}_s^{\prime }$ almost everywhere. Combining with Remark %
\ref{k3}, there exists a Borel $\mu$-null subset $N_{2}$ of $X$ such that $%
s\mapsto U_{s}a_j^{\prime }(s)U_{s}^{\ast}$ is Borel measurable when
restricted to $X\backslash N_{2} $ for all $j\in\mathbb{N}$. Hence the maps
\begin{equation}
( s,q_{1},\ldots ) \mapsto U_{s}a_j^{\prime }(s)U_{s}^{\ast}q_{i}\qquad
\text{ and } \qquad ( s,q_{1},\ldots ) \mapsto q_{i}U_{s}a_j^{\prime
}(s)U_{s}^{\ast}  \label{2}
\end{equation}
are Borel measurable when restricted to $( X\backslash N_{2} ) \times%
\mathcal{C}$.

From Lemma \ref{k5}, there exist a Borel $\mu$-null set $N_3$ and a family $%
\{ \xi_{n} \} _{n\in\mathbb{N}}$ of vectors in ${\mathcal{H}}$ such that
\begin{equation}  \label{eq6.6.5.1}
\tau_s(x(s))=\sum_{n\in\mathbb{N}} \langle x(s)\xi_n(s),
\xi_n(s)\rangle, \qquad \forall \ x(s)\in(\mathcal{M}_s)^+, \ s\in
X\setminus N_3.
\end{equation}
Define a mapping $\|\cdot\|_{s,2}:\mathcal{M}_s \rightarrow [0,\infty]$ as
follows. For each $s\in X\setminus N_3$,
\begin{equation}  \label{eq6.6.5}
\Vert x(s) \Vert _{s, 2} =\Big( \tau_s(|x(s)|^2) \Big)^{1/2} = \Big( %
\sum_{n=1}^{\mathcal{1}} \| x(s) \xi_{n} ( s ) \|^2 \Big)^{1/2} , \quad
\forall \ x(s)\in \mathcal{M}_s .
\end{equation}

For each $(q_1,\ldots, q_i,\ldots)\in\mathcal{C}$ and ${j_1},
j_2\in\mathbb{N}$,
\begin{align}
\sum_{n=1}^{j_2} &\| \Big(a ( s ) - \sum_{i=1}^{j_1}\lambda_iU_s^*q_iU_s\Big
) \xi_{n} ( s ) \|^2  \notag \\
& = \sum_{n=1} ^{j_2} \Big( \langle a(s)\xi_n(s), a(s)\xi_n(s) \rangle -
\langle U_sa(s)U_s^* U_s\xi_n(s), \sum_{i=1}^{j_1} \lambda_iq_iU_s\xi_n(s)
\rangle  \notag \\
& \qquad - \langle \sum_{i=1}^{j_1} \lambda_i q_iU_s\xi_n(s) , U_sa(s)U_s^*
U_s\xi_n(s) \rangle + \langle \sum_{i=1}^{j_1} \lambda_i q_iU_s\xi_n(s),
\sum_{i=1}^{j_1} \lambda_i q_iU_s\xi_n(s) \rangle \Big).  \label{equa6.11.5}
\end{align}%
Combining Remark \ref{k3} and (\ref{equa6.11.5}), we get, for all $j_1,j_2\in%
\mathbb{N}$,  the map
\begin{equation}
( s,q_{1},\ldots ) \mapsto \sum_{n=1}^{j_2} \|\Big (a ( s ) -
\sum_{i=1}^{j_1}\lambda_iU_s^*q_iU_s\Big ) \xi_{n} ( s ) \|^2  \label{3.1}
\end{equation}
is Borel measurable on $( X\backslash N_{4} ) \times\mathcal{C}$, where $N_4$
is a Borel $\mu$-null subset of $X$.

Fix a family $\{\eta_n\}_{n\in\mathbb{N}}$ of vectors in $\mathcal{H}$ such
that $\{\eta_n(s)\}_{n\in\mathbb{N}}$ is dense in $\mathcal{H}_s$ almost
everywhere. For each $(q_1,\ldots, q_i,\ldots)\in\mathcal{C}$ and ${j_1}, n
\in\mathbb{N}$,
\begin{align}
\| \Big(a ( s )& - \sum_{i=1}^{j_1} \lambda_iU_s^*q_iU_s\Big )\eta_n(s)\|^2
\notag \\
&= \langle a(s)\eta_n(s), a(s)\eta_n(s) \rangle - \langle U_sa(s)U_s^*
U_s\eta_n(s), \sum_{i=1}^{j_1} \lambda_iq_iU_s\eta_n(s) \rangle  \notag \\
& \quad - \langle \sum_{i=1}^{j_1}\lambda_i q_iU_s\eta_n(s) , U_sa(s)U_s^*
U_s\eta_n(s) \rangle + \langle \sum_{i=1}^{j_1} \lambda_i q_iU_s\eta_n(s),
\sum_{i=1}^{j_1} \lambda_i q_iU_s\eta_n(s) \rangle .  \label{equa6.11.6}
\end{align}
Combining Remark \ref{k3} and (\ref{equa6.11.6}), we get, for all $j_1, n \in%
\mathbb{N}$,  the map
\begin{equation}
( s,q_{1},\ldots ) \mapsto \| (a ( s ) - \sum_{i=1}^{j_1}
\lambda_iU_s^*q_iU_s )\eta_n(s)\|  \label{3.2}
\end{equation}
is Borel measurable on $( X\backslash N_{5} ) \times\mathcal{C}$, where $N_5$
is a Borel $\mu$-null subset of $X$.

Finally, let $N_6$ be a Borel $\mu$-null subset of $X$ such that, for all $%
n\in\mathbb{N}$
\begin{equation}
( s,q_{1},\ldots )\mapsto \|\eta_n(s)\|  \label{equation6.16.1}
\end{equation}
is Borel measurable on $( X\backslash N_{6} ) \times\mathcal{C}$.

Let
\begin{equation*}
N_{0}=N_{1}\cup N_{2}\cup N_{3}\cup N_4\cup N_5\cup N_6.
\end{equation*}
Consider \emph{a subset $A$ of $(X\backslash
N_{0})\times\mathcal{C}$, which consists of all these elements
\begin{equation*}
( s,q_{1},q_{2},\ldots ) \in (X\backslash N_{0})\times\mathcal{C}
\end{equation*}
satisfying }

\begin{enumerate}
\item[(a)] \emph{for every $i\ne j\in\mathbb{N},$ $\displaystyle %
q_{i}=q_{i}^{\ast}=q_{i}^{2}$ and $q_{i}q_{j}=0$; }

\item[(b)] \emph{for every $i, j\in\mathbb{N},$ $\displaystyle %
U_{s}a_j^{\prime} ( s ) U_{s}^{\ast}q_{i}=q_{i}U_{s}a_j^{\prime } ( s )
U_{s}^{\ast}$;
}

\item[(c)] \emph{for each $j_2, m\in\mathbb{N}$ and $s\in K_{m}\cap
(X\setminus N_0)$,
\begin{equation*}
\limsup_{j_1\rightarrow \infty}\Big(\sum_{n=1}^{j_2} \|\Big (a ( s ) -
\sum_{i=1}^{j_1}\lambda_iU_s^*q_iU_s\Big )    \xi_{n} ( s ) \| ^2\Big)^{1/2}
\le\frac{\epsilon}{ 2^{m} ( \mu ( K_{m} )+1 ) ^{1/2}};
\end{equation*}
}

\item[(d)] \emph{for each $n\in \mathbb{N}$,
\begin{equation*}
\limsup_{j_1\rightarrow \infty}\| \Big(a ( s ) - \sum_{i=1}^{j_1}
\lambda_iU_s^*q_iU_s \Big)\eta_n(s)\| \le \epsilon \|\eta_n(s)\|.
\end{equation*}
}
\end{enumerate}

From (\ref{1}), (\ref{2}), (\ref{3.1}), (\ref{3.2}), and (\ref%
{equation6.16.1}), it induces that $A$ is a Borel set. Then by Theorem
14.3.5 in \cite{Kadison}, the set $A$ is analytic.

\vspace{0.2cm}

\textit{Claim \ref{voiNormal}.1. Let }$\pi$\textit{\ be the projection from }%
$X\times\mathcal{C}$\textit{\ onto }$X.$\textit{\ Then }$\pi ( A )
=X\backslash N_{0}.$

\vspace{0.2cm}

\noindent\emph{Proof of the Claim: } Recall that $\{ \lambda_{i} \} _{i\in\mathbb{%
N}}$ is dense in the ball $ \{ \lambda\in\mathbb{C}: \vert \lambda \vert \leq2 \Vert
a \Vert  \} .$ For a given $s\in X\setminus N_0$, there exists an $m\in%
\mathbb{N}$ such that $s\in K_m$.  For such $m\in\mathbb{N}$, by (\ref{0})
and Corollary \ref{lemma6.1.5}, there is a family $\{\tilde q_i\}_{i\in%
\mathbb{N}}$ of orthogonal projections in $\mathcal{M}_s$ such that
\begin{equation}  \label{equ6.7}
\max\{\| a ( s ) -\sum_{i\in\mathbb{N}}\lambda_i\tilde q_i\|, \| a ( s )
-\sum_{i\in\mathbb{N}}\lambda_i\tilde q_i\|_{s,2}\} \le \frac 1 2 \cdot
\frac{\epsilon}{ 2^{m} ( \mu ( K_{m} )+1 ) ^{1/2}},
\end{equation}
where $\|\cdot\|_{s,2}$ is defined in (\ref{eq6.6.5}). Put $q_i=U_s\tilde
q_i U_s^*$ for $i\ge 1$. It is not difficult to see that $(s,q_1,\ldots,
q_i,\ldots)$ is an element satisfying (a) and (b). Note that  $%
\sum_{i=1}^{j_1} \lambda_iU_s^*q_iU_s $ converges to $\sum_{i\in\mathbb{N}}
\lambda_iU_s^*q_iU_s$ ($=\sum_{i\in\mathbb{N}}\lambda_i\tilde q_i$) in SOT
as $j_1$ goes to infinity. Now, combining (\ref{equ6.7}) with the definition
of $\|\cdot\|_{s,2}$ in (\ref{eq6.6.5}), we know that $(s,q_1,\ldots,
q_i,\ldots)$ satisfies (c) and (d). Therefore, $(s,q_1,\ldots, q_i,\ldots)$
is an element in $A$. So $\pi ( A ) =X\backslash N_{0}.$ This completes the
proof the claim.

\vspace{0.2cm}

\noindent\emph{Continue the proof of the theorem: } From Claim \ref%
{voiNormal}.1 and Theorem 14.3.6 in \cite{Kadison}, we can find a measurable
mapping
\begin{equation*}
s\mapsto ( q_{1}^{ ( s ) },\ldots,q_{i}^{ ( s ) },\ldots )
\end{equation*}
from $X\backslash N_{0}$ to $\mathcal{C}$ such that, for $s\in X\backslash
N_{0}$ almost everywhere,
\begin{equation*}
( s,q_{1}^{ ( s ) },\ldots,q_{i}^{ ( s ) },\ldots )
\end{equation*}
satisfies conditions (a), (b), (c) and (d).

By defining $q_{i}^{ ( s ) }=0$ for every $i\in\mathbb{N}$ and $s\in N_{0},$
we obtain \emph{a measurable mapping
\begin{equation*}
s\mapsto ( q_{1}^{ ( s ) },\ldots,q_{i}^{ ( s ) },\ldots )
\end{equation*}
from $X$ to $\mathcal{C}$ such that, for $s\in X$ almost everywhere,
\begin{equation*}
( s, q_{1}^{s},\ldots,q_{i}^{ ( s ) },\ldots )
\end{equation*}
satisfies conditions (a), (b), (c) and (d).} For all vectors
$\zeta_1, \zeta_2\in\mathcal{H}$ and $i\in\mathbb{N}$, we have
\begin{equation*}
\langle U_{s}^{\ast}q_{i}^{ ( s ) }U_{s}\zeta_1 ( s ) ,\zeta_2(s) \rangle =
\langle q_{i}^{ ( s ) }U_{s}\zeta_1(s),U_{s}\zeta_2(s) \rangle
\end{equation*}
and thus the map $s\mapsto \langle
U_{s}^{\ast}q_{i}U_{s}\zeta_1(s),\zeta_2(s) \rangle $ is measurable. Note
\begin{equation*}
\vert \langle U_{s}^{\ast}q_{i}^{ ( s ) }U_{s}\zeta_1(s),\zeta_2(s) \rangle
\vert \leq \Vert \zeta_1(s) \Vert \Vert \zeta_2(s) \Vert .
\end{equation*}
So the map $s\mapsto \langle U_{s}^{\ast}q_{i}^{ ( s )
}U_{s}\zeta_1(s),\zeta_2(s) \rangle $ is integrable. Then by Definition
14.1.1 in \cite{Kadison},
\begin{equation}
U_{s}^{\ast}q_{i}^{ ( s ) }U_{s}\zeta_1(s)= ( p_{i}\zeta_1 ) ( s )
\label{equ6.9}
\end{equation}
almost everywhere for some $p_i\zeta_1\in \mathcal{H}$. Thus (b) and (\ref%
{equ6.9}) imply that $p_{i} ( s ) =U_{s}^{\ast}q_{i}^{ ( s ) }U_{s}\in%
\mathcal{M}_{s}$ almost everywhere. It follows that $p_{i}\in\mathcal{M}$.
From (a), we conclude that $\{ p_{i} \} _{i\in\mathbb{N}}$ is a family of
orthogonal projections in $\mathcal{M} . $

Note that $\sum_{i\in\mathbb{N}}\lambda_ip_i(s)$ converges in SOT and $%
\{\eta_n(s)\}_{n\in\mathbb{N}}$ is dense in $\mathcal{H}_s$ almost
everywhere. From (d), it follows that
\begin{equation*}
\|a-\sum_{i\in\mathbb{N}} \lambda_i p_i\|= ess\sup_{s\in X} \|a(s)-\sum_{i\in%
\mathbb{N}}\lambda_ip_i(s)\| \le \epsilon,
\end{equation*}
From (c) and the fact that $\sum_{i\in\mathbb{N}}\lambda_ip_i(s)$ converges
in SOT almost everywhere, it follows that, for all $m, j_2\in\mathbb{N}$ and
$s\in K_m\cap (X\setminus N_0)$,
\begin{equation*}
\Big(\sum_{n=1}^{j_2} \|\Big (a ( s ) - \sum_{i=1}^{\infty}\lambda_ip_i(s)%
\Big )    \xi_{n} ( s ) \| ^2\Big)^{1/2} \le\frac{\epsilon}{ 2^{m} ( \mu (
K_{m} )+1 ) ^{1/2}}.
\end{equation*}
By (\ref{eq6.6.5}), for all $m \in\mathbb{N}$ and $s\in K_m\cap (X\setminus
N_0)$,
\begin{equation}
\Vert a ( s ) -\sum_{i=1}^{\infty}\lambda_{i}p_{i} ( s ) \Vert _{s, 2}^{2}\le \frac{%
\epsilon^2 }{ 2^{2m } ( \mu ( K_{m} )+1 ) }.  \label{equation6.20}
\end{equation}
Therefore,
\begin{align}
\Vert a -\sum_{i=1}^{\infty}\lambda_{i}p_{i} \Vert _{2}&= \big(
\int_{X} \Vert a ( s ) -\sum_{i=1}^{\infty}\lambda _{i}p_{i} ( s )
\Vert _{s,2}^{2}\ d\mu \big) ^{1/2}  \tag{by
(\ref{eq6.6.5.1}) and (\ref{eq6.6.5})} \\
& = \Big( \sum_{m\in\mathbb{N}}\int_{K_{m}} \Vert a ( s )
-\sum_{i=1}^{\infty}\lambda_{i}p_{i} ( s ) \Vert _{s, 2}^{2} \ d\mu\Big ) ^{1/2}  \notag \\
& \le\Big ( \sum_{m\in\mathbb{N}} \frac{\epsilon^2\cdot \mu(K_m)}{
2^{2m } (
\mu ( K_{m} )+1 ) } \Big ) ^{1/2}  \tag{by (\ref{equation6.20})} \\
& \leq \epsilon.  \notag
\end{align}
Let $d=\sum_{i=1}^{\infty}\lambda_{i}p_{i}$. Then $d$ is a diagonal
operator in $\mathcal{M} $ such that
\begin{equation*}
\max\{\|a-d\|, \|a-d\|_2\}\le \epsilon.
\end{equation*}
This finishes the proof of the theorem.
\end{proof}

The ``properly infinite'' condition on $\mathcal{M}$ in the preceding
theorem is unnecessary.

\begin{theorem}
\label{voiNormal2}  Let $\mathcal{M%
}$ be a   semifinite von Neumann algebra with separable pre-dual and
let $\tau$ be a faithful normal semifinite tracial weight of
$\mathcal M$.  Assume $a$ is a normal operator in $\mathcal{M}.$
Given an $\epsilon>0,$ there is a diagonal operator $d$ in
$\mathcal{M}$ such that
\begin{equation*}
\max \{ \|a-d\|, \|a-d\|_2 \}\le \epsilon.
\end{equation*}
\end{theorem}

\begin{proof}
Let $\mathcal{Z}$ be the center of $\mathcal{M}$. By the type decomposition
theorem, there exists a projection $p$ in $\mathcal{Z}$ such that $p\mathcal{%
M}$ is properly infinite (or $0$) and $(I-p)\mathcal{M}$ is finite (or $0$).
From Theorem \ref{voiNormal}, we can find a diagonal operator $d_0$ in $p%
\mathcal{M}$ such that
\begin{equation*}
\max \{ \|pa-d_0\|, \|pa-d_0\|_2 \}\le \frac \epsilon 2.
\end{equation*}
On the other hand, by Proposition I.6.10 in \cite{Dixmier} and the fact that
$\mathcal{M}$ has a separable pre-dual, there exists a family $%
\{p_n\}_{n=1}^\infty$ of orthogonal projections in $\mathcal{Z}$ such that $%
I-p=\sum_{n=1}^\infty p_n$ and $\tau(p_n)<\infty$ for all $n\ge 1$. For each
$n\ge 1$, as $\tau(p_n)<\infty$, we can find a diagonal operator $d_n$ in $p_n%
\mathcal{M}$ such that
\begin{equation*}
\max \{ \|p_na-d_n\|, \|p_na-d_n\|_2 \}\le \frac \epsilon {2^{n+1}}.
\end{equation*}
Let $d=\sum_{n=0}^\infty d_n$. Then $d$ is a diagonal operator $d$ in $%
\mathcal{M}$ satisfying
\begin{equation*}
\max \{ \|a-d\|, \|a-d\|_2 \}\le \epsilon.
\end{equation*}
\end{proof}

\vspace{1cm}


\begin{thebibliography}{99}
\bibitem{Ake} Charles   Akemann and Gert   Pedersen. Ideal perturbations of
elements in $C^{\ast }$-algebras. \emph{Math. Scand.} \textbf{41} (1977),
no. 1, 117--139.

\bibitem{Ave} William Arveson.  Notes on extensions of $C^*$-algebras. \emph{%
Duke Math. J.} \textbf{44} (1977), no. 2, 329--355.

\bibitem{BV} Hari Bercovici  and  Dan Voiculescu.   The analogue of Kuroda's theorem for $n$-tuples. {\em The Gohberg anniversary collection,} Vol. II (Calgary, AB, 1988), 57--60, {\em Oper. Theory Adv. Appl.}, 41, Birkh\"{a}user, Basel, 1989.


\bibitem{Berg}   David Berg. An extension of the Weyl-von Neumann theorem
to normal operators. \emph{Trans. Amer. Math. Soc.} \textbf{160} (1971),
365--371.

\bibitem {CP}  Richard Carey and  Joel Pincus.  Unitary equivalence module the trace class for self-adjoint
operators. {\em Amer. J. Math.}  \textbf{98} (1976), 481--514.

\bibitem{Niu} Alin Ciuperca, Thierry Giordano, Ping Wong Ng  and  Zhuang Niu. Amenability and
uniqueness. \emph{Adv. Math.} \textbf{240} (2013), 325--345.

\bibitem{Davidson} Kenneth   Davidson. \emph{$C^{\ast }$-algebras by
example.} Fields Institute Monographs, 6. American Mathematical Society,
Providence, RI, 1996.

\bibitem{DavidsonNormal} Kenneth   Davidson. Normal operators are diagonal
plus Hilbert-Schmidt. \emph{J. Operator Theory} \textbf{20} (1988), no. 2,
241--249.


\bibitem {DavidsonNormal2} Kenneth   Davidson,  Ideal perturbations of nests. \emph{ J. Operator Theory} \textbf{26} (1991), no. 2, 241--253.


\bibitem{Dixmier} Jacques Dixmier. \emph{von Neumann algebras.} With a
preface by E. C. Lance. Translated from the second French edition by F.
Jellett. North-Holland Mathematical Library, 27. North-Holland Publishing
Co., Amsterdam-New York, 1981.

\bibitem{Fang} Junsheng Fang, Don Hadwin, Eric Nordgren and  Junhao Shen.
Tracial gauge norms on finite von Neumann algebras satisfying the weak
Dixmier property. \emph{J. Funct. Anal.} \textbf{255} (2008), no. 1,
142--183.

\bibitem{Halmos} Paul Halmos. Ten problems in Hilbert space. \emph{Bull.
Amer. Math. Soc.} \textbf{76} (1970), 887--933.

\bibitem{Kadison} Richard  Kadison  and  John   Ringrose. \emph{Fundamentals
of the theory of operator algebras. Vol. II. Advanced theory.} Corrected
reprint of the 1986 original. Graduate Studies in Mathematics, 16. American
Mathematical Society, Providence, RI, 1997.

\bibitem{Kaftal} Victor Kaftal. On the theory of compact operators in von
Neumann algebras. II. \emph{Pacific J. Math.} \textbf{79} (1978), no. 1,
129--137.

\bibitem{Kato} Tosio Kato. Perturbation of continuous spectra by trace class
operators. \emph{Proc. Japan Acad.} \textbf{33} (1957), 260--264.

\bibitem{Kuroda} Shige Toshi Kuroda. On a theorem of Weyl-von Neumann. \emph{%
Proc. Japan Acad.} \textbf{34} (1958), 11--15.

\bibitem{Lin} Huaxin Lin. \emph{An introduction to the classification of
amenable $C^*$-algebras.} World Scientific Publishing Co., Inc., River Edge,
NJ, 2001.

\bibitem{MurrayVonNeuman1} Francis Murray  and  John von Neumann. {On Rings of
Operators.} \emph{Ann. of Math.} (2) \textbf{37} (1) (1936), 116--229.

\bibitem{MurrayVonNeuman2} Francis Murray  and John von Neumann. {On Rings of
Operators. II.} \emph{Trans. Amer. Math. Soc.} \textbf{41} (2) (1937),
208--248.

\bibitem{MurrayVonNeuman4} Francis Murray  and  John von Neumann. {On Rings of
Operators. IV.} \emph{Ann. of Math.} (2) \textbf{44} (1943), 716--808.

\bibitem{Nelson} Edward Nelson. Notes on non-commutative integration. \emph{%
J. Functional Analysis} \textbf{15} (1974), 103--116.

\bibitem{Von2}  Jon von Neumann. Charakterisierung des Spektrums eines Integraloperators. {\em Actualits
Sci. Indust.} \textbf{229}, Hermann, Paris, 1935.


\bibitem{MurrayVonNeuman3} John von Neumann. {On Rings of Operators. III.}
\emph{Ann. of Math.} (2) \textbf{41} (1940), 94--161.


\bibitem{MurrayVonNeuman5} John von Neumann. {On rings of operators. Reduction
theory.} \emph{Ann. of Math.} \textbf{50} (2) (1949), 401--485.


\bibitem{Pagter} Ben de Pagte. Non-commutative Banach function spaces.
Positivity, 197--227, Trends Math., Birkh\"{a}user, Basel, 2007.

\bibitem{Pisier} Gilles Pisier  and  Quanhua Xu.  Non-commutative $L^{p}$-spaces, {\em Handbook of the
geometry of Banach spaces,} North-Holland, Amsterdam,  2 (2003),
1459--1517.


\bibitem{Rosenblum} Marvin Rosenblum. Perturbation of the continuous
spectrum and unitary equivalence. \emph{Pacific J. Math.} \textbf{7} (1957),
997--1010.

\bibitem{Siknia} William   Sikonia. Essential, singular, and absolutely
continuous spectra. Thesis (Ph.D.)--University of Colorado at Boulder. 1970.

\bibitem{Takesaki} Masamichi Takesaki. \emph{Theory of operator algebras. I.
Reprint of the first (1979) edition.} Encyclopaedia of Mathematical
Sciences, 124. Operator Algebras and Non-commutative Geometry, 5.
Springer-Verlag, Berlin, 2002.

\bibitem{Voi2} Dan Voiculescu. A non-commutative Weyl-von Neumann theorem.
\emph{Rev. Roumaine Math. Pures Appl.} \textbf{21} (1976), no. 1, 97--113.

\bibitem{Voi} Dan Voiculescu. Some results on norm-ideal perturbations of
Hilbert space operators. \emph{J. Operator Theory} \textbf{2} (1979), no. 1,
3--37.

\bibitem{Voi3} Dan Voiculescu. A note on quasitriangularity and trace-class self-commutators. \emph{Acta Sci. Math. (Szeged)} \textbf{42} (1980), 195--199.

    \bibitem{Voi3.1} Dan Voiculescu.  Some results on norm-ideal perturbations of Hilbert space operators. II. {\em J. Operator Theory} \textbf{5} (1981), 77--100.

\bibitem{Voi4} Dan Voiculescu. Remarks on Hilbert--Schmidt perturbations of almost normal operators. {\em Topics in Modern Operator Theory, } Birkh\"{a}user,  1981, 311--318.

\bibitem{Voi5} Dan  Voiculescu. Hilbert space operators modulo normed ideals. {\em Proceedings of the International Congress of Mathematicians, } Vol. 1, 2 (Warsaw, 1983), 1041--1047, PWN, Warsaw, 1984.

 \bibitem{Voi6} Dan Voiculescu.  On the existence of quasicentral approximate units relative to normed ideals. I. {\em J. Funct. Anal.} \textbf{91} (1990), no. 1, 1--36.

\bibitem {Voi7} Dan Voiculescu. Perturbations of operators, connections with singular integrals, hyperbolicity and entropy. {\em Harmonic Analysis and Discrete Potential Theory (ed. M. A. Picardello)} Plenum Press, 1992, 181--191.

\bibitem {Voi8} Dan Voiculescu. Almost normal operators mod Hilbert-Schmidt and the K-theory of the Banach algebras $E\Lambda(\Omega)$. {\em J. Noncommut. Geom.} \textbf{8} (2014), no. 4, 1123--1145.

\bibitem {Voi9} Dan Voiculescu. Some $C^*$-algebras which are coronas of non-$C^*$-Banach algebras. {\em J. Geom. Phys.} \textbf{105} (2016), 123--129.

\bibitem{Weyl} Hermann Weyl. {\"{U}ber beschr\"{a}nkte quadratische formen, deren
differenz vollstetig ist.} Rend. Circ. Mat. Palermo \textbf{27} (1) (1909),
373--392.

\bibitem {Xia1} Jingbo Xia. Diagonalization modulo norm ideals with Lipschitz estimates. {\em J. Funct. Anal.} \textbf{145} (1997), 491--526.

\bibitem {Xia2} Jingbo Xia.  Diagonalization and unitary equivalence modulo Schatten p-classes. {\em J. Funct. Anal.} \textbf{175} (2000), 279--307.


 \bibitem {Xia3} Jingbo Xia. Singular integral operators and norm ideals satisfying a quantitative variant of Kuroda's condition. {\em  J. Funct. Anal.} \textbf{228} (2005), 369--393.


\bibitem {Xia4} Jingbo Xia.   Diagonalization modulo a class of Orlicz ideals, {\em J. Funct. Anal.} \textbf{239} (2006), 268--296.

\bibitem {Xia5} Jingbo Xia.   A condition for diagonalization modulo arbitrary norm ideals. {\em J. Funct. Anal.} \textbf{255} (2008), no. 5, 1039--1056.

\bibitem{Zaido} L\'{a}szl\'{o} Zsid\'{o}. The Weyl-von
Neumann theorem in semifinite factors. \emph{J. Functional Analysis} \textbf{%
18} (1975), 60--72.
\end{thebibliography}
\end{document}